\documentclass[11pt,reqno]{amsart}
\usepackage{graphicx,verbatim,lineno,titletoc}
\usepackage{subcaption}
\usepackage{amssymb,mathrsfs}
\usepackage{amsmath}
\usepackage{adjustbox}
\usepackage{calc}
\usepackage{ytableau}
\usepackage{array}
\usepackage{tikz}
\usetikzlibrary{shapes.multipart,patterns,arrows}
\usepackage[font=footnotesize]{caption}
\usepackage[T1]{fontenc} 
\usepackage{fourier} 
\usepackage[english]{babel} 
\usepackage{appendix}
\usepackage[all]{xy}
\usepackage{enumerate}
\usepackage[english]{babel}
\usepackage{multirow}
\usepackage{float}
\usepackage{enumitem}
\usepackage{accents}
\usepackage[numbers, sort]{natbib}
\input{insbox}
\usepackage{hyperref}
\hypersetup{colorlinks}
\usepackage{algorithm}
\usepackage[noend]{algpseudocode}
\usepackage{mathtools}
\mathtoolsset{showonlyrefs}
\usepackage{wrapfig}
\usepackage{pifont}

\newcommand{\xmark}{\text{\ding{55}}}

\usepackage[top=1.1in, left=1.2in, right=1.2in, bottom=1.1in]{geometry}
\hypersetup{colorlinks,linkcolor={red},citecolor={olive},urlcolor={red}}

\newcolumntype{M}[1]{>{\centering\arraybackslash}m{#1}}
\newcolumntype{N}{@{}m{0pt}@{}}

\newtheorem{theorem}{Theorem}[section]
\newtheorem{prop}[theorem]{Proposition}
\newtheorem{definition}[theorem]{Definition}
\newtheorem{lemma}[theorem]{Lemma}
\newtheorem{corollary}[theorem]{Corollary}

\newtheorem{remark}[theorem]{Remark}
\newtheorem{example}[theorem]{Example}

\def\liminf{\mathop{\rm lim\,inf}\limits}

\def\N{\mathbb{N}}

\def\R{\mathbb{R}}

\def\m{\mathtt{m}}

\def\eps{\varepsilon}

\def\param{\boldsymbol{\theta}}
\def\Param{\boldsymbol{\Theta}}

\definecolor{hancolor}{rgb}{0.1, 0.5, 0.7}
\definecolor{YLcolor}{rgb}{0.8, 0.1, 0.1}

\DeclareMathOperator{\diag}{diag}

\DeclareMathOperator*{\argmin}{arg\,min}

\DeclareMathOperator*{\argmax}{arg\,max}

\DeclareMathOperator{\rtr}{\textup{Rtr}}
\DeclareMathOperator{\rinj}{r_{\textup{inj}}}

\DeclareMathOperator{\rrtr}{r_{\textup{Rtr}}}

\DeclareMathOperator{\rexp}{r_{\textup{Exp}}}

\def\M{\mathcal{M}}

\newcommand{\Exp}{\operatorname{Exp}}
\newcommand{\grad}{\operatorname{grad}}

\newlength\myindent
\newtheorem{innercustomassumption}{}
\newenvironment{customassumption}[1]
{\renewcommand\theinnercustomassumption{#1}\innercustomassumption}
{\endinnercustomassumption}

\theoremstyle{definition}

\allowdisplaybreaks

\makeatletter
\newcommand{\addresseshere}{%
	\enddoc@text\let\enddoc@text\relax
}
\makeatother

\begin{document}
	\title[Convergence and complexity of block majorization-minimization on Riemannian manifolds]{Convergence and complexity of block majorization-minimization \\ for constrained block-Riemannian optimization
	}

	\author{Yuchen Li}
	\address{Department of Mathematics, University of Wisconsin-Madison, 480 Lincoln Dr., Madison, WI 53706, USA}
	\thanks{}
	\email{yli966@wisc.edu}

	\author{Laura Balzano}
	\address{Department of Electrical Engineering and Computer Science, University of Michigan, Ann Arbor, MI 48109, USA}
	\thanks{}
	\email{girasole@umich.edu}

	\author{Deanna Needell}
	\address{Department of Mathematics\\
		University of California, 
		Los Angeles, CA 90025, USA}
	\thanks{}
	\email{deanna@math.ucla.edu}

	\author{Hanbaek Lyu}
	\address{Department of Mathematics, University of Wisconsin-Madison, 480 Lincoln Dr., Madison, WI 53706, USA}
	\thanks{}
	\email{hlyu@math.wisc.edu}

	\begin{abstract}
		Block majorization-minimization (BMM) is a simple iterative algorithm for nonconvex optimization that sequentially minimizes a majorizing surrogate of the objective function in each block coordinate while the other block coordinates are held fixed. We consider a family of BMM algorithms for minimizing smooth nonconvex objectives, where each parameter block is constrained within a subset of a Riemannian manifold. We establish that this algorithm converges asymptotically to the set of stationary points, and attains an $\epsilon$-stationary point within $\widetilde{O}(\eps^{-2})$ iterations. In particular, the assumptions for our complexity results are completely Euclidean when the underlying manifold is a product of Euclidean or Stiefel manifolds, although our analysis makes explicit use of the Riemannian geometry. Our general analysis applies to a wide range of algorithms with Riemannian constraints: Riemannian MM, block projected gradient descent, optimistic likelihood estimation, geodesically constrained subspace tracking, robust PCA, and Riemannian CP-dictionary-learning.  We experimentally validate that our algorithm converges faster than standard Euclidean algorithms applied to the Riemannian setting. 
	\end{abstract}
	
	\maketitle

 \tableofcontents

\section{Introduction}
	Optimization over Riemannian manifolds has a wide array of applications ranging from the computation of linear algebraic quantities and factorizations and problems with nonlinear differentiable constraints to the analysis of shape space and automated learning \cite{ring2012optimization,jaquier2020bayesian}. These applications arise either because of implicit constraints on the problem to be optimized or because the domain is naturally defined as a manifold. Although not nearly as abundant as in the Euclidean optimization setting, there are various methods employed for such optimization \cite{baker2008implicit,yang2007globally, boumal2019global,edelman1998geometry}.  Most of these approaches follow standard optimization techniques by iteratively computing a descent direction and taking a step in that direction along a geodesic.  This computation is often challenging in practice so other approaches alleviate this burden by utilizing approximations or imposing additional assumptions \cite{absil2008optimization,boumal2023introduction}. Here, we consider a block majorization-minimization approach, based on the well-known majorization-minimization (MM) method \cite{lange2000optimization}. The MM method, as its name suggests, consists of two main steps. The majorization step identifies an upper bounding surrogate for the objective function and the minimization step then minimizes that surrogate.  The approach leads to iterations that decrease the objective function, with the benefit that each minimization of a well-chosen surrogate is computationally more efficient than the original problem. Such an approach has also been applied in the context of unconstrained manifold optimization with applications to robust sparse PCA on the Stiefel manifold \cite{breloy2021majorization}. Here unconstrained manifold optimization means the constraint set is only the manifold itself.  When the underlying manifold is a Hadamard manifold (i.e., complete, simply connected manifold with nonpositive sectional curvature everywhere), then one can consider majorizing the objective function by adding the Riemannian distance squared function. The resulting MM algorithm corresponds to the Riemannian proximal point method \cite{li2009monotone, bento2017iteration}.

	In this paper, we consider the minimization of a sufficiently smooth function $f:\Param=\Theta^{(1)}\times \dots \times \Theta^{(m)}\rightarrow \R$. Each constraint set $\Theta^{(i)}$ is a closed subset of a Riemannian manifold $\mathcal{M}^{(i)}$. More precisely, we seek to solve 
	\begin{align}\label{eq:def_CROPT_block}
		\min_{ \substack{\param=[\theta^{(1)},\dots,\theta^{(m)}] \\ \theta^{(i)} \in \Theta^{(i)} \subseteq \mathcal{M}^{(i)} \,\, \textup{for $i=1,\dots,m$}}  } f(\param).
	\end{align} 
    As the problem \eqref{eq:def_CROPT_block} is typically nonconvex, it is not always reasonable to expect that an algorithm would converge to a globally optimal solution starting from an arbitrary initialization. Instead, we aim to provide global convergence (from arbitrary initialization) to stationary points. In some problem classes, stationary points could be as good as global optimizers practically as well as theoretically (see \cite{mairal2010online, sun2015nonconvex}). Furthermore, obtaining the iteration complexity of such algorithms is of importance both for theoretical and practical purposes. In particular, one aims to bound the worst-case number of iterations to achieve an $\eps$-approximate stationary point (defined appropriately, see Sec. \ref{sec:optimality_measures}). 
    
	In order to obtain a first-order optimal solution to  \eqref{eq:def_CROPT_block}, we consider various Riemannian generalizations of the \textit{Block Majorization-Minimization} (BMM) algorithm in Euclidean space \cite{hong2015unified}. The high-level idea of BMM is that, in order to minimize a multi-block objective, one can minimize a majorizing surrogate of the objective in each block in a cyclic order. The algorithm we study in the present work, which we call \textit{Riemannian Block Majorization-Minimization} (RBMM), can be stated in a high-level as follows (see Algorithm \ref{algorithm:BMM} for the full statement): 
    \begin{align}\label{eq:RBMM}
    \hspace{-0.5cm} \textbf{\textup{RBMM}:} \, 
    \begin{cases}
&\hspace{-0.1cm}\textup{For $i=1,2,\dots, m$:}\\
        &g_{n}^{(i)} \leftarrow \left[ \textup{Majorizing surrogate of $\theta\mapsto f_{n}^{(i)}(\theta):=f\left(\theta_{n}^{(1)},\cdots,\theta_{n}^{(i-1)},\theta,\theta_{n-1}^{(i+1)},\cdots, \theta_{n-1}^{(m)}\right)$} \right] \\
       &  \theta_{n}^{(i)}\in \argmin_{\theta\in \Theta^{(i)}}  g_{n}^{(i)}(\theta).
    \end{cases}
    \end{align}
   In this work, we carefully analyze RBMM \eqref{eq:RBMM} in various settings and obtain first-order optimality guarantees and iteration complexity. Moreover, the connection between our RBMM and other existing algorithms is studied, providing complexity results to some of them for the first time in the literature.

    As a preview, we provide a special case of Corollary  \ref{cor:prox_Stiefel} of our main results concerning iteration complexity of RBMM on Stiefel manifolds (manifolds of orthonormal frames, see Ex. \ref{eg:stiefel}). Note the term $L$-smooth in the following corollary indicates the gradient of the function is $L$-Lipschitz continuous (see Sec. \ref{sec:results_option1} for comparison with geodesic smoothness). The $\widetilde{O}(\cdot)$ notation is the variant of ``big-O" notation that ignores the logarithmic factors.
    \begin{corollary}[Complexity of RBMM on Stiefel manifolds]\label{cor:prox_Stiefel_informal}
    Suppose we are minimizing an $L$-smooth function $f$ on the product of the Stiefel manifolds using RBMM \eqref{eq:RBMM} with $L'$-smooth surrogates with quadratic majorization gap 
    \begin{align}
    g_{n}^{(i)}(\theta)- f_{n}^{(i)}(\theta) \ge c \lVert \theta-\theta_{n-1}^{(i)} \rVert^{2}
    \end{align}
    for all $n\ge 1$ and $i=1,\dots,m$ for some $c>0$. Then the iterates asymptotically converge to the set of stationary points and the algorithm has iteration complexity of $\widetilde{O}(\eps^{-2})$. 
    \end{corollary}
    For instance, Euclidean block-proximal updates or Euclidean prox-linear updates on the product of Stiefel manifolds have iteration complexity $\widetilde{O}(\eps^{-2})$. To the best of our knowledge, these type of iteration complexity results for block Riemannian optimization methods are new to the literature.  Note that the conditions we need to check for applying Corollary \ref{cor:prox_Stiefel_informal} (and its generalization Cor. \ref{cor:prox_Stiefel}) are completely Euclidean. However, our proof of these corollaries  incorporates Riemannian geometry in a substantial manner.

	\subsection{Related work}
 \label{sec:related_work}
	Block majorization-minimization methods in the  Euclidean setting for convex problems are studied in 
 \cite{hong2017iteration} with general surrogates, in \cite{xu2013block} with prox-linear surrogates. On the other hand, the majorization-minimization method on Riemannian manifolds is studied in \cite{li2009monotone} and \cite{bento2017iteration} with proximal surrogates on Hadamard manifolds, in \cite{chen2020proximal} with tangent prox-linear surrogates on Stiefel manifolds. These works under the  Riemannian setting, however, only consider the MM method in the one-block case.

 Recently, a tangential type of block coordinate descent has been discussed in \cite{gutman2023coordinate}. There, the authors established a sublinear convergence rate for an interesting block-wise Riemannian gradient descent. However, the retraction considered there is restricted to the exponential map, which excludes many commonly used retractions in the literature.

 Most recently, block coordinate descent on smooth manifolds has been studied in \cite{peng2023block}. The authors studied three types of block coordinate methods on manifolds, namely block majorization-minimization (BMM), block coordinate descent (BCD), and block Riemannian gradient descent (BRGD), as well as the blending of the latter two. Asymptotic convergence to stationary points of BCD and BMM is guaranteed by Theorem 1 and Theorem 2, respectively. Furthermore, the convergence rate of BRGD, as well as its blending with BCD, is shown in their Theorem 3 and Theorem 4, where the manifold is assumed to be a compact smooth submanifold of $\R^n$. It is important to note, however, that convergence (rate) results are applicable to unconstrained problems with exact computations, i.e. an exact minimizer is needed for each sub-problem. Additionally, iteration complexity of $\widetilde{O}(\eps^{-2})$ is established for BMM on compact manifolds.

 In Table \ref{table:summary}, we provide a summary of the aforementioned related works, along with the details.

	\begin{table}[H]
            \small
		\centering
		\begin{tabular}{cccccccc}
			\multicolumn{1}{c}{Methods}& Manifold & Objective & Constraints & Blocks & Complexity & $\begin{matrix} \text{Inexact} \\ \text{comp.} \end{matrix}$ \\ \hline
			\multicolumn{1}{l}{Euclidean BMM \cite{hong2017iteration}}&\multirow{2}{*}{Euclidean} &\multirow{2}{*}{convex} & \multirow{2}{*}{convex} &\multirow{2}{*}{many} &  \multirow{2}{*}{$\widetilde{O}(\eps^{-1})$} &\multirow{2}{*}{\xmark}\\
			\multicolumn{1}{l}{Euclidean Block PGD \cite{beck2013convergence}} &&&&&&\\
   \multicolumn{1}{l}{Euclidean BMM-DR \cite{lyu2023block}} &Euclidean & non-convex &convex &many &$\widetilde{O}(\eps^{-2})$& \checkmark\\
   \hline
			\multicolumn{1}{l}{Riemannian prox. \cite{li2009monotone} } &Hadamard & $g$-convex & $g$-convex &1 &- &\xmark \\
			\multicolumn{1}{l}{Riemannian prox. \cite{bento2017iteration}} &Hadamard & $g$-convex & $g$-convex &1 &$\widetilde{O}(\eps^{-1})$ &\xmark\\
			\multicolumn{1}{l}{Riemannian Prox-linear}& Riemannian \& & non-convex \& & \multirow{2}{*}{N/A} & \multirow{2}{*}{1} & 
			\multirow{2}{*}{$\widetilde{O}(\eps^{-2})$} & \multirow{2}{*}{\xmark}\\
			\multicolumn{1}{l}{(linear search)\cite{chen2020proximal}} &Compact & smooth$^{\dagger}$ & & & &\\
                \multicolumn{1}{l}{Block Riemannian GD} & \multirow{2}{*}{Riemannian} & \multirow{2}{*}{non-convex} & \multirow{2}{*}{N/A} & \multirow{2}{*}{many} & 
			\multirow{2}{*}{$\widetilde{O}(\eps^{-2})$} & \multirow{2}{*}{\xmark}\\
                \multicolumn{1}{l}{(exponential map)\cite{gutman2023coordinate}} & & & & & &\\
                \multicolumn{1}{l}{BMM on manifolds\cite{peng2023block}} &Compact & non-convex &N/A &many &$\widetilde{O}(\eps^{-2})$ &\xmark\\
			\hline
			\textbf{RBMM} (\textbf{Ours}) with surrogates:& & & & & &\\
			 $g$-smooth (Thm. \ref{thm:BMM_rate}) & Riemannian & non-convex & $g$-convex & many & $\widetilde{O}(\eps^{-4})^{\dagger}$ &\checkmark\\[5pt]
			 Riemannian proximal  & \multirow{2}{*}{Riemannian} & \multirow{2}{*}{non-convex} & \multirow{2}{*}{$g$-convex} & \multirow{2}{*}{many} & \multirow{2}{*}{$\widetilde{O}(\eps^{-2})$} &\multirow{2}{*}{\checkmark}\\
    (Thm. \ref{thm:BMM_prox})& & & & & & \\[5pt]
    \multirow{1}{*}{Euclidean proximal} & Riemannian & \multirow{2}{*}{non-convex} & $g$-convex \& & \multirow{2}{*}{many} & \multirow{2}{*}{$\widetilde{O}(\eps^{-2})$} &\multirow{2}{*}{\checkmark}\\
    (Thm. \ref{thm:BMM_prox})&$\subseteq$ Euclidean & &compact & & &\\[5pt]
   
   \multirow{2}{*}{Smooth (Cor. \ref{cor:prox_Stiefel})} & Euclidean/ &non-convex \&& convex/  & \multirow{2}{*}{many} & \multirow{2}{*}{$\widetilde{O}(\eps^{-2})$} &\multirow{2}{*}{\checkmark}\\
            &Stiefel &smooth$^{\dagger}$ &$g$-convex &\\
   \hline 
		\end{tabular}
        
		\caption{Our main contributions and comparison to existing results. ``$g$-smooth" means being geodesically smooth with respect to the geometry of the underlying manifold, and ``smooth'' means being smooth with respect to the Euclidean geometry. ``$g$-convex'' means geodesic convexity of subsets of manifolds. $\widetilde{O}(\cdot)$ notations means big-$O$ up to logarithmic factors. The objective function marked by ``smooth$^{\dagger}$'' only needs to be smooth in the Euclidean sense; In all other cases, it is required to be $g$-smooth with respect to the underlying manifold. The last column shows whether the method allows the inexact solution to a subproblem, i.e. the robustness under inexact computation. Details of comparison to known results can be found in Section \ref{sec:examples}. The complexity of $\widetilde{O}(\eps^{-4})$ in the first row of RBMM can be improved to $\widetilde{O}(\eps^{-2})$ under some additional assumption, see Theorem \ref{thm:BMM_rate}. 
		}
		\label{table:summary} 
	\end{table}

 \subsection{Our contribution}
In this work, we thoroughly analyze RBMM \eqref{eq:RBMM} and obtain asymptotic convergence to stationary points and iteration complexity. The novelty of this work, compared to the aforementioned related work, lies especially in the following three aspects:
\begin{description}
    \item{(1)} (Constrained optimization) RBMM is applicable to constrained optimization problems on manifolds. Here, constrained optimization on manifolds means we allow the constraint set $\Param$ of the optimization problem to be a closed subset of the manifold, i.e. $\Param \subseteq \M$, which is not necessarily the entire manifold.
    \item{(2)} (Iteration complexity) While asymptotic convergence of certain RBMM variants is established in existing literature, the rate of convergences remains unknown. Our work fills this gap by deriving the iteration complexity of RBMM. See Theorems \ref{thm:BMM_prox} and \ref{thm:BMM_rate}.
    \item{(3)} (Robustness) RBMM is robust in the face of inexact computation of optimization sub-problems. See \ref{assumption:A0_optimal_gap}\textbf{(ii)}.
\end{description}

RBMM, as a general Riemannian block optimization framework, entails many classical algorithms including Euclidean block  MM, proximal updates on Hadamard manifolds, and MM methods on Stiefel manifolds. We apply our results to various stylized applications such as geodesically constrained subspace tracking, optimistic likelihood under Fisher-Rao distance, Riemannian CP-dictionary-learning, and robust PCA and obtain the following results:
\begin{description}
    \item{(4)} Asymptotic convergence and complexity of $\widetilde{O}(\eps^{-2})$ for Euclidean block proximal methods when the manifolds are embedded in Euclidean spaces and the constraint sets on the manifolds are compact and $g$-convex. See Theorem \ref{thm:BMM_prox}.
    \item{(5)} Asymptotic convergence and complexity of $\widetilde{O}(\eps^{-2})$ for block MM with Euclidean $L$-smooth surrogate on the product of Stiefel and Euclidean manifolds with $g$-convex constraints. See Corollary \ref{cor:prox_Stiefel}.
    \item{(6)} A Euclidean-regularized version of MM with linear surrogates on the Stiefel manifold is proposed, with guarantees on asymptotic convergence and complexity of $\widetilde{O}(\eps^{-2})$, which can be applied to geodesically constrained subspace tracking problems \cite{blocker2023dynamic}. See Corollary \ref{cor:regularized_linear_Stiefel} and Corollary \ref{cor:geodesic_subspace}.
\end{description}

Table \ref{table:summary} presents a summary of related work, along with our main theoretical contributions.

	\subsection{Organization}

	The paper is organized as follows. We introduce preliminaries and notations in Section \ref{sec:preliminaries_Riemannian} and \ref{sec:notations}. In Section \ref{sec:Algorithm}, we give a precise statement of the RBMM algorithm. We detail the standing assumptions and state our main results in Section \ref{sec:results}. In Section \ref{sec:examples}, we present some applications and classical algorithms as special cases of our general results. We present some stylized applications of our results in Section \ref{sec:apps}. We prove the convergence results of RBMM, Algorithm \ref{algorithm:BMM}, throughout  Section \ref{sec:conv_opt1}. Figure \ref{fig:diagram} provides a structure diagram of the present paper.

	\begin{figure*}[h]
		\centering
		\includegraphics[width=1 \linewidth]{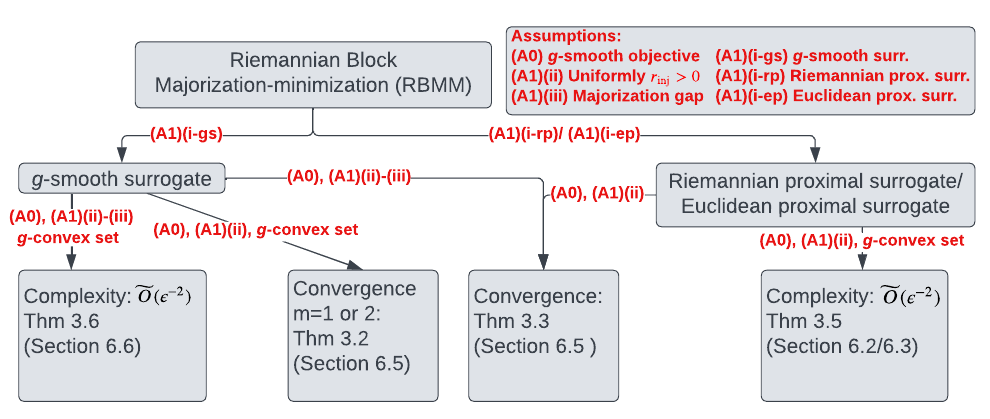}
		\vspace{-0.3cm}
		\caption{ Structure of the present paper. }
		\label{fig:diagram}
	\end{figure*}

	\section{Preliminaries and Algorithm}

    \subsection{Preliminaries on Riemannian geometry}
	\label{sec:preliminaries_Riemannian}

	The notations throughout this paper are consistent with the common literature, see e.g. \cite{absil2008optimization} and \cite{boumal2023introduction}. For background knowledge on Riemannian geometry, we refer the readers to \cite{sakai1996riemannian}, \cite{lee2003introduction}, \cite{do1992riemannian}, and \cite{helgason1979differential}. In this section, we provide a brief introduction to the notations used in our paper, with further details provided in the Appendix \ref{sec:notes}.

	A \textit{Riemannian manifold} $\mathcal{M}$ is a manifold endowed with a Riemannian metric $\left(\eta, \xi\right) \mapsto\left\langle\eta, \xi\right\rangle_x \in \mathbb{R}$, where $\eta$ and $\xi$ are tangent vectors in the tangent space $T_{x}\mathcal{M}$ (also denoted as $T_{x}$ when it is clear from context) of $\mathcal{M}$ at $x$.  This inner product in the tangent space is also denoted by $\langle\cdot, \cdot\rangle$ for convenience when the subscript is clear from the context. The induced norm on the tangent space is denoted by $\|\cdot\|_{x}$ or $\|\cdot\|$. The \textit{Riemannian gradient} of a smooth function $f: \mathcal{M} \rightarrow \R$ at $x$ is defined as the unique tangent vector, $\operatorname{grad} f(x) \in T_{x} \mathcal{M}$, such that $\langle\operatorname{grad} f(x), \xi_{x}\rangle=\mathrm{D} f(x)\left[\xi_{x}\right], \forall \xi_{x} \in T_{x} \mathcal{M}$, where $\mathrm{D} f(x)\left[\xi_{x}\right]$ is the differential of $f$ at point $x$ along the direction $\xi_{x}$. The geodesic distance between $x,y\in \mathcal{M}$ is denoted by $d_{\mathcal{M}}(x,y)$ or $d(x,y)$ when it is clear from the context. 
	
	A \textit{retraction} on a manifold $\mathcal{M}$ is a locally defined smooth mapping $\rtr$ from the tangent bundle $T\mathcal{M}$ to $\mathcal{M}$ with the following properties. 
	\begin{description}
		\item[(i)] For each $x\in \M$, let  $\rrtr(x)>0$ be the `retraction radius' such that the restriction $\rtr_{x} : T_{x}\mathcal{M} \rightarrow \mathcal{M}$ of $\rtr$ to $T_{x}\mathcal{M}$ is well-defined in a ball of radius $\rrtr(x)$ around the origin $\mathbf{0}=\mathbf{0}_{x}$. 
		
		\item[(ii)] $\rtr_{x}(\mathbf{0})=x$;  The differential of $\rtr_{x}$ at $\mathbf{0}$, $D \rtr_{x} (\mathbf{0})$, is the identity map on $T_{x}\mathcal{M}$. 
	\end{description}
	For each $x\in \mathcal{M}$ and $\eta\in T_{x}\mathcal{M}$, the retraction curve $t\mapsto \rtr_{x}(t\eta)$ agrees up to first order with geodesics passing through $x$ with velocity $\eta$ around $x$. Retractions provide a way to lift a function $g:\M\rightarrow \R$ onto the tangent space $T_{x}\M$ via its \textit{pullback} $\hat{g} :=g\circ \rtr_{x}:T_{x}\rightarrow \R$. We use this construction to lift an upper-bounding surrogate defined on the manifold onto the tangent spaces in Algorithm \ref{algorithm:BMM}. Note that for all $\eta\in T_{x}\M$, 
	\begin{align}\label{eq:pullback_derivative}
		\langle \nabla \hat{g}(\mathbf{0}),\, \eta \rangle = D \hat{g}(\mathbf{0})[\eta] = D g(x) [ D \rtr_{x}(\mathbf{0}) [\eta] ] = D g(x) [\eta] = \langle \grad g(x),\, \eta \rangle. 
	\end{align}

	If for all $x\in \M$ and $\eta\in T_{x}\M$, the retraction curve $t\mapsto \rtr_{t \eta}$ coincides with the geodesic curve $t\mapsto \gamma(t)$ with $\gamma(0)=x$ and $\gamma'(0)=\eta$ whenever $ \lVert t \eta \rVert \le \rexp(x)$ for some constant $\rexp(x)>0$,  then the retraction $\rtr$ is called the \textit{exponential map} and denoted as $\Exp$. Note that the exponential map is defined as the solution of a nonlinear ordinary differential equation. While every Riemannian manifold admits the exponential map, its computation is often challenging. Retractions provide computationally efficient alternatives to exponential maps. Some typical choices of retractions are $\rtr_{x}(\eta)=x+\eta$ on Euclidean spaces and $\rtr_{x}(\eta)=\frac{x+\eta}{\lVert x+\eta \rVert}$ on spheres. See \cite[Sec. 4.1]{absil2008optimization} for more examples of retractions. If the exponential map is defined on the entire tangent bundle (i.e., $\rexp(x)=\infty$ for all $x\in \mathcal{M}$), then we say $\mathcal{M}$ is (geodesically) \textit{complete}. Note by definition of exponential map we have $\|\eta\|=d(x,y)$ whenever $d(x,y)\le \rexp(x)$ and $\Exp_{x}(\eta)=y$.

	A Riemannian manifold is locally diffeomorphic to its tangent spaces, so it resembles the Euclidean space within a small metric ball around each point. Accordingly, there are several notions of radius functions $r:\M\rightarrow [0,\infty)$, including the injectivity and the convexity radii. For $x\in\mathcal{M}$, consider the open ball $B(x,r)=\{\eta\in T_x \mathcal{M}:\langle \eta,\eta\rangle<r\}\subseteq T_x \mathcal{M}$; the \textit{injectivity radius} of $\mathcal{M}$ at $x$, denoted as $\rinj(x)$, is the supremum of values of $r$ such that $\Exp_{x}$ defines a diffeomorphism from $B(x,r)$ to its image on $\mathcal{M}$. Thus, we can also define the inverse exponential map from $\mathcal{M}$ to $T_{x}\mathcal{M}$, denoted by $\Exp^{-1}_{x}(\cdot)$, within the injectivity radius. A set $C\subseteq \M$ is called (geodesically) \textit{strongly convex} if for any $x$ and $y$ in $C$, there is a unique minimal geodesic $\gamma$ in $\M$ joining $x$ and $y$, and $\gamma$ is contained in $C$. In particular, 
	compact manifolds have uniformly positive injectivity radius (see \cite[Thm. III.2.3]{chavel2006riemannian}). Furthermore, it is worth noting that Hadamard manifolds (complete and simply connected manifolds with non-positive curvature), which include Euclidean space, hyperbolic space, and manifolds of positive definite matrices, also have uniformly positive injectivity radius (\cite{afsari2011riemannian}, \cite[Theorem 4.1, p.221]{sakai1996riemannian}).
	When the injectivity radius is uniformly positive, there exists a retraction with a uniformly positive retraction radius (e.g., the exponential map).

For a subset $\Param\subseteq \mathcal{M}$ and $x\in \Param$, define the \textit{lifted constraint set} $T_{x}^{*}\mathcal{M}$ as
\begin{align}\label{eq:def_lift_constraints}
	T^{*}_{x} \mathcal{M} &:= \{ u\in T_{x} \mathcal{M} \,|\,  \textup{$\rtr_x(u)=x'$ for some $x'\in \Param$ with $d(x,x')\le r_0 /2$} \},
\end{align}
where $r_0$ is the lower bound of the injectivity radius (see \ref{assumption:A0_optimal_gap} in Section \ref{sec:results}).  When we use exponential map $\Exp$ as the retraction in \eqref{eq:def_lift_constraints}, one can think of the set $T^{*}_{x} \mathcal{M}$ as the `lift' of the constraint set $\Param$ onto the tangent space $T_{x}\M$ in the sense that if $\Param$ is contained in the metric ball of radius $\rinj(x)$ centered at $x$, then $T^{*}_{x} \mathcal{M}$ equals the inverse image $\Exp^{-1}_{x}(\Param)$ of $\Param$ under $\Exp_{x}$. In particular, this formula holds for all subset $\Param$ of $\M$ if $\M$ is a complete Riemannian manifold since $\Exp_{x}$ is defined on the entire $T_{x}\M$. If $\mathcal{M}$ is a Euclidean space and $\Param$ is a convex subset of it, then $T^{*}_{x} \mathcal{M}=\Param$. Lastly, we note that when $\Param$ is strongly convex in $\M$, the set $T_{x}^{*}\M$ above is locally defined near $x$. That is, we can replace the injectivity radius $r_0$ in \eqref{eq:def_lift_constraints} by any constant $\delta\in (0,r_0)$.

	On Riemannian manifolds, parallel transport provides a way to transport a vector along a smooth curve. For each smooth curve $\gamma:[0,1]\rightarrow \mathcal{M}$, denote the \textit{parallel transport} along $\gamma$ from $x=\gamma(0)$ to $y=\gamma(t)$ by $\Gamma^{\gamma}_{x\rightarrow y}$. If $\gamma$ is clear from the context (e.g., the unique distance-minimizing geodesic from $x$ to $y$), then we also write $\Gamma_{x\rightarrow y}^{\gamma}=\Gamma_{x\rightarrow y}$. Intuitively, a tangent vector $\eta\in T_x \mathcal{M}$ at $x$ of $\gamma$ is still a tangent vector $\Gamma(\gamma)\eta \in T_y \mathcal{M}$ of $\gamma$ at $y$. Recall that, one important property of parallel transport on a Riemannian manifold is that it preserves the inner product, i.e. 
	\begin{align}
		\left\langle \Gamma^{\gamma}_{\gamma(0)\rightarrow \gamma(t)} \, (\xi) ,\,  \Gamma^{\gamma}_{\gamma(0)\rightarrow \gamma(t)} \, (\zeta) \right\rangle_{\gamma(t)} = \left\langle  \xi,\, \zeta\right\rangle_{\gamma(0)}  \quad \textup{for all $t\in [0,1]$, $\xi,\zeta\in T_{\gamma(0)}$}.
	\end{align}
	See Figure \ref{fig:retr_para} for an illustration.

	\begin{figure}
		\centering
		\begin{subfigure}[t]{0.4\textwidth}
			\centering
			\includegraphics[width=\textwidth]{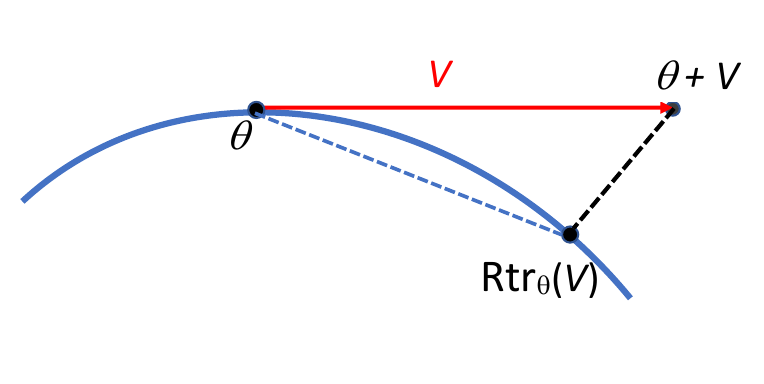}
		\end{subfigure}
		\hfill
		\begin{subfigure}[t]{0.4\textwidth}
			\centering
			\includegraphics[width=\textwidth]{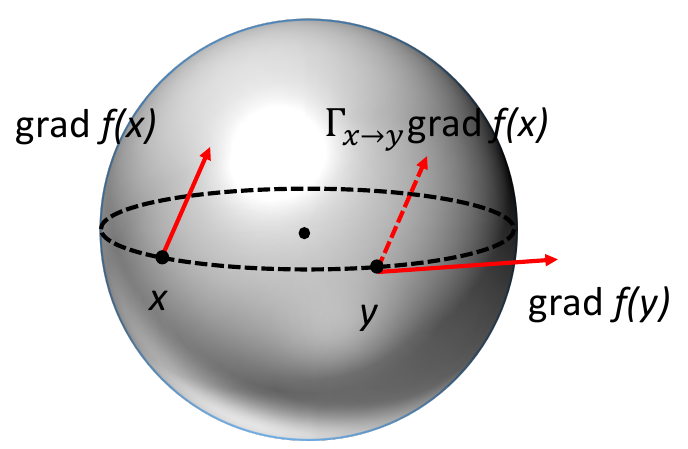}
		\end{subfigure}
		\caption{Illustration of (left) retraction and (right) parallel transport. }
		\label{fig:retr_para}
	\end{figure}

	\subsection{Notations for block Riemannian optimization}
	\label{sec:notations}
	In \eqref{eq:def_CROPT_block}, we are interested in minimizing an objective function $f$ within the product parameter space $\Param=\Theta^{(1)}\times \dots \times \Theta^{(m)}$, where each constraint set $\Theta^{(i)}$ is a subset of a Riemannian manifold $\mathcal{M}^{(i)}$. It will be convenient to introduce the following notations: For $\param=[\theta^{(1)},\dots,\theta^{(m)}]$, 
	\begin{align}
		\grad_{i} f(\param)&:=\textup{Riemmanian gradient of $\theta \mapsto f(\theta^{(1)},\dots,\theta^{(i-1)}, \theta, \theta^{(i+1)},\dots,\theta^{(m)})$},\\ 
		\grad f(\param)&:=[ \grad_{1} f(\param),\dots, \grad_{m} f(\param) ], \\
		 d(\mathbf{x}, \mathbf{y})  &:= \sqrt{\sum_{i=1}^{m} d(x^{(i)}, y^{(i)})^{2}} \quad \textup{for $\mathbf{x}=(x^{(1)},\dots,x^{(m)}), \,  \mathbf{y}=(y^{(1)},\dots,y^{(m)})\in \prod_{i=1}^{m}\M^{(i)}$}.  \label{eq:def_product_dist}
	\end{align}
    Note that if we endow the product $\prod_{i=1}^{m}\M^{(i)}$ of the manifolds a joint Riemannian structure, then we can interpret $\grad f(\param)$ above as the Riemannian gradient at $\param$ with respect to that joint Riemannian structure. However, we do not explicitly introduce or use such a product manifold structure in the manuscript. 
 
	Throughout this paper, we let $(\param_{n})_{n\ge 1}$  denote an output of Algorithm \ref{algorithm:BMM} and write $\param_{n}=[\theta_{n}^{(1)},\dots,\theta_{n}^{(m)}]$ for each $n\ge 1$. For each $n\ge 1$ and $i=1,\dots,m$, denote 
	\begin{align}\label{eq:def_f_n_marginal}
		f_{n}^{(i)}: \theta \mapsto f(\theta_{n}^{(1)},\dots,\theta_{n}^{(i-1)},\theta,\theta_{n-1}^{(i+1)},\dots,\theta_{n-1}^{(m)}),
	\end{align}
	which we will refer to as the $i$th marginal objective function at iteration $n$.

    \subsection{Statement of algorithm}
 \label{sec:Algorithm}
	Below in Algorithm \ref{algorithm:BMM}, we give a precise statement of the RBMM algorithm we stated in high-level at \eqref{eq:RBMM}.  We first define majorizing surrogate functions defined on Riemannian manifolds. 
	\begin{definition}[Majorizing surrogates on Riemannian manifolds]
		\label{def:majorizing_surrogates}
		Fix a function $h:\mathcal{M}\rightarrow \R$ and a point $\param\in \M$, where $\mathcal{M}$ is a Riemannian manifold. A function $g:\M\rightarrow \R$ is a \textit{majorizing surrogate} of $h$ at $\param$ if 
		\begin{align}
			g(x) \ge h(x) \quad \textup{for all $x\in \M$} \quad \textup{and} \quad g(\param)=h(\param). 
		\end{align}
  
	\end{definition}

	As mentioned before, the high-level idea is the following. In order to update the $i$th block of the parameter $\theta^{(i)}_{n}$ at iteration $n$, we use the RMM we described in the introduction.

	\begin{algorithm}[H]
		\small
		\caption{Riemannian Block Majorization-Minimization (RBMM) } 
		\label{algorithm:BMM}
		\begin{algorithmic}[1]
			\State \textbf{Input:} $\param_{0}=(\theta_{0}^{(1)},\cdots,\theta_{0}^{(m)})\in \Theta^{(1)}\times \cdots \times \Theta^{(m)}$ (initial estimate); $N$ (number of iterations); 
			
			\State \quad \textbf{for} $n=1,\dots,N$ \textbf{do}: \State \quad \quad Update estimate $\param_{n}=[\theta_{n}^{(1)},\cdots, \theta_{n}^{(m)}]$ by  
			\State \quad \quad \quad \textbf{For} $i=1,\cdots,m$ \textbf{do}: \quad 
			\State \hspace{1cm} $\displaystyle f_{n}^{(i)}(\cdot):=f\left(\theta_{n}^{(1)},\cdots,\theta_{n}^{(i-1)}, \,  \cdot \,\,  ,\theta_{n-1}^{(i+1)},\cdots, \theta_{n-1}^{(m)}\right): \mathcal{M}^{(i)}\rightarrow \R \qquad \textup{($\triangleright$ marginal objective function)}$ 
			
			\State  $\vspace{-0.5cm}$
			\begin{align}
				&\hspace{-3.5cm} \qquad 
				\begin{cases}\label{eq:RBMM_MmMm} 
					&g_{n}^{(i)} \leftarrow \left[ \textup{Majorizing surrogate of $f_{n}^{(i)}$ at $\theta_{n-1}^{(i)}$} \right] \\
					&\theta_{n}^{(i)}\in \argmin_{\theta\in \Theta^{(i)}}  g_{n}^{(i)}(\theta) 
				\end{cases} 
			\end{align}

			\State \quad \qquad \textbf{end for}
			\State \quad \textbf{end for}
			\State \textbf{output:}  $\param_{N}$ 
		\end{algorithmic}
	\end{algorithm}
	\vspace{-0.5cm}

 \vspace{0.2cm} 
	\noindent In Algorithm \ref{algorithm:BMM}, the majorizing surrogate $g_{n}^{(i)}$ at each iteration $n$ for each block $i$ is chosen so that 
	\begin{description}[itemsep=0.1cm]
		\item{(1)} (Majorization) $g_{n}^{(i)}(x)- f_{n}^{(i)}(x) \ge 0$ for all $x \in \M^{(i)}$;
		\item{(2)} (Sharpness) $g_{n}^{(i)}(\theta_{n-1}^{(i)})=f_{n}^{(i)}(\theta_{n-1}^{(i)})$. 
	\end{description}

	\section{Statement of results}\label{sec:results}

	In this section, we state our main results concerning the convergence and complexity of our RBMM algorithm (Alg. \ref{algorithm:BMM}) for the constrained block Riemannian optimization problem in \eqref{eq:def_CROPT_block}. 
	Figure \ref{fig:diagram} provides a structure diagram of the main results and assumptions.

	\subsection{Optimality and complexity measures} 
	\label{sec:optimality_measures}

	For iterative algorithms, 
	first-order optimality conditions may hardly be satisfied exactly in a finite number of iterations, so it is more important to know how the worst-case number of iterations required to achieve an $\eps$-approximate solution scales with the desired precision $\eps$. 

	More precisely, for the multi-block problem \eqref{eq:def_CROPT_block}, we say $\param_{*}=[\theta_{*}^{(1)},\dots,\theta_{*}^{(m)}]\in \Param$ is an \textit{$\eps$-stationary point} of $f$ over $\Param$ if

 \begin{align}\label{eq:stationary_approximate}
    \sum_{i=1}^{m} \, \left( -  \inf_{u \in T_{\theta_{*}^{(i)}}^{*},\|u\|\le 1 } \left\langle \grad_{i} f(\param_{*}) ,\, \frac{u}{\hat{r}}  \right\rangle  \right)\le \eps,
	\end{align}	
 where $\hat{r}=\min\{r_0,1\}$, i.e. the minimum of $1$ and the lower bound of injectivity radius.
    
    In the Euclidean setting, \eqref{eq:stationary_approximate} reduces to the optimality measure used in \cite{lyu2022convergence} for constrained nonconvex smooth optimization and is also equivalent to \cite[Def. 1]{nesterov2013gradient} for smooth objectives.  
	
	In the unconstrained block-Riemannian setting where  $\Theta^{(i)}=\mathcal{M}^{(i)}$ for $i=1,\dots,m$, the above equation  \eqref{eq:stationary_approximate} becomes 
	\begin{align}\label{eq:stationary_option2}
		\sum_{i=1}^{m} \, \lVert \grad_{i} f(\param_{*}) \rVert \le \eps. 
	\end{align}	
	In the case of single-block $m=1$, the above is the standard definition of $\eps$-stationary points for unconstrained Riemannian optimization problems.

	Next, for each $\eps>0$ we define the \textit{(worst-case) iteration complexity} $N_{\eps}$ of an algorithm computing $(\param_{n})_{n\ge 1}$ for solving \eqref{eq:def_CROPT_block} as 
	\begin{align}\label{eq:Neps}
		N_{\eps}:= \sup_{\param_{0}\in \Param} \, \inf\, \{ n\ge 1 \,|\, \text{$\param_{n}$ is an $\eps$-approximate stationary point of $f$ over $\Param$} \}, 
	\end{align}
	where $(\param_{n})_{n\ge 0}$ is a sequence of estimates produced by the algorithm with an initial estimate $\param_{0}$. Note that $N_{\eps}$ gives the \textit{worst-case} bound on the number of iterations for an algorithm to achieve an $\eps$-approximate solution due to the supremum over the initialization $\param_{0}$ in \eqref{eq:Neps}.

	\subsection{Statement of results}\label{sec:results_option1}

Here we state our main convergence results of solving the minimization problem \eqref{eq:def_CROPT_block} using Algorithm \ref{algorithm:BMM}. 
	First, we introduce a Riemannian counterpart of a smoothness property of a function $F:\mathcal{M}\rightarrow \R$. 
    In the Euclidean setting, The function $F$ is $L$-smooth if its gradient $\nabla F$ is $L$-Lipschitz continuous. In the Riemannian case, Riemannian gradients $\grad F(x)$ and $\grad F(y)$ at two base points $x,y\in \mathcal{M}$ live in different tangent spaces $T_x$ and $T_y$, so they have to be compared using a parallel transport (see Figure \ref{fig:retr_para} for an illustration). We extend this notion of smoothness to the setting where the function $F$ is defined on the product of Riemannian manifolds.

	\begin{definition}[Geodesic smoothness]\label{def: G-L-smooth}
		A function $F: \prod_{i=1}^{m}\M^{(i)} \to \R$ is \textit{geodesically smooth} ($g$-smooth in short) with parameter $L>0$ if $F$ is block-wise continuously differentiable and for each $\mathbf{x}=(x^{(1)},\dots,x^{(m)}),\mathbf{y}=(y^{(1)},\dots,y^{(m)}) \in \prod_{i=1}^{m}\M^{(i)}$ where there exists a minimizing geodesic joining $x^{(i)}$ and $y^{(i)}$ for each $i=1,\dots,m$,
			\begin{equation}
				\left\lVert \grad_{i} F(\mathbf{x}) - \Gamma_{y^{(i)}\rightarrow x^{(i)}}(\grad_{i} F(\mathbf{y})) \right\rVert \le \frac{L}{m} d(\mathbf{x}, \mathbf{y}),
			\end{equation}
		where $\Gamma_{y^{(i)}\rightarrow x^{(i)}}$ is the parallel transport along a distance-minimizing geodesic joining $x^{(i)}$ and $y^{(i)}$ in $\M^{(i)}$, and $d(\mathbf{x}, \mathbf{y})$ is defined in \eqref{eq:def_product_dist}. 
	\end{definition}	
	
	An important consequence of the $g$-smoothness is the following quadratic bound on first-order approximation: 
	\begin{equation}\label{eq:linear_approx_quad_bd}
		\left|F(y) - F(x)- \left\langle \grad F(x), \gamma'(0) \right\rangle_x  \right|\le \frac{L}{2} d^{2}(x, y),
	\end{equation}
	where $\gamma$ is any distance-minimizing geodesic in $\mathcal{M}$ from $x$ to $y$ and $d(x,y)$ is the Riemannian distance between $x$ and $y$. See Appendix \ref{lem:g_smooth_linear_approx} for the proof. Throughout this paper, the term ``$g$-smooth'' denotes geodesic smooth (Def \ref{def: G-L-smooth}), while ``$L$-smooth'' indicates the function $L$-smooth in Euclidean sense, meaning the Euclidean gradient of the function is $L$-Lipschitz continuous.

We start by stating some general assumptions. We allow inexact computation of the solution to the minimization sub-problems in Algorithm \ref{algorithm:BMM}. This is practical since minimizing the surrogates on the manifold (possibly with additional constraints) may not always be exactly solvable. To be precise, for each $n \geq 1$, we define the \textit{optimality gap} $\Delta_n$ by
	\begin{equation}\label{eq:def_optimality_gap}
		\Delta_n=\Delta_n\left(\boldsymbol{\theta}_0\right):=
			\max _{1 \leq i \leq m}\left(g_n^{(i)}(\theta_n^{(i)})-\inf _{\theta \in \Param^{(i)}} g_n^{(i)}(\theta)\right) 
	\end{equation}
	For the convergence analysis to hold, we require that the optimal gaps decay fast enough so that they are summable, stated as Assumption \ref{assumption:A0_optimal_gap}.

\begin{customassumption}{(A0)}
		\label{assumption:A0_optimal_gap}
		For RBMM (Alg. \ref{algorithm:BMM}), we make the following assumptions: 
		\begin{description}
			\item[(i)] (Objective) The objective $f:\Param=\prod_{i=1}^{m} \Theta^{(i)} \rightarrow \R$ is $g$-smooth with some parameter $L_{f}>0$ and 
    its values are uniformly lower bounded by some $f^{*}\in \R$. Furthermore, the sublevel sets $f^{-1}((-\infty, a))=$ $\{\boldsymbol{\theta} \in \boldsymbol{\Theta}: f(\boldsymbol{\theta}) \leq a\}$ are compact for each $a \in \mathbb{R}$.
			
			\item[(ii)] (Inexact computation) The optimality gaps $\Delta_n$ in \eqref{eq:def_optimality_gap} are summable, that is, $\sum_{n=1}^{\infty} \Delta_n<\infty$. Furthermore, let $\theta_{n}^{(i\star)}$ be an exact solution of the minimization step in \eqref{eq:RBMM_MmMm} and let $\theta_{n}^{(i)}$ be the inexact output. Then for all $i=1,\dots,m$, 
			\begin{align}\label{eq:assumption_inexact_sol_convervence}
				d(\theta_{n}^{(i\star)}, \theta_{n}^{(i)}) = o(1).
			\end{align}
		\end{description}
	\end{customassumption}

	\noindent If the surrogate function $g^{(i)}_n$ is geodesically strongly convex (see Definition \ref{def:g_strongly_convex}), then \eqref{eq:assumption_inexact_sol_convervence} is a direct consequence of the summability of the optimality gaps. In particular, this is the case for Riemannian proximal surrogates on Hadamard manifolds as a special case of \ref{assumption:A1_smoothness_surrogates}\textbf{(i-rp)} (see Prop. \ref{prop:optimality_gap_ite}).

	Next, we impose the following conditions on the underlying manifolds, majorizing surrogates, and constraint sets. 
	
	\begin{customassumption}{(A1)}
		\label{assumption:A1_smoothness_surrogates}
		
		For the manifolds and surrogates, at least one of the following holds: 
		\begin{description}[itemsep=0.1cm]
			\item[(i-gs)] ($g$-smooth surrogates) Each surrogate $g_{n}^{(i)}:\mathcal{M}^{(i)}\rightarrow \R$ is $g$-smooth with some parameter $L_{g}\ge 0$ for all $n\ge 1$ and $i=1,\dots,m$. 
			
			\item[(i-rp)] (Riemannian Proximal surrogates)
			Each surrogate $g_{n}^{(i)}$ is a Riemannian proximal surrogate; that is, for each $n\ge 1$ and some constant $\lambda_n \ge L_{f}$ with $\lambda_{n}=O(1)$, 
			\begin{align}\label{eq:def_prox_surrogate}
				g_{n}^{(i)}(\theta)  = 
				f_{n}^{(i)}(\theta)  + \frac{\lambda_n}{2} d^{2}(\theta, \theta_{n-1}^{(i)}).
			\end{align}    

            \item[(i-ep)] (Euclidean Proximal surrogates) 
            Each $\mathcal{M}^{(i)}$ is an embedded submanifold in a Euclidean space and the constraint set $\Theta^{(i)}$ is compact. Each surrogate $g_{n}^{(i)}$ is a Euclidean proximal surrogate, that is, for each $n\ge 1$ and some constant $\lambda_n \ge L_f$ with $\lambda_{n}=O(1)$, 
			\begin{align}\label{eq:def_Euclidean_prox_surrogate}
				g_{n}^{(i)}(\theta)  = 
				f_{n}^{(i)}(\theta)  + \frac{\lambda_n}{2} \lVert \theta-\theta_{n-1}^{(i)} \rVert^{2}.
			\end{align}

		\end{description}
   Furthermore, we require the following for the constraint sets: 
    \begin{description}
        \item[(ii)] For each $i=1,\dots,m$, there exists a uniform lower bound $r_{0}>0$ for  the injectivity radius $\rinj(x)$ over $x\in \Theta^{(i)}$, i.e. $\rinj(x)\ge r_0$ for all $x\in \Theta^{(i)}$.
    \end{description}
	\end{customassumption}

    When the manifolds $\mathcal{M}^{(i)}$ are non-Euclidean, constructing $g$-smooth surrogates may not always be easy. However, for Stiefel manifolds (see Ex. \ref{eg:stiefel}), Euclidean smoothness and $g$-smoothness are essentially equivalent (see Lem. \ref{lem:g_smooth_Stiefel}) so standard smooth surrogates in the Euclidean space (e.g, proximal \eqref{eq:def_Euclidean_prox_surrogate} and prox-linear \eqref{eq:block_prox_linear}) verifies \ref{assumption:A1_smoothness_surrogates}\textbf{(i-gs)}. This fact is crucially used in the proof of Corollaries \ref{cor:prox_Stiefel_informal} and \ref{cor:prox_Stiefel}.

	Note that both the Riemannian and the Euclidean proximal surrogates in \eqref{eq:def_prox_surrogate}-\eqref{eq:def_Euclidean_prox_surrogate} are not necessarily $g$-smooth (see Section \ref{sec:had} for details). Our analysis of RBMM with these proximal surrogates uses explicit forms of the Riemannian/Euclidean gradient of the squared geodesic/Euclidean distance function (see Prop. \ref{prop:R_gradient_geodesic_distance}) instead of $g$-smoothness of the surrogates. 
 
    The surrogates in \ref{assumption:A1_smoothness_surrogates}\textbf{(i-rp)} and \ref{assumption:A1_smoothness_surrogates}\textbf{(i-ep)} are identical when the manifold $\mathcal{M}^{(i)}$ is the  Euclidean space. For non-Euclidean manifolds, the difference between these two surrogates could be significant since the Riemannian proximal surrogates in \ref{assumption:A1_smoothness_surrogates}\textbf{(i-rp)} use the geometry of the non-Euclidean manifold $\mathcal{M}^{(i)}$ while the Euclidean proximal surrogates in \ref{assumption:A1_smoothness_surrogates}\textbf{(i-ep)} use the geometry of the ambient Euclidean space. For example, in the case of a sphere, \ref{assumption:A1_smoothness_surrogates}\textbf{(i-rp)} uses the arc length of great circles while \ref{assumption:A1_smoothness_surrogates}\textbf{(i-ep)} uses length of line segments. While the Riemannian proximal surrogates in  \ref{assumption:A1_smoothness_surrogates}\textbf{(i-rp)} can handle non-compact manifolds without compact constraints (e.g., see Sec. \ref{sec:app_fisher_rao}), the Euclidean proximal surrogates in \ref{assumption:A1_smoothness_surrogates}\textbf{(i-ep)}, if applicable, are computationally simpler to handle (e.g., ease of differentiation). We also remark that \ref{assumption:A1_smoothness_surrogates}\textbf{(i-ep)}  holds for unconstrained optimization problems on compact manifolds embedded in Euclidean space, i.e. $\Theta^{(i)} = \mathcal{M}^{(i)}$ and $\mathcal{M}^{(i)}$ is a compact embedded submanifold of a Euclidean space (e.g., Stiefel manifolds).

	We now state our first main result, concerning the asymptotic convergence of RBMM for the constrained single/two-block block Riemannian optimization problem \eqref{eq:def_CROPT_block}. 
	
	\begin{theorem}[Asymptotic convergence to stationary points; one or two blocks]\label{thm:RBMM_1}
		Let $f$ denote the objective function in \eqref{eq:def_CROPT_block} with $m=1$ or $2$.
        Let $(\param_{n})_{n\ge 0}$ be an output of Algorithm \ref{algorithm:BMM} under \ref{assumption:A0_optimal_gap} and \ref{assumption:A1_smoothness_surrogates}\textbf{(i-gs)},\textbf{(ii)}. Further assume that the constraint set $\Theta^{(i)}$ of $\M^{(i)}$ is strongly $g$-convex for $i=1,\dots,m$. Then every limit point of $(\param_{n})_{n\ge 0}$ is a stationary point of $f$ over $\Param$. 
	\end{theorem}
	
	Next, we established similar asymptotic stationarity of RBMM for more than two blocks.  In this case, a well-known counterexample by Powell \cite{powell1973search} shows that block coordinates descent methods for $m\ge 3$ with block optimization may not converge to stationary points. In deed, block coordinate descent is a special case of block MM with the surrogate functions $g_n^{(i)}$ identical to $f_n^{(i)}$. To overcome this issue, the proximal regularized version of block coordinate descent was proposed (\cite{grippo2000convergence, xu2013block, attouch2010proximal}, which had the convergence guarantees under certain assumptions. With this insight, we need to require one more assumption on the surrogate function $g_{n}^{(i)}$. Namely, writing 
	\begin{align}
		g_{n}^{(i)}(\theta) = f_{n}^{(i)}(\theta) + h_{n}^{(i)}(\theta),
	\end{align}
	one can think of the nonnegative `surrogate gap' function $h_{n}^{(i)}$ as a regularization term. A primary example of such regularization in the Euclidean case is the proximal point regularization, where we would take $h_{n}^{(i)}(\theta) = \lambda \lVert \theta-\theta_{n-1}^{(i)} \rVert^{2}$. This has the effect of penalizing a large change in the iterates $\theta_{n}^{(i)}-\theta_{n-1}^{(i)}$. In the general Riemannian case, the following assumption states that the surrogate gap function $h_{n}^{(i)}$ is penalizing a large geodesic distance $d(\theta,\theta_{n-1}^{(i)})$.

	\begin{customassumption}{(A1)}[\textbf{iii}] (Distance-regularizing surrogates)\label{assumption:A1_1}
		There exists a strictly increasing function $\phi:[0,\infty) \rightarrow \R$ such that $\phi(0)=0$ and 
		\begin{align}
			h_{n}^{(i)}(\theta):=g_{n}^{(i)}(\theta)- f_{n}^{(i)}(\theta) \ge c\phi( d(\theta, \theta^{(i)}_{n-1}) )
		\end{align}
		for all $n\ge 1$ and $i=1,\dots,m$. 
	\end{customassumption}

    Note that if we use Riemannian proximal surrogates as in \ref{assumption:A1_smoothness_surrogates}\textbf{(i-rp)}, then \ref{assumption:A1_1}\textbf{(iii)} is automatically satisfied with $c=\lambda/2$ and $\phi(x)=x^{2}$. This also holds with a possibly different parameter $c$ when we use the Euclidean proximal surrogates within a compact constraint set on an embedded manifold as in \ref{assumption:A1_smoothness_surrogates}\textbf{(i-ep)}. In fact, the geodesic distance and Euclidean distance are equivalent over compact sets, see Lemma \ref{lem:equiv_dist} for details. Also, if the surrogates $g_{n}^{(i)}$ are $\rho$-strongly $g$-convex and if $\theta_{n+1}^{(i)}$ is the exact minimizer of $g_{n}^{(i)}$, then by the second-order growth property, one can verify 
	\begin{align}
		h_{n}^{(i)}(\theta) \ge \frac{\rho}{2}  d^{2}(\theta, \theta^{(i)}_{n-1}). 
	\end{align}
	Hence \ref{assumption:A1_1}\textbf{(iii)} is verified with $c=\rho/2$ and $\phi(x)=x^{2}$ in this case as well. 
	
	Note a direct implication of \ref{assumption:A1_1}\textbf{(iii)} is that if $g_{n}^{(i)}(\theta_{n})-f_{n}^{(i)}(\theta_{n})=o(1)$ then $d(\theta^{(i)}_{n}, \theta^{(i)}_{n-1})=o(1)$. This fact will be crucial to our proof of asymptotic stationarity of RBMM for $m\ge 3$ blocks, which is our second main result stated below. Note that in the following result, we do not require geodesic convexity of the constraint sets. 
	
	\begin{theorem}[Asymptotic convergence to stationary points; many blocks]\label{thm:RBMM_2}
		Let $f$ denote the objective function in \eqref{eq:def_CROPT_block} with $m\ge 1$.  Let $(\param_{n})_{n\ge 0}$ be an output of Algorithm \ref{algorithm:BMM}. Under \ref{assumption:A0_optimal_gap}, \ref{assumption:A1_smoothness_surrogates}\textbf{(ii)}, and any of \textbf{(i-gs), (i-rp)} and \textbf{(i-ep)},  every limit point of $(\param_{n})_{n\ge 0}$ is a stationary point of $f$ over $\Param$. 
	\end{theorem}
	
	\begin{remark}[Convexity of the constraint set]
		\normalfont
		\label{rmk:constraints}
		The proof of Theorem \ref{thm:RBMM_1} requires the constraint set $\Theta^{(i)}\subseteq \mathcal{M}^{(i)}$ to be strongly convex w.r.t. the geometry of the ambient manifold $\mathcal{M}^{(i)}$ (due to the Riemmanian line search used in Prop. \ref{prop:Grippo}). However, Theorem \ref{thm:RBMM_2} does not require the constraint set $\Theta^{(i)}\subseteq \mathcal{M}^{(i)}$  to be strongly convex, since we can use Assumption \ref{assumption:A1_1}\textbf{(iii)} to avoid using Prop. \ref{prop:Grippo}. Therefore, Theorem \ref{thm:RBMM_2} applies when $\Theta^{(i)}$ is a manifold by itself (e.g., low-rank manifold or Stiefel manifold) and $\mathcal{M}^{(i)}$ is a Euclidean space in which $\Theta^{(i)}$ is embedded.  See Section \ref{sec:EBMM} for details. 
	\end{remark}

	Now we state our result concerning the rate of convergence of RBMM for the case of Riemannian/Euclidean proximal (and not necessarily $g$-smooth) surrogates. 
	
	\begin{theorem}[Rate of convergence for Riemannian/Euclidean proximal surrogates]\label{thm:BMM_prox}
		Let $f$ denote the objective function in \eqref{eq:def_CROPT_block} with $m\ge 1$.
		Let $(\param_{n})_{n\ge 0}$ be an output of Algorithm \ref{algorithm:BMM} under \ref{assumption:A0_optimal_gap} and  \ref{assumption:A1_smoothness_surrogates}\textbf{(ii)}, and either \textbf{(i-rp)} or \textbf{(i-ep)}. Assume the geodesic convexity of the constraint sets.  Then the following hold: 
		\begin{description}
			\item[(i)] (Rate of convergence) There exists constants $M$ and $c>0$ independent of $\param_0$ such that 
			\begin{equation}
				\min _{1 \leq k \leq n}\left[-\sum_{i=1}^{m} \inf_{\eta \in T_{\theta_{k}^{(i)}}^{*},\|\eta\|\le 1 } \left\langle \grad_{i} f(\param_{k}),\frac{\eta}{\min\{r_0,1\}}\right\rangle\right] \leq \frac{M+c\sum_{n=1}^{\infty}\Delta_n}{\sqrt{n} / \log n}
			\end{equation} 
			
			\item[(ii)] (Worst-case iteration complexity) The worst-case iteration complexity $N_{\epsilon}$ for Algorithm \ref{algorithm:BMM} satisfies $N_{\epsilon}=O(\varepsilon^{-2}\left(\log \varepsilon^{-2}\right)^2)$
		\end{description}
	\end{theorem}
    Theorem \ref{thm:BMM_prox} establishes that RBMM with either Riemannian or Euclidean proximal regularizer achieves an iteration complexity of $\widetilde{O}(\eps^{-2})$. The case of Euclidean proximal regularizer in \ref{assumption:A1_smoothness_surrogates}\textbf{(i-ep)} may be of wider practical interest than the Riemannian proximal regularizer. Note in \ref{assumption:A1_smoothness_surrogates}\textbf{(i-rp)}, the Riemannian proximal regularizer is geodesic distance squared, which may introduce additional computational difficulties, especially when the close-form expression of end-point geodesic distance is unknown. When the underlying manifold is embedded in Euclidean space and the constraint set is compact (or the manifold itself is compact), as in \ref{assumption:A1_smoothness_surrogates}\textbf{(i-ep)}, we can replace the Riemannian proximal regularizer by the Euclidean proximal regularizer, which is often much easier to deal with computationally.

	Next, the theorem below states similar rates of convergence results for $g$-smooth surrogates on general manifolds.
	
	\begin{theorem}[Rate of convergence for $g$-smooth surrogates]\label{thm:BMM_rate}
		Let $f$ denote the objective function in \eqref{eq:def_CROPT_block} with $m\ge 1$. 
		Let $(\param_{n})_{n\ge 0}$ be an output of Algorithm \ref{algorithm:BMM} under \ref{assumption:A0_optimal_gap} and  \ref{assumption:A1_smoothness_surrogates}\textbf{(ii)}, and \textbf{(i-gs)}. Assume the geodesic convexity of the constraint sets. Further assume that \ref{assumption:A1_1}\textbf{(iii)} holds with $\phi(x)=c x^{2}$ for some constant $c>0$. Then the following hold: 
		\begin{description}
			
			\item[(i)] (Worst-case rate of convergence) There exists constants $M,c>0$ independent of $\param_{0}$ such that 
			\begin{equation}\label{eqn:proof_thm_bmm_ii}
				\min _{1 \leq k \leq n}\left[-\sum_{i=1}^{m} \inf_{\eta \in T_{\theta_{k}^{(i)}}^{*},\|\eta\|\le 1 } \left\langle \grad_{i} f(\param_{k}),\frac{\eta}{\min\{r_0,1\}}\right\rangle\right] \leq \frac{M+c\sum_{n=1}^{\infty}\Delta_n}{n^{1/4}/(\log n)^{1/2}}
			\end{equation}

			\item[(ii)] (Worst-case iteration complexity) The worst-case iteration complexity $N_{\epsilon}$ for Algorithm \ref{algorithm:BMM} satisfies $N_{\epsilon}=O\left(\varepsilon^{-4}\left(\log \varepsilon^{-2}\right)\right)$.
			
			\item[(iii)] (Optimal convergence rate)  Further assume that the surrogate gaps $h_{n}^{(i)}= g_{n}^{(i)}-f_{n}^{(i)}$ satisfy $h_{n}^{(i)}(\theta) \le C d^{2}(\theta, \theta_{n}^{(i)})$ for some constant $C>0$. Then 
			\begin{equation}\label{eqn:proof_thm_bmm_iii}
				\min _{1 \leq k \leq n}\left[-\sum_{i=1}^{m} \inf_{\eta \in T_{\theta_{k}^{(i)}}^{*},\|\eta\|\le 1 } \left\langle \grad_{i} f(\param_{k}),\frac{\eta}{\min\{r_0,1\}}\right\rangle\right] \leq \frac{M+c\sum_{n=1}^{\infty}\Delta_n}{n^{1/2}/(\log n)^{1/2}}
			\end{equation}
			and the worst-case iteration complexity $N_{\epsilon}$ for Algorithm \ref{algorithm:BMM} is $N_{\epsilon}=O\left(\varepsilon^{-2}\left(\log \varepsilon^{-2}\right)\right)$.
		\end{description}
	\end{theorem}

Lastly, we state a practical corollary of Theorem \ref{thm:BMM_rate} for Riemannian optimization problems involving Stiefel manifolds (see Ex. \ref{eg:stiefel}). A special case of it was stated in Corollary \ref{cor:prox_Stiefel_informal}  in the introduction. An important fact about Stiefel manifolds is that, if a function $f$ is (Euclidean) $L$-smooth in the ambient Euclidean space, then it is $g$-smooth with respect to the geometry on the Stiefel manifold for some smoothness parameter $L'$ (see Lem. \ref{lem:g_smooth_Stiefel}). Therefore, when the underlying manifolds are either Euclidean or Stiefel, we can apply Theorem \ref{thm:BMM_rate} for Euclidean smooth objectives and surrogates and obtain iteration complexity of $\widetilde{O}(\eps^{-2})$. This expands the applicability of our result to various optimization problems involving Stiefel manifolds.

\begin{corollary}[Complexity of RBMM on Stiefel manifolds]\label{cor:prox_Stiefel}
Suppose each underlying manifold $\mathcal{M}^{(i)}$ is either a Stiefel manifold or a Euclidean space. Assume the objective function $f$ in \eqref{eq:def_CROPT_block} is Euclidean $L$-smooth for some $L>0$. Suppose the surrogates $g_{n}^{(i)}$ are Euclidean $L'$-smooth for some constants $L'>0$ and for some constant $c> 0$, 
\begin{align}\label{eq:quadratic_majorization_euclidean}
    h_{n}^{(i)}(\theta):=g_{n}^{(i)}(\theta)- f_{n}^{(i)}(\theta) \ge c \lVert \theta-\theta_{n-1}^{(i)} \rVert^{2}
\end{align}
for all $n\ge 1$ and $i=1,\dots,m$. Assume the constraint sets $\Theta^{(1)},\dots,\Theta^{(m)}$ are geodesic convex. Allow inexact computation in the sense of \ref{assumption:A0_optimal_gap}\textbf{(ii)} with $d(\cdot,\cdot)$ in \eqref{eq:assumption_inexact_sol_convervence} replaced by the Euclidean distance. Then the iterates produced by Algorithm \ref{algorithm:BMM} asymptotically converge to the set of stationary points and the algorithm has iteration complexity of $\widetilde{O}(\eps^{-2})$. 
\end{corollary}

It is important to note that the conditions one needs to check to apply Corollary \ref{cor:prox_Stiefel} are completely Euclidean, except the $g$-convexity of constraint sets of Stiefel manifolds, which becomes vacuous when there are no additional constraints on the Stiefel manifolds. The results presented in Corollary \ref{cor:prox_Stiefel} are applicable to various MM methods on Stiefel manifolds, as discussed later in Section \ref{sec:Stiefel_prox}, including the recent MM methods investigated in \cite{breloy2021majorization}. Furthermore, in Section \ref{app:GST} we study the geodesically constrained subspace tracking problem as a stylized application.

	\section{Applications}
	\label{sec:examples}

	In this section, we discuss some examples of our general framework of RBMM and its connection to other classical algorithms.

    \subsection{Examples of  manifolds}\label{sec:examples_manifold}

    In this section, we list several examples of manifolds that are typically used in various machine learning problems. 

    \begin{example}[Stiefel Manifold]
		\label{eg:stiefel}
		\normalfont 
		
		The Stiefel manifold $\mathcal{V}^{n\times k}$ is the set of all orthonormal $k$-frames in $\mathbb{R}^n$. That is, it is the set of ordered orthonormal $k$-tuples of vectors in $\R^{n}$, i.e. 
		\begin{equation}
			\mathcal{V}^{n\times k}=\left\{A \in \mathbb{R}^{n \times k}: A^T A=I_k\right\},
		\end{equation}
		where $I_k$ denotes the $k\times k $ identity matrix. For $X\in \R^{n\times k}$, denotes its SVD as $X=U\Sigma V^{T}$, then the projection of $X$ onto $\mathcal{V}^{n\times k}$ is
		\begin{equation}\label{eq:def_proj_stiefel}
			\textup{Proj}_{\mathcal{V}^{n\times k}}(X)=UV^{T}.
		\end{equation}
		
		\hfill $\blacktriangle$
	\end{example}

    \begin{example}[Fixed-rank Matrices Manifold]
		\label{eg:fixed-rank}
		\normalfont
		The set of matrices with fixed rank-$r$
		\begin{equation}\label{eq:def_fixed_rank_manifold}
			\mathcal{R}_r=\left\{X \in \mathbb{R}^{m\times n}: \operatorname{rank}(X)=r\right\}
		\end{equation}
		is a smooth submanifold of $\R^{m\times n}$. For $X\in \R^{m\times n}$, denote its SVD as $X=U\Sigma V^{T}$, where the two matrices $U=\left[u_1, u_2, \ldots, u_m\right]$ and $V=\left[v_1, v_2, \ldots, v_n\right]$ are orthogonal matrices. The diagonal entries of $\Sigma$, which are the singular values of $X$, are written in nonincreasing order, 
		\begin{equation}
			\sigma_1(X) \geq \sigma_2(X) \geq \cdots \geq \sigma_{\min \{n, m\}}(X) \geq 0.
		\end{equation}
		Then the projection of $X$ onto $\mathcal{R}_r$ is given by
		\begin{equation}
			\textup{Proj}_{\mathcal{R}_r}(X)= \sum_{i=1}^{r}\sigma_{i}(X)u_i v_i^{T} := U\Sigma_{r} V^{T}.
		\end{equation}
		\hfill $\blacktriangle$
	\end{example}

    Also, Hadamard manifolds are a class of manifolds that is widely studied in the literature, since it includes many commonly encountered manifolds. \textit{Hadamard manifolds} are Riemannian
	manifolds with nonpositive sectional curvature that are complete and simply connected, see \cite{Burago2001ACI} and \cite{burago1992alexandrov}. Hadamard manifolds have infinite injectivity radii at every point.

    Below we provide some examples of Hadamard manifolds (more details can be found in e.g. \cite{Bacak2014convex}).
	\begin{example}[Euclidean spaces]
		\normalfont
		The Euclidean space $\mathbb{R}^n$ with its usual metric is a Hadamard manifold with constant sectional curvature equal to 0.
		\hfill $\blacktriangle$
	\end{example}
	
	\begin{example}[Hyperbolic spaces]
        \label{eg:hyperbolic_space}
		\normalfont
		We equip $\mathbb{R}^{n+1}$ with the $(-1, n)$-inner product
		$$
		\langle x, y\rangle_{(-1, n)}:=-x^0 y^0+\sum_{i=1}^n x^i y^i
		$$
		for $x:=\left(x^0, x^1, \ldots, x^n\right)$ and $y:=\left(y^0, y^1, \ldots, y^n\right)$. Define
		$$
		\mathbb{H}^n:=\left\{x \in \mathbb{R}^{n+1}:\langle x, x\rangle_{(-1, n)}=-1, x_0>0\right\} .
		$$
		Then $\langle\cdot, \cdot\rangle$ induces a Riemannian metric $g$ on the tangent spaces $T_p \mathbb{H}^n \subset T_p \mathbb{R}^{n+1}$ for $p \in \mathbb{H}^n$. The sectional curvature of $\left(\mathbb{H}^n, g\right)$ is $-1$ at every point.
		\hfill $\blacktriangle$
	\end{example}
	
	\begin{example}[Manifolds of positive definite matrices]
		\normalfont
		\label{eg:PSD}
		The space $\mathbb{S}_{++}^n$ of symmetric positive definite matrices $n\times n$ with real entries is a Hadamard manifold if it is equipped with the Riemannian metric
		\begin{equation}
			\left\langle\Omega_1, \Omega_2\right\rangle_{\Sigma} \triangleq \frac{1}{2} \operatorname{Tr}\left(\Omega_1 \Sigma^{-1} \Omega_2 \Sigma^{-1}\right) \quad \forall \Omega_1, \Omega_2 \in T_{\Sigma} \mathbb{S}_{++}^n.
		\end{equation}
		\hfill $\blacktriangle$
	\end{example}
	
	\subsection{Euclidean block MM}
 \label{sec:EBMM}
	
	When specialized on the standard Euclidean manifold, our RBMM becomes the standard Euclidean Block MM (e.g., see BSUM in \cite{hong2015unified}), where convergence rate for convex and strongly convex objectives are known. Recently, Lyu and Li \cite[Thm. 2.1]{lyu2023block} obtained convergence rates for Euclidean Block MM algorithms for nonconvex objectives with convex constraints. Our general results can recover part of their results: 
	
	\begin{corollary}[Complexity of Euclidean Block MM]\label{cor:EBMM}
		Theorems \ref{thm:RBMM_1}, \ref{thm:RBMM_2}, \ref{thm:BMM_prox}, and \ref{thm:BMM_rate} hold for Algorithm \ref{algorithm:BMM} when the underlying manifolds are Euclidean. In particular, the complexity result in Theorems \ref{thm:BMM_prox} and \ref{thm:BMM_rate} hold for the BSUM algorithm in \cite{hong2015unified}. 
	\end{corollary}

    We remark that \cite[Thm. 2.1]{lyu2023block}] also covers the case when convex surrogates with non-strongly-convex majorization gaps are used along with trust-regions and diminishing radii. Corollary \ref{cor:EBMM} does not cover this case.

    Below we give some examples of the Euclidean block MM. One primary example is Euclidean block proximal updates, namely, applying the surrogates in \eqref{eq:def_prox_surrogate} when the underlying manifold is the Euclidean space. The Euclidean block proximal updates reads
    \begin{equation}
        \theta_{n}^{(i)} \leftarrow \argmin_{\theta\in \Theta^{(i)}} \left(g^{(i)}_n(\theta)= f^{(i)}_n(\theta) + \lambda_n \|\theta-\theta^{(i)}_{n-1}\|^2_F\right).
    \end{equation}
    
	Another example of Euclidean block MM is the following block prox-linear update proposed in \cite{xu2013block}: For minimizing a differentiable function $f$ defined on the Euclidean space, 
	\begin{align}\label{eq:block_prox_linear}
		\theta_{n}^{(i)} &\leftarrow \argmin_{\theta\in \Theta^{(i)}} \left( g_{n}^{(i)}(\theta):= f_{n}^{(i)}(\theta_{n-1}^{(i)} ) + \langle \nabla f_{n}^{(i)}(\theta_{n-1}^{(i)}) ,\, \theta- \theta_{n-1}^{(i)}  \rangle + \frac{\lambda}{2} \lVert \theta - \theta_{n-1}^{(i)}\rVert^{2} \right).
	\end{align}
	In \cite{xu2013block}, under a mild condition, it was shown that the above algorithm converges asymptotically to a Nash equilibrium (a weaker notion than stationary points) and also a local rate of convergence under the Kurdyka-\L{}ojasiewicz condition is established. Notice that when $f$ is block-wise $L$-smooth and if $\lambda\ge L$, then $g_{n}^{(i)}$ is a majorizing surrogate of $f_{n}^{(i)}$ at $\theta_{n-1}^{(i)}$. Thus \eqref{eq:block_prox_linear} is a special instance of our RBMM algorithm in this case and hence Theorems \ref{thm:RBMM_1} and \ref{thm:BMM_rate} apply. That is, our general results imply that the block prox-linear algorithm \eqref{eq:block_prox_linear} in the Euclidean space converges asymptotically to the stationary points (not only Nash equilibrium) and also has iteration complexity of $\widetilde{O}(\eps^{-2})$. 
	
	In fact, the block prox-linear update \eqref{eq:block_prox_linear} coincides with block projected gradient descent with a fixed step size. Indeed, denoting $\nabla:=\nabla f_{n}^{(i)}(\theta_{n-1}^{(i)})$, \eqref{eq:block_prox_linear} is equivalent to 
	\begin{align}\label{eq:def_BPGD}
		\theta_{n}^{(i)} \leftarrow \argmin_{\theta\in \Theta^{(i)}} \left(  \langle \nabla ,\, \theta   \rangle + \frac{\lambda}{2} \lVert \theta \rVert^{2} - \lambda  \langle \theta, \theta^{(i)}_{n-1} \rangle  \right)
		& = \argmin_{\theta\in \Theta^{(i)}} \left\lVert   \theta - \left( \theta^{(i)}_{n-1} - \frac{1}{\lambda} \nabla \right)  \right\rVert^{2} \\ 
		&   = \textup{Proj}_{\Theta^{(i)}} \left( \theta^{(i)}_{n-1} - \frac{1}{\lambda} \nabla \right). 
	\end{align}
	Notice that since $\lambda\ge L$, the above becomes the standard \textit{block projected gradient descent} (Block PGD) update with step-size $\in (0,1/L]$. For block PGD with convex objectives, convergence of function value with complexity $\tilde{O}(\eps^{-1})$ is established in \cite{beck2013convergence}. Recently, in \cite{lyu2023block}, the authors showed the block PGD for smooth non-convex objectives converges to the set of stationary points with complexity $\tilde{O}(\eps^{-2})$. Our general results apply to this classical algorithm as well and recover the same complexity as \cite{lyu2023block}. 

    When the constraint set $\Theta^{(i)}$ is a Riemannian manifold embedded in the Euclidean space, in many cases it is non-convex as a subset of Euclidean space. Therefore, when applying the prox-linear updates as in \eqref{eq:block_prox_linear} and \eqref{eq:def_BPGD}, the complexity result in Theorem \ref{thm:BMM_prox} and \ref{thm:BMM_rate} do not hold, while the asymptotic convergence result in Theorem \ref{thm:RBMM_2} still hold. Moreover, The projection in \eqref{eq:def_BPGD} can be solved exactly on some well-known manifolds, including low-rank manifolds and Stiefel manifolds, which can be found in Section \ref{sec:examples_manifold}. More examples and details can be found in \cite{absil2012projection}. 

    Lastly, we provide two stylized applications of the Euclidean block MM in Section \ref{sec:apps}, namely the robust PCA (Section \ref{sec:RPCA}) and the Riemannian CP-dictionary learning (Section \ref{sec:RCPD_learning}).

	\subsection{Block Riemannian proximal updates on Hadamard manifolds}
	\label{sec:eg_had}
	
	
	We consider the Riemannian proximal surrogates in \ref{assumption:A1_smoothness_surrogates}\textbf{(i-rp)} on Hadamard manifolds. As shown later in Section \ref{sec:had}, this Riemannian proximal surrogate is geodesically strongly convex on Hadamard manifolds. Hence, the subproblem in the minimization step in \eqref{eq:RBMM_MmMm}, i.e. minimizing $g_{n}^{(i)}$ could be solved efficiently using classical Riemannian optimization methods (\cite{udriste1994convex}, \cite{zhang2016first}, \cite{liu2017accelerated}).

	Our general results in Theorems  \ref{thm:RBMM_2} and \ref{thm:BMM_prox} (with Riemannian proximal surrogates) apply to Hadamard manifolds and we obtain the following corollary.

 \begin{corollary}[Complexity of block Riemannian proximal updates on Hadamard manifolds]
     Theorems \ref{thm:RBMM_2} and \ref{thm:BMM_prox} hold for block Riemannian proximal updates \eqref{eq:def_prox_surrogate} on Hadamard manifolds. That is, the RBMM with proximal surrogates converges asymptotically to the set of stationary points and has iteration complexity of $\widetilde{O}(\eps^{-2})$.
 \end{corollary}
In Section \ref{sec:app_fisher_rao}, we give a stylized example of block Riemannian proximal updates solving the optimistic likelihood problem, where the manifold of PSD matrices is involved.

	\subsection{Block Riemannian/Euclidean proximal and prox-linear updates on Stiefel manifolds}\label{sec:Stiefel_prox}

    In this section, we discuss both the Riemannian and Euclidean proximal/prox-linear updates on the Stiefel manifold, as well as their variants and applications.

	It is known that the Stiefel manifold has non-negative sectional curvature (\cite{ziller2007examples}), so it is not a Hadamard manifold. Nevertheless, compact manifolds have a positive injectivity radius (\cite[Thm. III.2.3]{chavel2006riemannian}). In particular, the Stiefel manifold has an injectivity radius of at least $0.89\pi$ (\cite[Eq. (5.13)]{Rentmeesters2013AlgorithmsFD}), which satisfies \ref{assumption:A1_smoothness_surrogates}\textbf{(ii)}. Hence the results in Theorems \ref{thm:RBMM_2} and \ref{thm:BMM_prox} apply and we state the results in the following corollary. 
 
 \begin{corollary}[Complexity of block Riemannian/Euclidean proximal updates on Stiefel manifolds]\label{cor:Stiefel_E/R_proximal}
     Theorems \ref{thm:RBMM_2} and \ref{thm:BMM_prox} hold for both block Riemannian proximal updates \eqref{eq:def_prox_surrogate} and block Euclidean proximal updates \eqref{eq:def_Euclidean_prox_surrogate} on Stiefel manifolds. That is, the RBMM with Euclidean/Riemannian proximal surrogates converges asymptotically to the set of stationary points and has iteration complexity of $\widetilde{O}(\eps^{-2})$.
 \end{corollary}

We remark that in Corollary \ref{cor:Stiefel_E/R_proximal}, both Riemannian proximal surrogates and Euclidean proximal surrogates are allowed. The Riemannian proximal surrogates involve the geodesic distance in the updates, which brings additional computational difficulty since the closed-form solution of the end-point geodesic is unknown for the Stiefel manifold. A survey of numerical methods on computing geodesics on the Stiefel manifold can be found in \cite{edelman1998geometry}. Instead, the Euclidean proximal surrogates use the Euclidean distance function as a regularizer, which provides computational savings.

In fact, our general framework of RBMM can utilize standard (Euclidean) $L$-smooth surrogates for block optimization problems involving Stiefel manifolds, as stated in Corollary \ref{cor:prox_Stiefel}. Below we first give a key lemma from \cite{chen2021decentralized} for deriving Corollary \ref{cor:prox_Stiefel} from Theorems \ref{thm:RBMM_2} and \ref{thm:BMM_rate}. This lemma states any $L$-smooth function in Euclidean space is a $g$-smooth function on the Stiefel manifold, and therefore an $L$-smooth objective function and surrogates satisfy the $g$-smoothness assumption in \ref{assumption:A0_optimal_gap} and \ref{assumption:A1_smoothness_surrogates}. 

\begin{lemma}[Euclidean smoothness implies geodesic smoothness on Stiefel manifold; Lem. 2.4 and Appendix C.1 in \cite{chen2021decentralized}]\label{lem:g_smooth_Stiefel}
    If $f$ is $L$-smooth in Euclidean space $\R^{n\times d}$, then there exists a constant $L_g = L+ L_N$ such that $f$ is $g$-smooth with parameter $L_g$ on the Stiefel manifold $\mathcal{V}^{n\times d}$, where $L_N = \max_{x\in \mathcal{V}^{n\times d}} \|\nabla f(x)\|$.
\end{lemma}

In order for the iteration complexity results in Corollary \ref{cor:prox_Stiefel} to hold, one also needs to verify assumption \ref{assumption:A1_1}\textbf{(iii)}. Namely, the majorization gap should be lower bounded by an increasing function of geodesic distance. The following geometric lemma states that for compact embedded submanifolds of the Euclidean space, the Riemannian and the Euclidean distances are within constant multiples of each other. Therefore, in this case, the quadratic majorization gap in terms of the Euclidean distance in \eqref{eq:quadratic_majorization_euclidean} implies the quadratic majorization gap in terms of the Riemannian distance in \ref{assumption:A1_1}\textbf{(iii)}. This allows us to use a wide range of computationally efficient surrogates on Stiefel manifolds such as Euclidean proximal surrogates \eqref{eq:def_Euclidean_prox_surrogate},  Euclidean prox-linear surrogates \eqref{eq:block_prox_linear}, and \textit{Euclidean regularized linear surrogates}, which will be discussed later in this section. 

See Figure \ref{fig:GST_lb} for an illustration of Lemma \ref{lem:equiv_dist} on Stiefel manifolds.

\begin{lemma}[Equivalence of distance on Riemannian manifold; Lem. 4.1 in \cite{michels2019riemannian}]\label{lem:equiv_dist}
    Let $M \subset \mathbb{R}^n$ be a smooth submanifold, equipped with a Riemannian metric $g$. The geodesic distance between $x,y\in M$ induced by $g$ is denoted as d(x,y). Consider the Euclidean norm $\|\cdot\|$ on $\mathbb{R}^n$. Let $K \subseteq M$ be compact. Then there exists $c>0$ such that for $x, y \in K$,
\begin{equation}\label{eq:equiv_dist}
    c d(x,y) \le \|x-y\| \le d(x,y). 
\end{equation}
\end{lemma}

    \begin{figure*}
    	\centering
    	\vspace{-0cm}
    	\includegraphics[width=0.35\linewidth]{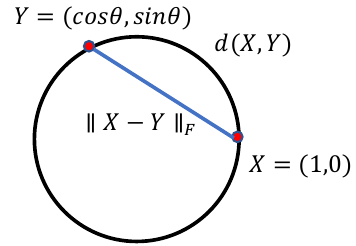}
    	\caption{Illustration of distances on $\mathcal{V}^{n\times p}$ with $n=2, p=1$. The length of the blue segment is Euclidean distance $\|X-Y\|_F$. The length of the arc between $X$ and $Y$ is $d(X,Y)$.}
    	\label{fig:GST_lb}
    \end{figure*}

Below we give a proof of Corollary \ref{cor:prox_Stiefel}. Recall that Corollary \ref{cor:prox_Stiefel_informal} is a direct consequence of Corollary \ref{cor:prox_Stiefel}.

\begin{proof}[\textbf{Proof of Corollary \ref{cor:prox_Stiefel}}]
    To deduce Corollary \ref{cor:prox_Stiefel} from Theorems \ref{thm:RBMM_2} and \ref{thm:BMM_rate}, we verify the assumptions one by one. First, by the hypothesis on Euclidean smoothness of the objective $f$ and the surrogates $g_{n}^{(i)}$ and Lemma \ref{lem:g_smooth_Stiefel} (and also using the definition of $d$ on the product manifold in \eqref{eq:def_product_dist}), we have that $f$ and $g_{n}^{(i)}$s are also $g$-smooth. This verifies \ref{assumption:A0_optimal_gap}\textbf{(i)} and \ref{assumption:A1_smoothness_surrogates}\textbf{(i-gs)}. By Lemma \ref{lem:equiv_dist} and the hypothesis in Corollary \ref{cor:prox_Stiefel}, we verify \ref{assumption:A0_optimal_gap}\textbf{(ii)} and \ref{assumption:A1_smoothness_surrogates}\textbf{(iii)}. Lastly, \ref{assumption:A1_smoothness_surrogates}\textbf{(ii)} is satisfied since both Stiefel manifolds and Euclidean space have a uniformly positive injectivity radius. Then the assertion in Corollary \ref{cor:prox_Stiefel} follows from Theorems \ref{thm:RBMM_2} and \ref{thm:BMM_rate}.
\end{proof}

\subsection{Iteration complexity of MM on Stiefel manifolds in \cite{breloy2021majorization}}

In order to solve minimization over Stiefel manifolds
\begin{align}\label{eq:opt_over_stiefel_main}
		\min_{U\in \mathcal{V}^{p\times k}}	f(U),
\end{align}
Breloy et al. \cite{breloy2021majorization} recently proposed an MM method incorporating linear surrogates and reducing the surrogate minimization problem to a projection onto the Stiefel manifold. More precisely, given an iterate $U_{n-1}\in \mathcal{V}^{p\times k}$ at iteration $n-1$, the surrogate $g_n$ at iteration $n$ takes the following form 
\begin{equation}\label{eq:linear_surr_stiefel}
    g_n (U)= -\operatorname{Tr}\left(\bold{R}^H (U_{n-1}) U\right) - \operatorname{Tr}\left(U^H \bold{R}(U_{n-1})\right) + \textup{const},
\end{equation}
where $^{H}$ is the conjugate transpose operator and  $\bold{R}(\cdot) : \mathbb{C}^{p\times k}\to \mathbb{C}^{p\times k}$ is a matrix function chosen so that $g_{n}(U_{n-1}) = f(U_{n-1})$ and $g_{n}\ge f$.  Since $g_{n}$ is linear and since $U\in \mathcal{V}^{p\times k}$, we have 
\begin{align}
	U_{n+1} \overset{\textup{def}}{=} \argmin_{U\in \mathcal{V}^{p\times k}} g_{n}(U)  =  \argmin_{U\in \mathcal{V}^{p\times k}}\,\,  \lVert \mathbf{R}^{H}(U_{n-1}) - U \rVert_{F}. 
\end{align}
When $\mathbf{R}^{H}(U_{n-1})$ is of full-rank, then the rightmost projection problem has a unique solution given by the projection operator in  \eqref{eq:def_proj_stiefel}. See various applications of this MM method in \cite{breloy2021majorization}.

In \cite{breloy2021majorization}, the authors showed this MM method asymptotically converges to the set of stationary points by adapting a convergence result of Euclidean BMM algorithm in \cite{razaviyayn2013unified}. However, there is no known bound on the iteration complexity due to the non-convexity of the constraint set and the objective function. An iteration complexity of $\tilde{O}(\eps^{-2})$ of Euclidean BMM for non-convex smooth objectives with convex constraints has been obtained recently in \cite{lyu2023block}. However, this result is not applicable here since the Stiefel manifolds cannot be viewed as convex constraint sets within Euclidean spaces.

By applying Corollary \ref{cor:prox_Stiefel}, we can obtain an iteration complexity bound of $\tilde{O}(\eps^{-2})$ for the above MM method on the Stiefel manifold with a slight modification via Euclidean proximal regularization. Namely, for a fixed proximal regularization parameter $\lambda\ge 0$, consider the following surrogate
\begin{align}\label{eq:linear_surr_stiefel_prox}
    \tilde{g}_n (U) &=  g_{n}(U) + \lambda \lVert U - U_{n-1} \rVert_{F}^{2} \\
    &= -\operatorname{Tr}\left((\bold{R}(U_{n-1})+\lambda U_{n-1})^H U\right) - \operatorname{Tr}\left(U^H (\bold{R}(U_{n-1})+\lambda U_{n-1}))\right) + \textup{const}.
\end{align}
Note that minimizing the above surrogate $\tilde{g}_{n}$ is as easy as minimizing $g_{n}$: 
\begin{align}\label{eq:MM_Stiefel_regularized}
	U_{n+1} \overset{\textup{def}}{=} \argmin_{U\in \mathcal{V}^{p\times k}} \tilde{g}_{n}(U)  &=  \argmin_{U\in \mathcal{V}^{p\times k}}\,\,  \lVert \mathbf{R}^{H}(U_{n-1}) + \lambda U_{n-1} - U \rVert_{F}.
\end{align}

Note that the MM update in \eqref{eq:MM_Stiefel_regularized} satisfies the hypothesis of Corollary \ref{cor:prox_Stiefel}. Indeed, the surrogates $g_n$ in \eqref{eq:linear_surr_stiefel} and $\tilde{g}_n$ in \eqref{eq:linear_surr_stiefel_prox} are linear and hence are $L$-smooth in Euclidean space. The addition of Euclidean proximal regularization ensures that we have at least a quadratic majorization gap as in \eqref{eq:quadratic_majorization_euclidean}. Therefore, by Corollary \ref{cor:prox_Stiefel}, the iteration complexity of the MM update in \eqref{eq:MM_Stiefel_regularized} is of $\tilde{O}(\eps^{-2})$. These results are formally stated in Corollary \ref{cor:regularized_linear_Stiefel}.

\begin{corollary}[Complexity of MM with regularized linear surrogates on the Stiefel manifold]\label{cor:regularized_linear_Stiefel}
    Consider the problem of minimizing a differentiable objective $f$ on the Stiefel manifold $\mathcal{V}^{p\times k}$  as in \eqref{eq:opt_over_stiefel_main}. 
   	
    Then the iterates generated by \eqref{eq:MM_Stiefel_regularized} with arbitrary initialization converges asymptotically to the set of stationary points of \eqref{eq:opt_over_stiefel_main}. Moreover, the iteration complexity is $\widetilde{O}(\eps^{-2})$.
\end{corollary}

To the best of our knowledge, this is the first complexity result for MM on Stiefel manifolds in the literature. In particular, the iteration complexity in Corollary \ref{cor:regularized_linear_Stiefel} applies for various problem instances discussed in \cite[Sec. V]{breloy2021majorization} including power iteration for computing top eigenvector, generic non-homogenous quadratic form, and nonconvex subspace recovery.  In Section \ref{app:GST}, we use a similar idea to obtain the same iteration complexity bound for block MM algorithm for the geodesically constrained subspace tracking problem \cite{blocker2023dynamic}.

		\section{Stylized Applications}\label{sec:apps}

        In this section, we will discuss the following four stylized applications of our theory: 
        \begin{enumerate}
            \item Geodesically constrained subspace tracking \cite{blocker2023dynamic};
            \item Optimistic likelihood under Fisher-Rao distance \cite{nguyen2019calculating};
            \item Riemannian CP-dictionary-learning \cite{lyu2020online_tensor, dong2022new};
            \item Robust PCA \cite{candes2009robust, rodriguez2013fast}.
        \end{enumerate}
        The first problem above is an application of block Euclidean proximal updates on Stiefel manifolds, which verifies our assumptions for Theorem \ref{thm:BMM_rate} and we are able to derive a new iteration complexity result (Cor. \ref{cor:geodesic_subspace}), as aforementioned in Corollary \ref{cor:prox_Stiefel} and Section \ref{sec:Stiefel_prox}.
        The second problem above is an application of Riemannian proximal updates on Hadamard manifolds (see Section \ref{sec:eg_had}), which verifies our assumptions for Theorem \ref{thm:BMM_prox} and we are able to derive a new iteration complexity result (Cor. \ref{cor:Fisher_Rao}). For the last two problems above, we only prove asymptotic convergence to stationary points by using Theorem \ref{thm:RBMM_2}. 
        These two problems are formulated as minimizing a cost function involving the Euclidean distance function (e.g., the matrix Frobenius norm) over low-rank manifolds. 
        The Euclidean distance function is not $g$-smooth over low-rank manifolds (see Appendix \ref{sec:Hess_fixed_rank} for details). Therefore, in order to satisfy the  assumption of $g$-smoothness of the objective in \ref{assumption:A0_optimal_gap}\textbf{(i)}, we choose Euclidean geometry as the underlying manifold structure, and take the embedded submanifolds as constraints. In this case, the constraints are not $g$-convex with respect to the underlying Euclidean geometry, so we are not able to apply our iteration complexity results (Theorems \ref{thm:BMM_prox} and \ref{thm:BMM_rate}). However, we can still deduce asymptotic convergence to stationary points by using Theorem \ref{thm:RBMM_2} (see Section \ref{sec:EBMM}), which fortunately does not require $g$-convexity of the constraint sets (see Remark \ref{rmk:constraints}).

		\subsection{Geodesically constrained subspace tracking \cite{blocker2023dynamic}} \label{app:GST}
		Let $X_i \in \R^{d\times l}$ for $i=1,\cdots, T$ be data generated from a low-rank model with noise,
		\begin{equation}
			X_i = U_i G_i + N_i,
		\end{equation}
		where $U_i \in \mathbb{R}^{d \times k}$ has orthonormal columns representing a point on the Grassmannian $\mathcal{G}(k, d)$, the space of all rank-$k$ subspaces in $\mathbb{R}^d ; G_i \in \mathbb{R}^{k \times \ell}$ holds weight or loading vectors; and $N_i \in \mathbb{R}^{d \times \ell}$ is an independent additive noise matrix. 
		
		For the geodesic subspace tracking problem, we observe $X_i$ and our objective is to estimate $U_i$ for $i=1, \cdots, T$. Let $\mathcal{V}^{d\times k}$ denote the Stiefel manifold. We model each $U_i$ as an orthonormal basis whose span is sampled from a single continuous Grassmanian geodesic $U(t):[0,1]\to \mathcal{V}^{d\times k}$, parameterized as follows: For $H,Y\in \mathcal{V}^{d\times k}$ consisting of orthogonal columns, i.e., $H^{T}Y=O$,
		\begin{equation}
			U_i= U(t_i)= H \cos(\Theta t_i) + Y \sin(\Theta t_i).
		\end{equation}
		Here  
        $\Theta \in \R^{k \times k}$ is a diagonal matrix where its $j$th diagonal entry, $\theta_j$, is the $j$-th principal angle between the two endpoints of the geodesic. We assume either the time-points $t_i$ are given, or the observed matrices $X_i$
		are equidistant along a geodesic curve.

		The objective function is formulated as follows,
		\begin{equation}
			\label{eqn:obj_subspace_tracking}
			f(U)=f(H,Y,\Theta)=\min_{\{G_i\}_{i=1}^{T}}\|X_i -U(t_i) G_i\|^2_F = -\sum_{i=1}^{T}\|X_i^{T} U(t_i)\|^2_F +c,
		\end{equation}
		where 
        for the last equality we have substituted the optimal $G_i=U\left(t_i\right)^{T} X_i$ and $c$ is a constant (see \cite{golub2003separable}).

        To approximately minimize \eqref{eqn:obj_subspace_tracking}, Blocker et al. proposed a two-block MM approach in \cite{blocker2023dynamic}, where one alternatively optimizes two block parameters $Q:=(H, Y)\in \mathcal{V}^{d \times 2 k}$ and $\Theta$. Here we propose a proximal regularized version of the BMM method in \cite{blocker2023dynamic} and establish its asymptotic convergence property in Corollary \ref{cor:geodesic_subspace}. 
        For proximal regularization parameters $\lambda_{Q},\lambda_{\Theta}\ge 0$, the proposed algorithm reads as 
        \begin{align}
			& Q_{n+1}  \leftarrow WV^{T}, \,\, \textup{where $W \Sigma V^{T}$ is the SVD of $\lambda_{Q} Q_{n} -  \nabla_{Q} f(Q_{n},\Theta_n)$} \label{eq:Q_MM_subspace_traking}\\ 
            & (\theta_j)_{n+1} \leftarrow 
			(\theta_j)_{n}-\frac{1}{w_{j}((\theta_j)_n)+\lambda_{\Theta}}\nabla f^{(2)}_{n+1, j}\left((\theta_j)_n\right) \,\, \textup{for $j=1,\dots,k$}, \label{eq:subspace_g2}
		\end{align}
        where $(\theta_{j})_{n}$ denotes the $j$th diagonal entry of the $k\times k$ diagonal matrix $\Theta_{n}$ at iteration $n$, $w_j=\sum_{i=1}^T w_{f_{i,j}}$ is the "weighting function" defined in \cite{blocker2023dynamic}, and 
\begin{equation}\label{eq:geodesic_subspace_theta_marginal}
			f^{(2)}_{n+1, j}\left(\theta_j\right) :=-\sum_{i=1}^{T}(r_{i, j})_{n+1} \cos \left(2 \theta_j t_i-(\phi_{i, j})_{n+1}\right)+(b_{i, j})_{n+1}.
        \end{equation}
        The definition of the parameters $\phi_{i,j}$, $r_{i,j}$ and $b_{i,j}$ in \eqref{eq:geodesic_subspace_theta_marginal} can be found in \cite{blocker2023dynamic}, which we also provide in Appendix \ref{appendix:GST} for completeness. Note that when $\lambda_Q=0$, \eqref{eq:Q_MM_subspace_traking} becomes the updates for $Q$ in \cite{blocker2023dynamic}. Similarly, $\lambda_\Theta=0$ gives the updates for $\Theta$ in \cite{blocker2023dynamic}. Thus our algorithm \eqref{eq:Q_MM_subspace_traking}-\eqref{eq:subspace_g2} generalizes the BMM algorithm proposed in \cite{blocker2023dynamic}.

        We first argue why we can cast the above algorithm as RBMM. The discussion we provide here is a minor modification of the derivation in \cite{blocker2023dynamic}. First, we discuss the $Q$-update \eqref{eq:Q_MM_subspace_traking}. Let $Z_i=[\cos(\Theta t_i)\;;\;\sin(\Theta t_i)]$, which is the vertical concatenation of $\cos(\Theta t_i)$ and $\sin(\Theta t_i)$. The objective function can be rewritten as
		\begin{equation}
			\label{eqn:obj_subspace_QTheta}
			f(Q, \Theta) = -\sum_{i=1}^{T}\|X_{i}^{T}QZ_{i}\|_F^{2}
		\end{equation}
		and the gradient with respect to the first block $Q$ is given by $\nabla_{Q}f= -2\sum_{i=1}^{T}X_i X_{i}^{T}QZ_{i} Z_{i}^{T}$. The marginal objective function for updating $Q$ can be rewritten as 
		\begin{align}
			f_{n+1}^{(1)}(Q) &:=-\left\langle \sum_{i=1}^{T}X_i X_{i}^{T}Q(Z_{i})_n (Z_{i})_n^{T}, Q \right\rangle,
		\end{align}
		which is a concave-up quadratic function in $Q$. Also note that $f_{n+1}^{(1)}$ is $L$-smooth for some constant $L>0$ over the compact parameter space.

		We consider the following proximal majorizer of $f_{n+1}^{(1)}$: For $\lambda\ge 0$,
		\begin{align}\label{eq:GST_Q_surrogates}
			g_{n+1}^{(1)}(Q) &:=-\left\langle \sum_{i=1}^{T}X_i X_{i}^{T}Q_n(Z_{i})_n (Z_{i})_n^{T}, Q\right\rangle +\frac{\lambda}{4} \lVert Q-Q_{n} \rVert_{F}^{2} \\
			&= \frac{1}{2}\langle \nabla_{Q}f(Q_{n},\Theta_n), Q\rangle+ \frac{\lambda}{4} \lVert Q-Q_{n} \rVert_{F}^{2}.
		\end{align}
		We claim that 
		\begin{align}\label{eq:Q_MM_subspace_traking0}
			Q_{n+1} &=\argmin_{Q\in \mathcal{V}^{d\times 2k}} g_{n+1}^{(1)}(Q) =
			\begin{cases}
				\textup{Proj}_{\mathcal{V}^{d\times 2k}} \left(  Q_{n} - \frac{1}{\lambda} \nabla_{Q} f(Q_{n},\Theta_n) \right) & \textup{if $\lambda>0$}   \\
				\textup{Proj}_{\mathcal{V}^{d\times 2k}} \left(-  \nabla_{Q} f(Q_{n},\Theta_n) \right) & \textup{if $\lambda=0$}.
			\end{cases}
		\end{align}
		Indeed, the above MM update for $\lambda>0$ follows from \eqref{eq:def_BPGD}. For $\lambda=0$, we use the fact that $Q^{T}Q = I$ for all $Q$ in the Stiefel manifold so that 
		\begin{align}
			\argmin_{Q\in \mathcal{V}^{d\times 2k}}   \langle \nabla_{Q}f(Q_{n},\Theta_n), Q\rangle  =  \argmin_{Q\in \mathcal{V}^{d\times 2k}}  \lVert Q + \nabla_{Q}f(Q_{n},\Theta_n) \rVert^{2}_{F} =  \textup{Proj}_{\mathcal{V}^{d\times 2k}} \left(-  \nabla_{Q} f(Q_{n},\Theta_n) \right).
		\end{align}
		Notice that projecting onto the Stiefel manifold can be easily done by SVD (see, e.g., \cite{absil2012projection}). Hence \eqref{eq:Q_MM_subspace_traking0} coincides with \eqref{eq:Q_MM_subspace_traking}.

Next, we discuss the $\Theta$-update. 
For conciseness, we only put the expression of the loss function and surrogates here. For full details, we refer the readers to \cite{blocker2023dynamic}. 
The marginal loss function for $\Theta$ is separable for each diagonal element $\theta_j$ of $\Theta$. 
Consider the following proximal majorizer of $f_{n+1,j}^{(2)}$ in \eqref{eq:geodesic_subspace_theta_marginal}: For $\lambda_{\Theta} \ge 0$,
		\begin{equation}\label{eq:geodesic_subspace_theta_surrogate}
			g_{n+1,j}^{(2)}(\theta_j)= f^{(2)}_{n+1, j}\left((\theta_j)_n\right) + \nabla f^{(2)}_{n+1, j}\left((\theta_j)_n\right) \left(\theta_j - (\theta_j)_n\right) + \frac{w_{j}((\theta_j)_n)+\lambda_{\Theta}}{2}\left(\theta_j - (\theta_j)_n\right)^2
		\end{equation}
		where $(\theta_j)_n$ is the value of $\theta_j$ at iteration $n$. Then by using \eqref{eq:def_BPGD}, we see that \eqref{eq:subspace_g2} coincides with minimizing the majorizing surrogate $g_{n+1,j}^{(2)}$ above. 

Now we discuss the convergence of this block MM algorithm. In fact, we can apply Corollary \ref{cor:prox_Stiefel} and directly get the convergence and complexity results. The asymptotic convergence result in Corollary \ref{cor:prox_Stiefel} for the updates \eqref{eq:Q_MM_subspace_traking}-\eqref{eq:subspace_g2} with $\lambda_Q,\lambda_\Theta \ge 0$. Note when $\lambda_Q=\lambda_\Theta=0$, the updates \eqref{eq:Q_MM_subspace_traking}-\eqref{eq:subspace_g2} become the vanilla block MM method in \cite{blocker2023dynamic}. Moreover, when the proximal parameter $\lambda_{Q},\lambda_{\Theta}$ are strictly positive, the complexity result in Corollary \ref{cor:prox_Stiefel} holds. We state the convergence and complexity results of this proximal regularized MM method in the following corollary:

\begin{corollary}[Convergence and complexity of regularized BMM for geodesically constrained subspace tracking] \label{cor:geodesic_subspace}
    Given a sequence of data $X_{i}$ for $i=1, \cdots, T$. Let $\param_k = (Q_k, \Theta_k)$ be generated by \eqref{eq:Q_MM_subspace_traking}-\eqref{eq:subspace_g2} with arbitrary initialization $\param_0 \in \Param = \mathcal{V}^{d\times 2k}\times \R^{k\times k}$. Suppose the proximal parameters $\lambda_Q$ and $\lambda_\Theta$ are non-negative. Then the limit points of $(\param_n)_{n\ge0}$ are stationary points of problem \eqref{eqn:obj_subspace_QTheta}. Moreover, if $\lambda_Q,\lambda_\Theta>0$, then the iteration complexity is $\widetilde{O}(\eps^{-2})$.
    \end{corollary}
\begin{proof}
In fact, in order to apply Corollary \ref{cor:prox_Stiefel}, we only need to verify the Euclidean smoothness of the marginal loss function. Note 
\begin{align}
        \nabla_{Q}f &= -2\sum_{i=1}^{T}X_i X_{i}^{T}QZ_{i} Z_{i}^{T},\qquad
        \nabla f^{(2)}_{n+1, j}\left(\theta_j\right) = 2 \sum_{i=1}^{T}(r_{i, j})_{n+1} t_i \sin \left(2 \theta_j t_i-(\phi_{i, j})_{n+1}\right),
    \end{align}
    so $f$ is block-wise Euclidean $L$-smooth. Hence, the convergence and complexity results in Corollary \ref{cor:prox_Stiefel} follow.
\end{proof}

Now we compare our results with other existing works. We remark that the BMM algorithm in   \cite{blocker2023dynamic} is based on the MM algorithm on Stiefel manifold by Breloy et al. \cite{breloy2021majorization}. As aforementioned in Section \ref{sec:Stiefel_prox}, the asymptotic convergence to the set of all stationary points in \cite{breloy2021majorization} was established by adopting a convergence result of Euclidean BMM algorithm in \cite{razaviyayn2013unified}. The complexity results are still unknown due to the non-convexity of the constraint set. Here, we give the first complexity result in the literature. Moreover, our method is computationally efficient with close-form updates as shown in \eqref{eq:Q_MM_subspace_traking}-\eqref{eq:subspace_g2}.

		\begin{figure}  
			\begin{subfigure}{0.49\textwidth}
				\includegraphics[width=\linewidth]{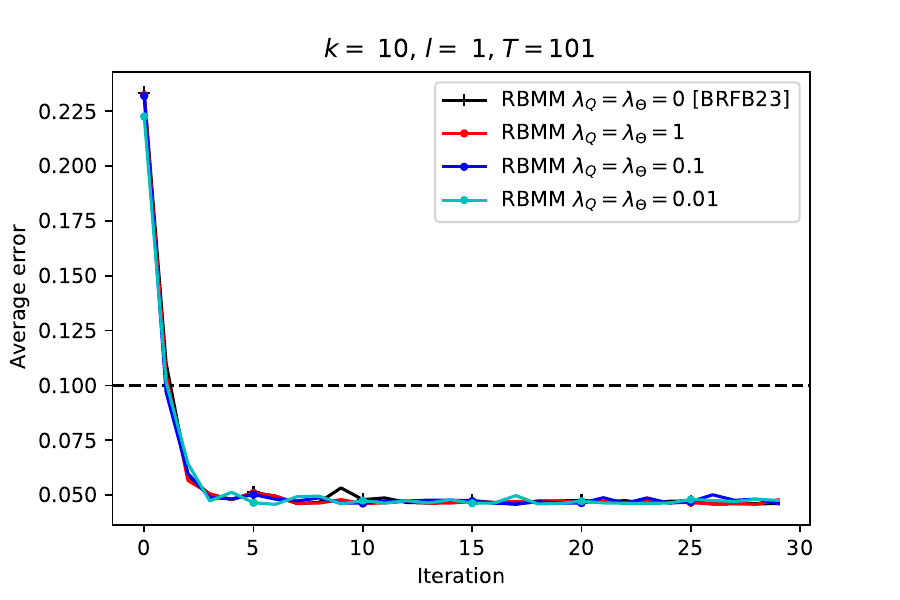}
				\caption{}
			\end{subfigure}
			\hfill 
			\begin{subfigure}{0.49\textwidth}
				\includegraphics[width=\linewidth]{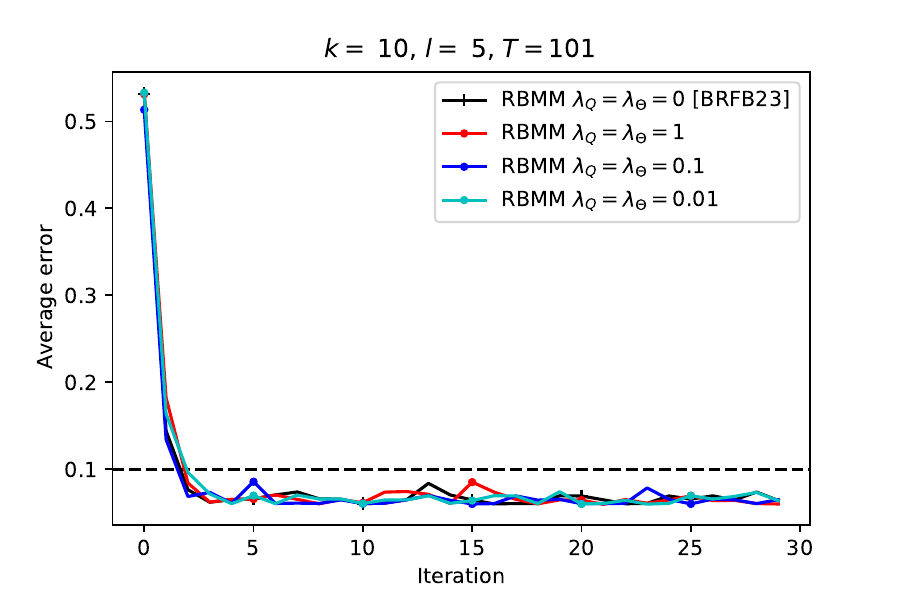}
				\caption{}
			\end{subfigure}
			\bigskip  
			\begin{subfigure}{0.49\textwidth}
				\includegraphics[width=\linewidth]{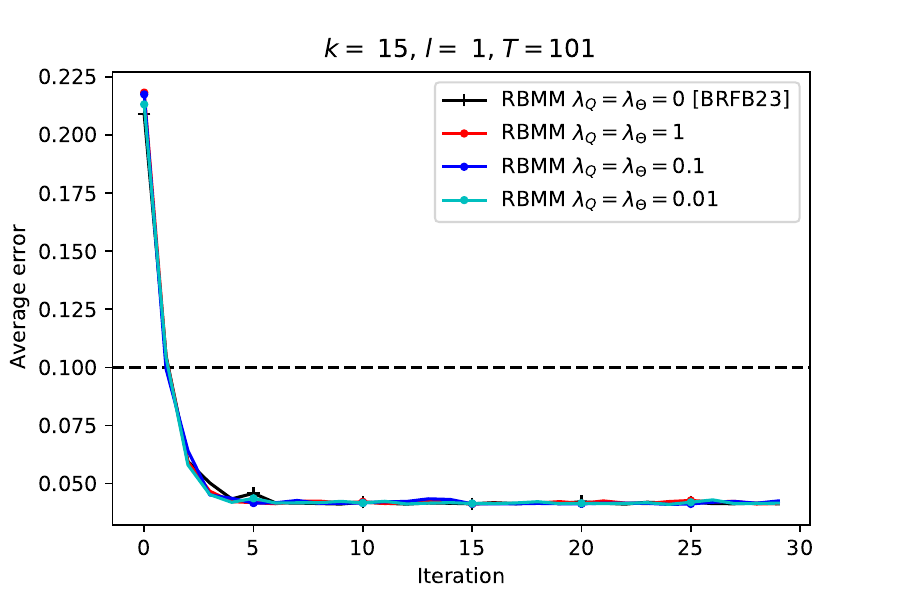}
				\caption{}
			\end{subfigure}
			\hfill 
			\begin{subfigure}{0.49\textwidth}
				\includegraphics[width=\linewidth]{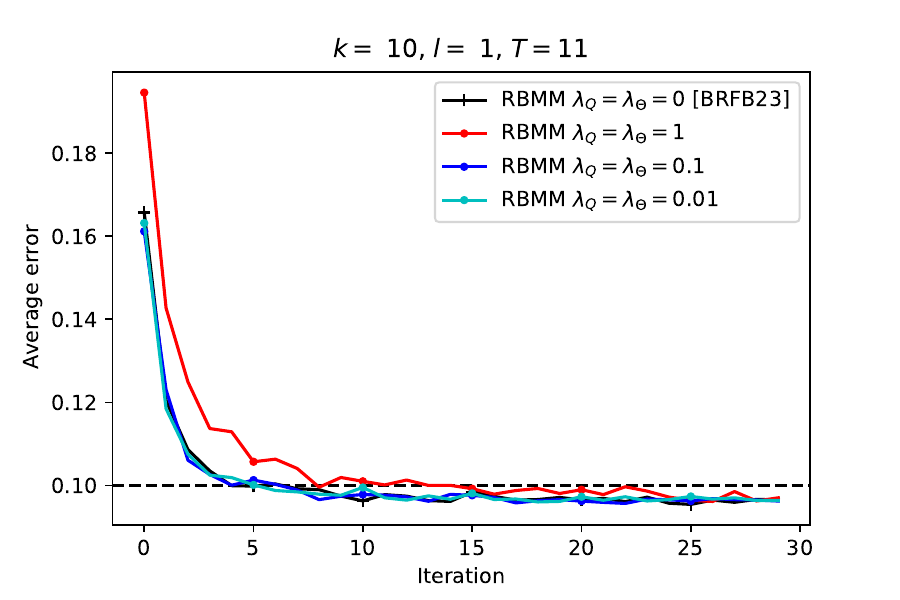}
				\caption{}
			\end{subfigure}
			\caption{Convergence of RBMM in geodesic error under different settings. Average geodesic error is computed over 50 independent trials. The dimension is $d=30$ and the additive Gaussian noise has a standard deviation $\sigma=0.1$. The value of other parameters is shown in the title for each panel. Adding Euclidean proximal regularizer does not accelerate the convergence, mainly due to the orthogonal property on Stiefel manifolds. Namely, any point $Q$ on the Stiefel manifold satisfies $Q^T Q=\mathbf{I}$. Hence the proximal regularizer degenerates into a linear term providing with limited acceleration.} 
			\label{fig:geo_conv}
		\end{figure}
		
		Next, we discuss the results of numerical experiments. In the case of synthetic data, we have access to the true geodesic, so we could compare the estimated geodesic $\hat{U}(t)$ with the true geodesic $U(t)$. The following error metric is used, which is the square root of the average squared subspace error between corresponding points along the geodesic,
		\begin{equation}
			\label{eqn:geo_error}
			\text { Geodesic Error }=\sqrt{\int_0^1 \frac{1}{2 k}\left\|\hat{U}(t) \hat{U}(t)^{T}-U(t) U(t)^{T}\right\|_{\mathrm{F}}^2 \mathrm{~d} t}.
		\end{equation}
		In practice, we approximate the integral by Riemann sum using the sample points. The geodesic error \eqref{eqn:geo_error} takes minimum value of 0 when $\operatorname{span}(\hat{U}) =\operatorname{span}(U)$ and maximum value of 1 when $\operatorname{span}(\hat{U}) \perp \operatorname{span}(U)$.
		
		In the numerical experiments, we set the dimension $d=30$ and set the elements of the noise matrix $N_i$ to be independent Gaussian noise with standard deviation $\sigma=0.1$. The other parameters including the rank $k$, number of samples $l$, number of time points $T$, and proximal parameters $\lambda_Q$ and $\lambda_\Theta$ are set differently for different experiments. We run each experiment 50 times independently to compute the average geodesic error. As shown in Figure \ref{fig:geo_conv}, RBMM with different $\lambda$ converges  under different settings and is as good in terms of convergence speed as the block MM method in \cite{blocker2023dynamic}. Furthermore, Figure \ref{fig:geo_conv} shows that changing the proximal parameter $\lambda_Q$ and $\lambda_\Theta$ only slightly affects the rate of convergence, this is due to the following two reasons: first, there are only two blocks for this problem, which gives a relatively simple setting for RBMM; second, note $Q\in \mathcal{V}^{d\times 2k}$ which gives $Q^T Q=I_{2k}$, therefore the quadratic term added would reduce to a linear term, which does not significantly accelerate convergence. For more numerical results of this geodesic subspace tracking problem, we refer the reader to \cite{blocker2023dynamic}.

		\subsection{Optimistic likelihood under Fisher-Rao distance \cite{nguyen2019calculating}}
  \label{sec:app_fisher_rao}
  
		Given a set of i.i.d data points $x^M \triangleq x_1, \ldots, x_M \in \mathbb{R}^N$ that are generated from one of several Gaussian distributions $\mathbb{P}_c$, where $c\in C$ and $|C|<\infty$. We want to determine the distribution $\mathbb{P}_{c^*}$, $c^* \in C$, under which the following log-likelihood function $\ell(x^M,\mathbb{P}_c)$ is maximized, i.e. we want to solve 
		\begin{equation}
			c^* =\argmax_{c\in C}\left\{\ell\left(x^M, \mathbb{P}_c\right) \triangleq-\frac{1}{M} \sum_{m=1}^M\left(x_m-\mu_c\right)^{T} \Sigma_c^{-1}\left(x_m-\mu_c\right)-\log \operatorname{det} \Sigma_c\right\},
		\end{equation}
		where $\mu_c$ and $\Sigma_c$ denote the means and covariance matrices of $\mathbb{P}_c$. Since methods that sample points to get good estimators of $\mu_c$ and $\Sigma_c$ are usually costly, we consider the following \textit{optimistic likelihood problem} instead. Namely, instead of aiming to find the Gaussian distribution from the set, we look for a Gaussian distribution that maximizes the likelihood close to the empirical distribution under some distance measure. More precisely, we consider the following optimistic likelihood problem
		
		\begin{equation}
			\max _{\mathbb{P} \in \mathcal{P}_c} \ell\left(x^M, \mathbb{P}\right) \text { with } \mathcal{P}_c=\left\{\mathbb{P} \in \mathcal{P}: \varphi\left(\hat{\mathbb{P}}_c, \mathbb{P}\right) \leq \rho_c\right\},
		\end{equation}
		where $\mathcal{P}$ is the set of all non-degenerate Gaussian distributions on $\mathbb{R}^N$, $\hat{\mathbb{P}}_c$ is the empirical distribution estimated from training data, $\varphi$ is the Fisher-Rao distance, and $\rho_c \in \mathbb{R}_{+}$are the radii of the ambiguity sets $\mathcal{P}_c$. For conciseness, we put the details of the Fisher-Rao distance in Appendix \ref{appendix:fisher_rao}. We refer the readers to \cite{nguyen2019calculating} for full details.

		Now denoting the empirical mean and covariance from the data by $\hat{\mu}$ and $\hat{\Sigma}$, we can explicitly write down the optimization problem as
		\begin{equation}\label{eqn:loss_likelihood}
			\min_{\mu,\Sigma}f(\mu,\Sigma) \triangleq\left\langle M^{-1} \sum_{m=1}^M\left(x_m-\mu\right)\left(x_m-\mu\right)^{T}, \Sigma^{-1}\right\rangle+\log \operatorname{det} \Sigma,
		\end{equation}
		where $\mu\in \Theta^{(1)}=\left\{\mu \in \mathbb{R}^{N}:(\mu-\hat{\mu})^{T} (\mu-\hat{\mu}) \leq \rho_{1}^2\right\}$ and $\Sigma \in \Theta^{(2)}=\left\{\Sigma \in \mathbb{S}_{++}^N: d(\Sigma, \hat{\Sigma}) \leq \rho_{2}\right\}$. Here $d(\Sigma, \hat{\Sigma})$ is the Fisher-Rao distance of two Gaussian distributions with identical mean (see Appendix \ref{appendix:fisher_rao}).

		Denote $\mu_{n}$ and $\Sigma_{n}$ as the $\theta_{n}^{(1)}$ and $\theta_{n}^{(2)}$ generated by Algorithm \ref{algorithm:BMM} applied on $f(\mu,\Sigma)$. The marginal objective functions are denoted by
		\begin{align}
			&f^{(1)}_n:=f(\mu,\Sigma_{n-1})= \left\langle M^{-1} \sum_{m=1}^M\left(x_m-\mu\right)\left(x_m-\mu\right)^{T}, \Sigma^{-1}_{n-1}\right\rangle+\log \operatorname{det} \Sigma_{n-1} \\
			&f^{(2)}_n:=f(\mu_{n},\Sigma)= \left\langle S_n, \Sigma^{-1}\right\rangle+\log \operatorname{det} \Sigma \qquad \text{where} \qquad S_n=M^{-1} \sum_{m=1}^M\left(x_m-\mu_n\right)\left(x_m-\mu_n\right)^{T}.
		\end{align}

		As aforementioned in Example \ref{eg:PSD}, this manifold of positive definite matrices is a Hadamard manifold. Thus we could construct Riemannian proximal surrogate functions as in \eqref{eq:def_prox_surrogate}, more 
 discussions can be found in Section \ref{sec:had}, i.e. 
		\begin{align}\label{eq:ite_fisher_rao}
			&\mu_n = \argmin_{\mu\in\R^N} g_{n}^{(1)}(\mu): =\left\langle M^{-1} \sum_{m=1}^M\left(x_m-\mu\right)\left(x_m-\mu\right)^{T}, \Sigma^{-1}_{n-1}\right\rangle+\log \operatorname{det} \Sigma_{n-1} + \frac{\lambda}{2}\|\mu-\mu_{n-1}\|^2 \\
			&\Sigma_n = \argmin_{\Sigma\in\mathbb{S}^{N}_{++}}g_{n}^{(2)}(\Sigma):= \left\langle S_n, \Sigma^{-1}\right\rangle+\log \operatorname{det} \Sigma +\frac{\lambda}{4}\left\|\log \left(\Sigma_{n-1}^{-\frac{1}{2}} \Sigma \Sigma_{n-1}^{-\frac{1}{2}}\right)\right\|_F^2. 
		\end{align}
		Note this block MM is a special instance of proximal methods on the Hadamard manifold discussed in Section \ref{sec:eg_had}. Hence Theorems \ref{thm:RBMM_2} and \ref{thm:BMM_prox} apply. We state it as a corollary as follows,
\begin{corollary}[Complexity of Riemannian proximal updates for optimistic likelihood problem]\label{cor:Fisher_Rao}
    Given a set of data points $(x_i)_i$ for $i=1, \cdots, M$. Let $\param_n=(\mu_n,\Sigma_n)$ be the iterates generated by \eqref{eq:ite_fisher_rao} with arbitrary initialization $\param_0\in \Param\subseteq \R^{N}\times \mathbb{S}^{N}_{++}$. Suppose the proximal regularization parameter $\lambda$ is strictly positive. Then the limit points of $(\param_n)_n$ are stationary points of problem \eqref{eqn:loss_likelihood}. Furthermore, it has iteration complexity of $\widetilde{O}(\eps^{-2})$.
\end{corollary}
\begin{proof}
    Note we only need to show \ref{assumption:A0_optimal_gap}, \ref{assumption:A1_smoothness_surrogates}, and also the geodesic convexity of the constraint set are satisfied. The proof of \ref{assumption:A0_optimal_gap}\textbf{(i)} and geodesic convexity of the constraint sets are established based on some propositions adapted from \cite{nguyen2019calculating} which can be found in Appendix \ref{appendix:fisher_rao}. \ref{assumption:A1_smoothness_surrogates}\textbf{(i-rp)} and \ref{assumption:A0_optimal_gap}\textbf{(ii)} holds by our choice of Riemannian proximal surrogates. \ref{assumption:A1_smoothness_surrogates}\textbf{(ii)} holds since the underlying manifolds are Hadamard manifolds. Hence the corollary holds as a result of Theorems \ref{thm:RBMM_2} and \ref{thm:BMM_prox}.
\end{proof}

		\begin{figure}
			\centering
			\begin{subfigure}[t]{0.51\textwidth}
				\centering
				\includegraphics[width=\textwidth]{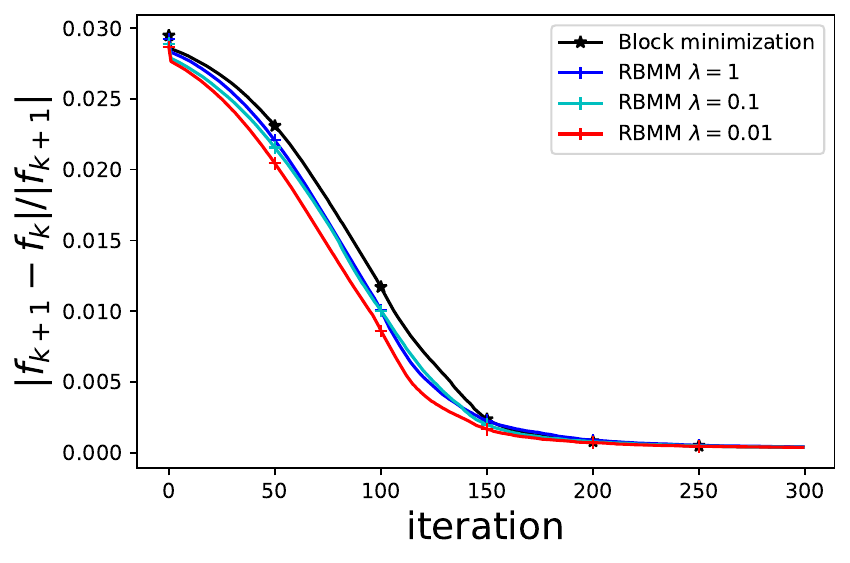}
				\caption{}
			\end{subfigure}
			\hfill
			\begin{subfigure}[t]{0.47\textwidth}
				\centering
				\includegraphics[width=\textwidth]{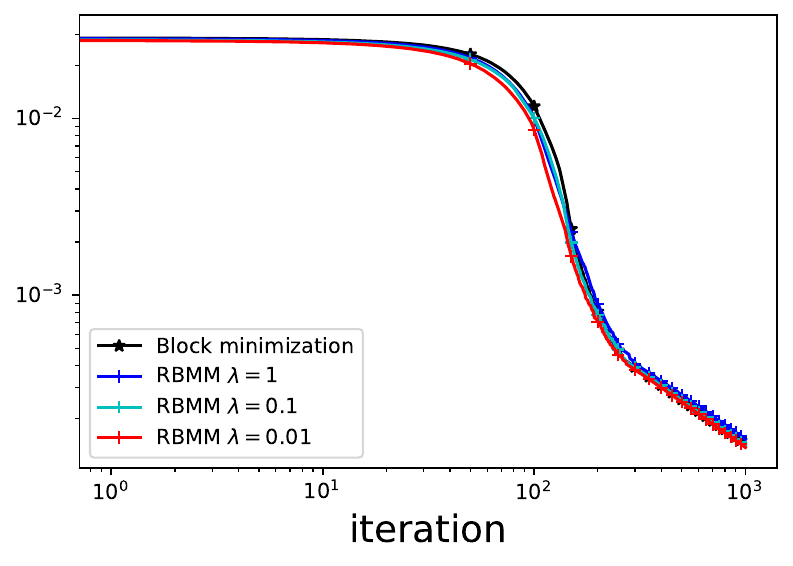}
				\caption{}
			\end{subfigure}
			\caption{Comparison of block minimization and RBMM applied to optimistic likelihood problem under Fisher-Rao distance. RBMM is implemented with $\lambda=0.01,0.1,1$ respectively. Panel (A) shows the averaged error for the first 300 iterations and Panel (B) shows the averaged error using a log-log plot for 1000 iterations.}
			\label{fig:likelihood}
		\end{figure}

		Now we show some numerical results. For numerical experiments, we use the same setup as in \cite{nguyen2019calculating} and study the empirical convergence behavior. We compare the performance of block minimization using $f_{k}^{(i)}(\theta)$ for $i$-th block and RBMM using $g_{k}^{(i)}(\theta)=f_{k}^{(i)}(\theta)+\lambda d^{2}(\theta,\theta_{k-1}^{(i)})$ for $i$-th block with different values of $\lambda$. We set the dimension of the data to be $N=10$ and denote $f_k =f (\mu_k, \Sigma_k)$, the relative improvement is thus denoted as $|f_{k+1}-f_k|/f_{k+1}$, which is computed via 10 independent experiments. Numerical results are shown in Figure \ref{fig:likelihood}. While RBMM and block minimization both perform well, RBMM is slightly faster, where the fastest convergence of RBMM is achieved with $\lambda=0.01$.

		\subsection{Riemannian CP-dictionary-learning}\label{sec:RCPD_learning}
		
		\begin{figure*}[h]
			\centering
			
			\begin{subfigure}[b]{0.49\textwidth}
				\centering
				\includegraphics[width=\textwidth]{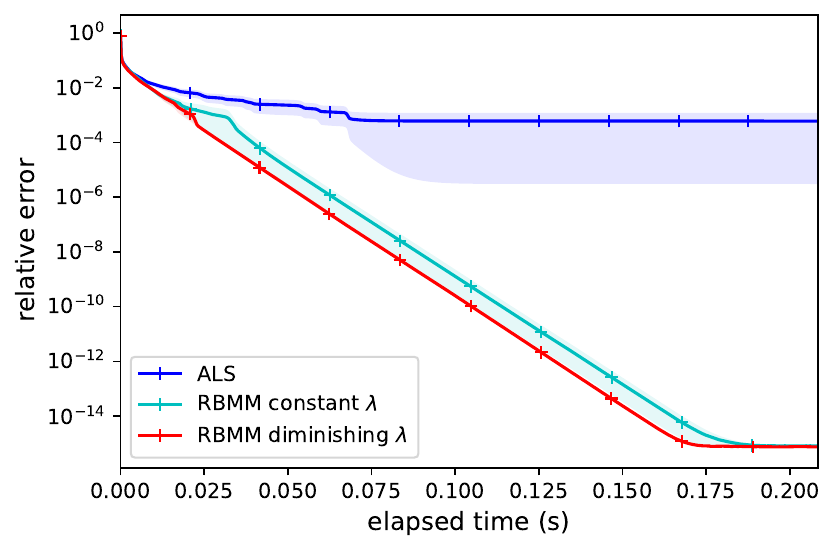}
				\caption{Euclidean: Synthetic data 1}
				\label{fig:subfig1}
			\end{subfigure}
			\hfill
			\begin{subfigure}[b]{0.49\textwidth}
				\centering
				\includegraphics[width=\textwidth]{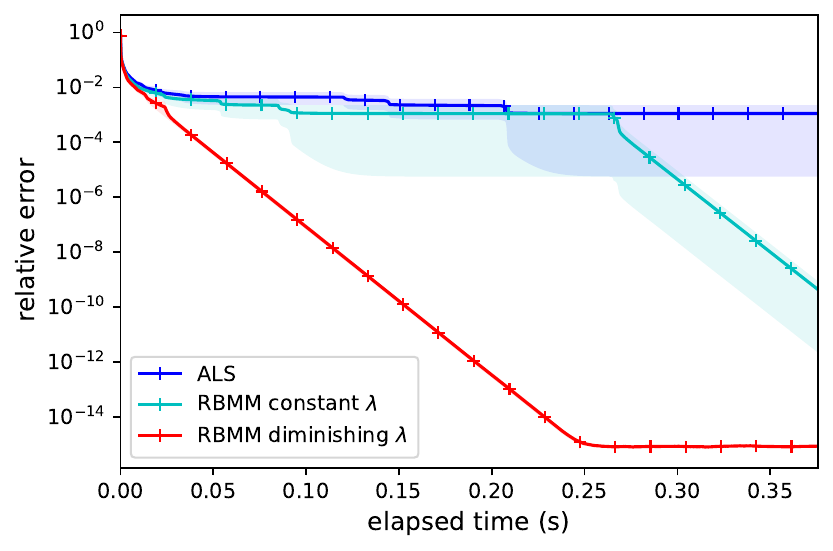}
				\caption{Euclidean: Synthetic data 2}
				\label{fig:subfig2}
			\end{subfigure}
			\hfill
			\begin{subfigure}[b]{0.49\textwidth}
				\centering
				\includegraphics[width=\textwidth]{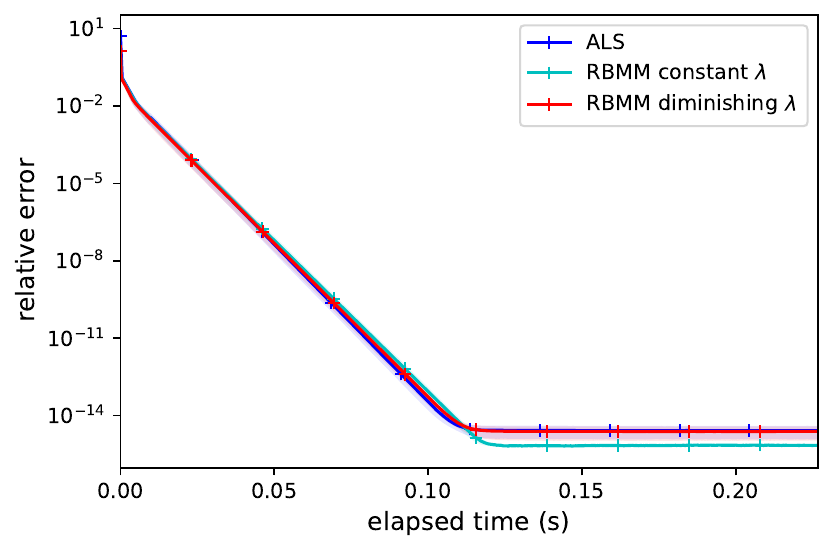}
				\caption{Stiefel: Synthetic data 1}
				\label{fig:subfig5}
			\end{subfigure}
			\begin{subfigure}[b]{0.49\textwidth}
				\centering
				\includegraphics[width=\textwidth]{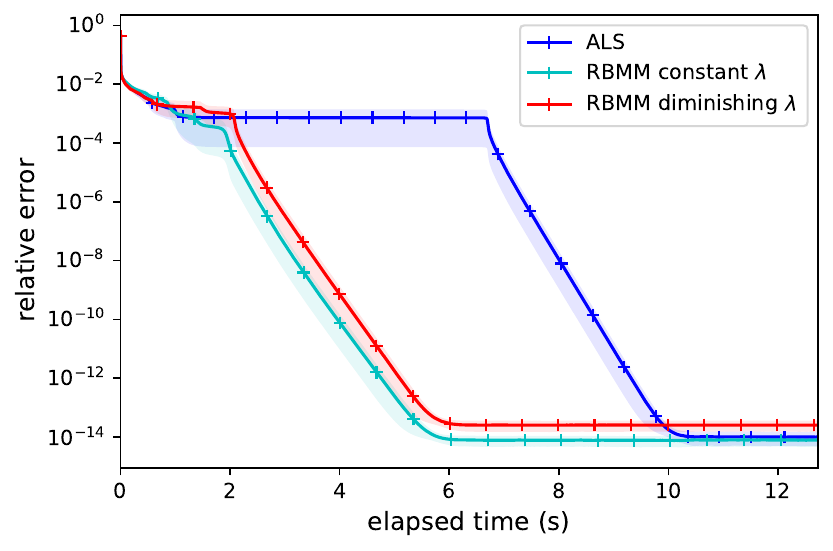}
				\caption{Low-rank: Synthetic data 1}
				\label{fig:subfig6}
			\end{subfigure}
			\caption{(A) and (B) are some typical cases of synthetic data in Euclidean case; (C) is the typical result when the first block is Stiefel manifold; (D) is a synthetic example where the first block is a point on low-rank manifold. The average relative reconstruction error with 
			standard deviations 
    are shown by the solid lines and shaded regions of respective colors.}
			\label{fig:CP}
		\end{figure*}
		
		In the CANDECOMP/PARAFAC (CP) decomposition problem \cite{kolda2009tensor}, given a data tensor $X\in \R^{I_1 \times\dots\times I_m}$ and an integer $R>0$, we would like to find the \textit{loading matrices} $U^{(i)}\in\R^{I_i \times R}$ for $i=1,\cdots, m$ such that 
		\begin{equation}
			X \approx \sum_{k=1}^R \bigotimes_{i=1}^m U^{(i)}[:, k],
		\end{equation}
		where $U^{(i)}[:, k]$ denotes the $k^{\text {th }}$ column of the $I_i \times R$ loading matrix matrix $U^{(i)}$ and $\otimes$ denotes the outer product. We could formulate the above tensor decomposition problem as the following optimization problem: 
		\begin{equation}
			\label{eq:cp_org}
			\underset{U^{(1)} \in \Theta^{(1)}, \ldots, U^{(m)} \in \Theta^{(m)}}{\operatorname{argmin}}\left(f\left(U^{(1)}, \ldots, U^{(m)}\right):=\left\|X-\sum_{k=1}^R \bigotimes_{i=1}^m U^{(i)}[:, k]\right\|_F^2\right),
		\end{equation}
		where $\Theta^{(i)} \subseteq \R^{I_i\times R}$ is an embedded manifold, which gives  Riemannian constraints on the factor matrices. This setup of Riemmanian CP-dictionary learning is related to the recent work \cite{dong2022new}, where the authors used a CP-decomposition with  Riemannian structure on the space of factor matrices as a pre-conditioning algorithm for tensor completion.
		
		It is easy to see \eqref{eq:cp_org} is equivalent to 
		\begin{equation}
			\label{eq:cp_dl}
			\underset{U^{(1)} \in \Theta^{(1)}, \ldots, U^{(m)} \in \Theta^{(m)}}{\operatorname{argmin}}\left\|X-\operatorname{Out}\left(U^{(1)}, \ldots, U^{(m-1)}\right) \times_m\left(U^{(m)}\right)^T\right\|_F^2,
		\end{equation}
		which is the CP-dictionary-learning problem in \cite{lyu2020online_tensor}. Here $\times_m$ denotes the mode-$m$ product (see \cite{kolda2009tensor}) and the outer product is given by
		\begin{equation}
			\operatorname{Out}\left(U^{(1)}, \ldots, U^{(m)}\right):=\left[\bigotimes_{k=1}^m U^{(k)}[:, 1], \bigotimes_{k=1}^m U^{(k)}[:, 2], \ldots, \bigotimes_{k=1}^m U^{(k)}[:, R]\right] \in \mathbb{R}^{I_1 \times \cdots \times I_m \times R} .
		\end{equation}
		The constrained CP-dictionary-learning problem \eqref{eq:cp_dl} falls under the framework of block minimization. In fact, let 
		\begin{align}
			\mathbf{A} &= \operatorname{Out}\left(U_{n-1}^{(1)}, \ldots, U_{n-1}^{(i-1)}, U_{n-1}^{(i+1)}, \ldots, U_{n-1}^{(m-1)}\right)^{(m)} \in \mathbb{R}^{\left(I_1 \times \cdots \times I_{i-1} \times I_{i+1} \times \cdots \times I_m\right) \times R} \\
			B &= \operatorname{unfold}(\mathbf{A}, m) \in \mathbb{R}^{\left(I_1 \cdots I_{i-1} I_{i+1} \cdots I_m\right) \times R},
		\end{align}
		where unfold(·, $i$ ) denotes the mode- $i$ tensor unfolding (see \cite{kolda2009tensor}). Then the marginal objective function can be rewritten as 
		\begin{equation}
			\label{eq:cp_marginal}
			f_n^{(i)}(U^{(i)})= \left\|\operatorname{unfold}(X, i)-B\left(U^{(i)}\right)^T\right\|_F^2 .
		\end{equation}
		Alternating block minimization of \eqref{eq:cp_org}, i.e. directly minimizing \eqref{eq:cp_marginal} in each iteration, is known as \textit{alternating least squares} (ALS), which is studied in \cite{kolda2009tensor,navasca2008swamp,hong2015unified,razaviyayn2013unified} under Euclidean settings, i.e. each $\Theta$ is a convex subset of $\R^{I_i\times R}$. Recently, ALS with diminishing radius is also studied in \cite{lyu2023block}.

        Besides the standard Euclidean setting, we also study the more interesting setting when some blocks are Riemannian manifolds, e.g. Stiefel manifolds or rank-$R$ manifolds. In order to solve the corresponding Riemannian CP-dictionary learning problem \eqref{eq:cp_org}, instead of directly minimizing $f_n^{(i)}$, we may apply RBMM, which cyclically minimizes a surrogate function $g_n^{(i)}$ in each iteration $n$ given by 
		\begin{equation}
			\label{eq:cp_surr}
			g_n^{(i)}(U^{(i)}):=f_n^{(i)}(U^{(i)})+\lambda_n \|U^{(i)}-U^{(i)}_{n-1}\|_F^2= \left\|\operatorname{unfold}(X, i)-B\left(U^{(i)}\right)^T\right\|_F^2+\lambda_n \|U^{(i)}-U^{(i)}_{n-1}\|_F^2.
		\end{equation}

        As aforementioned at the beginning of Section \ref{sec:apps},  the Euclidean distance function is not $g$-smooth with respect to the low-rank manifold. Hence when at least one $\Theta^{(i)}$ is a low-rank manifold, then in order to ensure $g$-convexity of the objective function, we will need to take the underlying manifold to be the Euclidean space, i.e., $\mathcal{M}^{(i)}=\R^{I_{i}\times R}$, and allow the constraint set $\Theta^{(i)}$ to be an embedded submanifold of $\R^{I_{i}\times R}$. Note that $\Theta^{(i)}$ is not geodesically convex in $\R^{I_{i}\times R}$ with respect to the Euclidean geometry in $\R^{I_{i}\times R}$, so we cannot apply our asymptotic complexity results (Theorems \ref{thm:BMM_prox} and \ref{thm:BMM_rate}). However, we can still apply Theorem \ref{thm:RBMM_2} (see also Remark \ref{rmk:constraints}). When all the 
 $\Theta^{(i)}$ are either a Stiefel manifold or the Euclidean space, then we can apply the results in Corollary \ref{cor:prox_Stiefel} and therefore have the complexity of $\widetilde{O}(\eps^{-2})$ along with the asymptotic convergence result.

	\begin{corollary}[Complexity of block Euclidean proximal updates for CP-dictionary learning]\label{cor:EBMM_Riemannian}
	Let $(U_{n}^{(1)},\dots,U_{n}^{(m)})_{n\ge 0}$  be a sequence of factor matrices obtained by RBMM (Alg. \ref{algorithm:BMM}) with Euclidean proximal surrogates $g_{n}^{(i)}$ as in \eqref{eq:cp_surr}. If $\inf_{n}\lambda_{n}>0$, then the sequence of iterates converges to the set of stationary points of \eqref{eq:cp_org}. Moreover, if all the 
 $\Theta^{(i)}$ are either a Stiefel manifold or the Euclidean space, then the iteration complexity is $\widetilde{O}(\eps^{-2})$.
	\end{corollary}

    In the following numerical experiments, we present the advantage of using RBMM in both the Euclidean setting and the Riemannian setting. We generate synthetic data $X\in \R^{30\times20\times10}$, then apply ALS (\cite{kolda2009tensor}), RBMM with constant $\lambda_n=0.1$ and RBMM with diminishing $\lambda_n=0.1\times 0.5^{n}$ as suggested in \cite{navasca2008swamp} to find
		loading matrices $U^{(1)},U^{(2)},U^{(3)}$ with $R = 3$ columns. We run each algorithm 100 times from the independent random initialization and then compute the averaged relative error. The typical results of the Euclidean setting, i.e. $\mathcal{M}^{(i)}=\R^{I_i \times R}$ for $i=1,2,3$ are shown in Figure \ref{fig:CP} (A) and (B), where (A) is the case when two RBMM algorithms are significantly better than ALS. When RBMM reaches a relative error of order $10^{-14}$, the relative error of ALS is still of order $10^{-3}$; (B) shows RBMM with constant $\lambda$ is better than ALS while RBMM with diminishing $\lambda$ is much better than the other two. After 0.25s elapsed time, RBMM with diminishing $\lambda$ reaches a relative error of order $10^{-14}$, while that of the other two are still of order $10^{-2}$. After 0.35s elapsed time, RBMM with constant $\lambda$ reaches error of order $10^{-8}$ while that of ALS remains $10^{-2}$.

		Figure \ref{fig:CP} (C) shows the result when  $U^{(1)}$ is a point on the Stiefel manifold $ \mathcal{V}^{I_1 \times R}$ (see Example \ref{eg:stiefel}), i.e. we let $\Theta^{(1)}=\mathcal{V}^{I_1 \times R}$, $\mathcal{M}^{(2)}=\R^{I_2\times R}$ and $\mathcal{M}^{(3)}=\R^{I_3\times R}$. In contrast to the Euclidean case, all three algorithms demonstrate good performance for random synthetic data with random initialization. One possible explanation of this phenomenon is that since $U^{(1)}_n \in \mathcal{V}^{I_1 \times R}$ for all $n$, we have $(U^{(1)}_n)^T U^{(1)}_n = I_{R}$. As a result, the quadratic terms added in RBMM would reduce to a linear term that does not significantly accelerate convergence. A similar observation is made when dealing with Stiefel manifolds in the geodesic subspace tracking problem (see Section \ref{app:GST} for details).
		
		Figure \ref{fig:CP} (D) shows the convergence results for some synthetic data when $U^{(1)}$ is a point on fixed-rank matrices manifold $\mathcal{R}_r$ with $r<R$ (see Example \ref{eg:fixed-rank}) where $r=5$ and $R=10$, i.e. we let $\Theta^{(1)}=\mathcal{R}_r$, $\mathcal{M}^{(2)}=\R^{I_2\times R}$ and $\mathcal{M}^{(3)}=\R^{I_3\times R}$. There, RBMM with constant or diminishing $\lambda$ performs much better than ALS. RBMM with constant and diminishing $\lambda$ require only $5.16$s and $5.51$s of elapsed time respectively to reach a relative error of order $10^{-13}$ while that of ALS takes $9.66$ seconds.

		\subsection{Robust PCA}
  \label{sec:RPCA}
		Principal component analysis (PCA) is a popular technique for reducing the dimensionality of a data set, while preserving the maximum amount of information, which is done by seeking the best low-rank approximation of the high dimensional data set \cite{jolliffe1986principal}. Mathematically, consider the data matrix $M\in \R^{m\times n}$, PCA seeks the best $r$-rank approximation of $M$ where $r<\min\{m,n\}$, by solving $\min_{L}\{\|L-M\|, \operatorname{rank}(L)\leq r\}$, where $\|X\|$ denotes the spectral norm. However, it is well known that classical PCA would break down when entries of $M$ are grossly corrupted. The PCA problem with corrupted entries is called \textit{robust PCA} (RPCA) problem \cite{candes2009robust}. More specifically, it considers the matrix $M$ of the form $M=L_0 +S_0$ where $L_0$ is the ideal low-rank matrix and $S_0$ is a sparse matrix. This problem can be mathematically formulated as
		\begin{equation}
			\label{eqn:robust_pca} 
			\min_{L,S}\,\, \operatorname{rank}(L) + \lambda \|S\|_0, \quad \text{subject to} \quad M=L+S,
		\end{equation}
		where $\lambda>0$ is a trade-off parameter and $\|S\|_0$ is the number of none zero entries of $S$.

        Since the objective function in  \eqref{eqn:robust_pca} is the sum of two nonconvex functions, it is natural to consider a convex relaxation of \eqref{eqn:robust_pca}  by relaxing the rank function by the nuclear norm and the $\ell_{0}$-norm by the $\ell_{1}$-norm. This gives the following \textit{Principal Component Pursuit} (PCP) \cite{candes2009robust,ma2015generalized}:
        \begin{equation}
			\label{eqn:robust_pca_convex}
			\min_{L,S} \,\, \|L\|_{*} + \lambda \|S\|_1, \quad \text{subject to} \quad M=L+S
		\end{equation}
		where $\|X\|_*$ is the nuclear norm of $X$, defined to be the sum of the singular values. Under certain conditions on the unknown parameters $L_0$ and $S_0$, it is possible to recover $L_0$ and $S_0$ by solving the relaxed problem \eqref{eqn:robust_pca_convex}.  A more general and realistic setting is that the observations are noisy, i.e. $M=L+S+Z$ where $Z$ represents the noise term. This is the \textit{stable PCP} (SPCP) problem proposed in \cite{zhou2010stable}, where the authors showed the following optimization problem efficiently recovers the true $L$ and $S$ with properly chosen $\mu$,
\begin{equation}
			\label{eqn:RPCA_noise}
			\min_{L,S} \, \, \|L\|_{*} + \lambda \|S\|_1 + \frac{1}{2\mu}\|M-L-S\|_{F}^2
\end{equation}
  More studies on solving \eqref{eqn:RPCA_noise} can also be found in \cite{yin2019stable,aravkin2014variational}.

While \eqref{eqn:RPCA_noise} and \eqref{eqn:robust_pca_convex} are well-studied formulations of RPCA, they do not always guarantee a solution with a fixed rank since they promote the rank of $L$ to be small indirectly through the nuclear norm penalization. Rodriguez and Wohlberg \cite{rodriguez2013fast}  proposed the following  alternative formulation of RPCA by using a hard low-rank constraint on $L$ instead of the soft nuclear-norm penalization: 

\begin{equation}
\label{eqn:RPCA_low_rank}
\min_{L\in \mathcal{R}_r,\;S\in\R^{m\times n}} \lambda \|S\|_1 + \frac{1}{2\mu}\|M-L-S\|_{F}^2
\end{equation}
where $\mathcal{R}_r$ is the rank-$r$ matrix manifold (see Example \ref{eg:fixed-rank}). In \cite{rodriguez2013fast}, an alternating minimization algorithm for solving \eqref{eqn:RPCA_low_rank} was demonstrated to be effective through experiments, but no theoretical guarantee of such an algorithm was provided.

Here, we revisit the alternating minimization algorithm for the Riemannian formulation of RPCA  \eqref{eqn:RPCA_low_rank} in \cite{rodriguez2013fast} through our general framework of RBMM. In order to apply our framework, we consider a smooth approximation of \eqref{eqn:RPCA_low_rank}. Note the 1-norm $\|\cdot\|_1$ is in general not smooth. A 
smooth
approximation of $g(S)=\lambda\|S\|_1$ with smoothness parameter $\sigma$ is given by the following (see e.g. \cite{goldfarb2013fast})
		\begin{equation}
			\label{eqn:g_sigma}
			g_{\sigma}(S)=\max\left\{\langle S, Z\rangle-\frac{\sigma}{2}\|Z\|^{2}_{F} : \|Z\|_{\infty}\leq \lambda\right\}.
		\end{equation}
		It is easy to check the optimal solution of \eqref{eqn:g_sigma} is $Z_{\sigma}(S)$ with 
		\begin{equation}
			\label{eqn:Z_S}
			Z_{\sigma}(S)_{ij}=\min\left\{\lambda,\max\{S_{ij}/\sigma,-\lambda\}\right\}.
		\end{equation}
		Also, the gradient of $g_{\sigma}(S)$ is given by $\nabla g_{\sigma}(S)=Z_{\sigma}(S)$ and is Lipschitz continuous with parameter $1/\sigma$.

        Now consider the following problem, which is a smooth approximation of the Riemannian RPCA \eqref{eqn:RPCA_low_rank}:
		\begin{equation}
			\label{eqn:RPCA_smooth}
			\min_{L\in\mathcal{R}_r,\;S\in \R^{m\times n}}F(L,S):=g_{\sigma}(S) + \frac{1}{2\mu}\|M-L-S\|_{F}^2 .
		\end{equation}
    The alternating minimization algorithm for \eqref{eqn:RPCA_smooth} analogous to the one in \cite{rodriguez2013fast} for \eqref{eqn:RPCA_low_rank} reads as follows. For any $k=1,\cdots, T$, 
denoting the solution at iteration $k$ as $L_k$ and $S_k$, the next 
iterates are computed by 
		\begin{align}
			L_{k+1} &=\argmin_{L\in \mathcal{R}_r} F(L,S_k)=\argmin_L  \,\,  g_{\sigma}(S_k) + \frac{1}{2\mu}\|M-L-S_k\|_{F}^2, \label{eqn:L_k} \\
			S_{k+1} &=\argmin_{S \in \R^{m\times n}} F(L_{k+1},S)= \argmin_S \,\,  g_{\sigma}(S)  + \frac{1}{2\mu}\|M-L_{k+1}-S\|_{F}^2 . \label{eqn:S_k}
		\end{align}
		The first-order optimality condition of \eqref{eqn:L_k} gives $L_{k+1} = U \Sigma_r V^{T}$, where $U\Sigma V^T$ is the SVD of $M-S_k$ with the singular values ordered in nonincreasing order and $\Sigma_r$ keeps the first $r$ singular values (see Example \ref{eg:fixed-rank}). 
		
		Similarly, the first-order optimal solution for \eqref{eqn:S_k} is 
		\begin{equation}
			(S_{k+1})_{i j}=B_{i j}-\mu \min \left\{\rho, \max \left\{-\lambda, \frac{B_{i j}}{\sigma+\mu}\right\}\right\} \text { for } i=1, \ldots, m \text { and } j=1, \ldots, n
		\end{equation}
		where $B=M-L_{k+1}$.

        Note that the alternating minimization algorithm in \cite{rodriguez2013fast} is a special instance of our RBMM with zero majorization gap. Hence can generalize it by using proximal regularization as below:
		\begin{align}\label{eq:BMM_RPCA}
			L_{k+1} =\argmin_{L\in \mathcal{R}_r}G^{(1)}_{k+1}(L):&= F(L,S_k) + \frac{\lambda_{k+1}}{2}\|L-L_k\|^{2}_F\\
            &= g_{\sigma}(S_k) + \frac{1}{2\mu}\|M-L-S_k\|_{F}^2 + \frac{\lambda_{k+1}}{2}\|L-L_k\|^{2}_F, \\
			S_{k+1} =\argmin_{S \in \R^{m\times n}}G^{(2)}_{k+1}(S):&= F(L_{k+1},S) + \frac{\lambda_{k+1}}{2}\|S-S_k\|^{2}_F &\\
            &=  g_{\sigma}(S) + \frac{1}{2\mu}\|M-L_{k+1}-S\|_{F}^2+ \frac{\lambda_{k+1}}{2}\|S-S_k\|^{2}_F.
		\end{align}
		The minimization of $G^{(1)}_{k+1}$ and $G^{(2)}_{k+1}$  admits similar closed form  
        solutions as of \eqref{eqn:L_k} and \eqref{eqn:S_k}, which can be found in Appendix \ref{appendix:RPCA}. %
        
        We remark again that these iterative updates fall in the framework of RBMM. Namely, since $\M^{(1)}=\M^{(2)}=\Theta^{(2)}=\R^{m\times n}$ and $\Theta^{(1)}=\mathcal{R}_r$, the assumptions \ref{assumption:A0_optimal_gap} and \ref{assumption:A1_smoothness_surrogates} are satisfied trivially. Therefore, we deduce the following corollary of Theorem \ref{thm:RBMM_2}:
  \begin{corollary}[Convergence of BMM for PCP]\label{cor:RPCA}
      Fix a matrix $M\in \R^{m \times n}$. Let $\param_{k}:=(L_{k}, S_{k})$ be generated by \eqref{eq:BMM_RPCA} with arbitrary initialization $\param_{0}\in \Param:=\mathcal{R}_r \times \R^{m\times n}$. Suppose the proximal regularization parameters $\lambda_{k}$ are strictly positive for all $k\ge 0$. For the smoothed PCP problem \eqref{eqn:RPCA_smooth}, the updates in \eqref{eq:BMM_RPCA} converge to the set of stationary points.
  \end{corollary}

We remark that in \cite{rodriguez2013fast}, the authors numerically verified the efficiency of the alternating minimization for solving \eqref{eqn:RPCA_low_rank}, but no theoretical results on asymptotic convergence or iteration complexity are established. Corollary \ref{cor:RPCA}  provides the asymptotic convergence of the algorithm \eqref{eq:BMM_RPCA} by using the general RBMM framework.

	\section{Convergence analysis}\label{sec:conv_opt1}
	In this section, we prove the convergence and complexity of RBMM (Algorithm \ref{algorithm:BMM}) stated in Theorems \ref{thm:RBMM_1},  \ref{thm:RBMM_2}, \ref{thm:BMM_prox},  and \ref{thm:BMM_rate}. Throughout this section, we let $(\param_{n})_{n\ge 1}$ denote the sequence of iterates in $\Param\subseteq \prod_{i=1}^{m} \M^{(i)}$ generated by Algorithm \ref{algorithm:BMM}. We will also use the notations introduced in Section \ref{sec:notations}. Note in this section, we use the exponential map as the retraction in the definition of lifted constraint set \eqref{eq:def_lift_constraints}, i.e.
 \begin{align}
	T^{*}_{x} \mathcal{M} &:= \{ u\in T_{x} \mathcal{M} \,|\,  \textup{$\Exp_x(u)=x'$ for some $x'\in \Param$ with $d(x,x')\le r_0 /2$} \}.
\end{align}
	It is worth mentioning that though in the analysis, we will be using an exponential map, which is computationally expensive, in practice access to it is not required, as shown in Algorithm \ref{algorithm:BMM}.

	\subsection{Preliminary analysis}

	In this section, we present some preliminary results that will be used in the analysis of RBMM with both $g$-smooth and proximal surrogates.

	We first establish basic monotonicity properties of the iterates $(\param_{n})_{n\ge 1}$. 
	
	\begin{prop}[Monotonicity of objective and Stability of iterates]\label{prop:forward_monotonicity_proximal}
		Suppose \ref{assumption:A0_optimal_gap}, except for the $g$-smoothness of $f$. 
  Then the following hold: 
		\begin{description}[itemsep=0.1cm]
			\item[(i)] $f(\param_{n-1}) - f(\param_{n}) \ge   \sum_{i=1}^{m}\left( g^{(i)}_{n}(\theta^{(i)}_{n}) -  f^{(i)}_{n}(\theta^{(i)}_{n}) -\Delta_n \right) \ge -m\Delta_n$;
			\vspace{0.1cm}
			\item[(ii)] $ \sum_{n=1}^{\infty} \sum_{i=1}^{m}\left( g^{(i)}_{n}(\theta^{(i)}_{n}) -  f^{(i)}_{n}(\theta^{(i)}_{n}) \right) <  f(\param_{0}) - f^{*} +m\sum_{n=1}^{\infty}\Delta_n<\infty$.
		\end{description}
		
		\begin{description}[itemsep=0.1cm]
			\item[(iii)] Further assume \ref{assumption:A1_1}\textbf{(iii)} holds. Then the following also holds:
			\begin{align}
				\sum_{n=1}^{\infty}\sum_{i=1}^{m} \phi\left( d\left(\theta_{n-1}^{(i)},\theta_{n}^{(i)}\right) \right) &\le  \sum_{n=1}^{\infty} \sum_{i=1}^{m}\left( g^{(i)}_{n}(\theta^{(i)}_{n}) -  f^{(i)}_{n}(\theta^{(i)}_{n}) \right) <  f(\param_{0}) - f^{*} +m\sum_{n=1}^{\infty}\Delta_n<\infty.
			\end{align}
			In particular, $d(\theta_{n-1}^{(i)}, \theta_{n}^{(i)})=o(1)$ for all $i=1,\dots,m$.
		\end{description}
		
	\end{prop}
	
	\begin{proof}
		Fix $i\in \{1,\dots,m\}$. Since $\theta_{n}^{(i\star)}$ is a minimizer of $g_{n}^{(i)}$ over $\Theta^{(i)}$ (see \ref{assumption:A0_optimal_gap}\textbf{(ii)}), we get  $g^{(i)}_{n}(\theta^{(i\star)}_{n}) \le g^{(i)}_{n}(\theta^{(i)}_{n-1})=f^{(i)}_{n}(\theta^{(i)}_{n-1})$, for $n\ge 1$. Hence we deduce 
		\begin{align}\label{eq:monotonicity_g_comparison}
			f^{(i)}_{n}(\theta^{(i)}_{n-1}) - f^{(i)}_{n}(\theta^{(i)}_{n}) &=   g^{(i)}_{n}(\theta^{(i)}_{n-1}) - g^{(i)}_{n}(\theta^{(i\star)}_{n}) +g^{(i)}_{n}(\theta^{(i\star)}_{n})- g^{(i)}_{n}(\theta^{(i)}_{n})   +g^{(i)}_{n}(\theta^{(i)}_{n}) -  f^{(i)}_{n}(\theta^{(i)}_{n})\\
			& \ge   -\Delta_n+ g^{(i)}_{n}(\theta^{(i)}_{n}) -  f^{(i)}_{n}(\theta^{(i)}_{n})\ge -\Delta_n .
		\end{align}
		It follows that 
		\begin{align}
			&f(\param_{n-1}) - f(\param_{n}) \\
			&\qquad = \sum_{i=1}^{m} f(\theta_{n}^{(1)},\dots, \theta_{n}^{(i-1)},\theta_{n-1}^{(i)},\theta_{n-1}^{(i+1)}, \dots, \theta_{n-1}^{(m)}) - f(\theta_{n}^{(1)},\dots, \theta_{n}^{(i-1)},\theta_{n}^{(i)},\theta_{n-1}^{(i+1)}, \dots, \theta_{n-1}^{(m)}) \\ 
			&  \qquad = \sum_{i=1}^{m} f_{n}^{(i)}(\theta^{(i)}_{n-1}) -  f_{n}^{(i)}(\theta^{(i)}_{n}) \\
			& \qquad \geq \sum_{i=1}^{m}g^{(i)}_{n}(\theta^{(i)}_{n}) -  f^{(i)}_{n}(\theta^{(i)}_{n})-\Delta_n\\
			&\qquad \ge -m\Delta_n.
			\quad  \label{eq:proximal_monotonicity}
		\end{align}
		This shows \textbf{(i)}. 
		
		Next, to show \textbf{(ii)}, adding up the above inequality for $n\ge 1$,
		\begin{align}
			\sum_{n=1}^{\infty} \sum_{i=1}^{m}\left( g^{(i)}_{n}(\theta^{(i)}_{n}) -  f^{(i)}_{n}(\theta^{(i)}_{n}) \right) & \le \left( \sum_{n=1}^{\infty} f(\param_{n-1}) - f(\param_{n}) \right)  +m\sum_{n=1}^{\infty}\Delta_n \le  f(\param_{0}) - f^{*}+m\sum_{n=1}^{\infty}\Delta_n<\infty.
		\end{align} 
		To show \textbf{(iii)}, note that by \ref{assumption:A1_1}\textbf{(iii)},
		\begin{equation}
			g^{(i)}_{n}(\theta^{(i)}_{n}) -  f^{(i)}_{n}(\theta^{(i)}_{n})\ge \phi\left( d\left(\theta_{n-1}^{(i)},\theta_{n}^{(i)}\right) \right).
		\end{equation}
		Together with \textbf{(ii)} we get \textbf{(iii)}.
	\end{proof}
	
	The following is an immediate consequence of Proposition \ref{prop:forward_monotonicity_proximal} and sub-level compactness.

	\begin{prop}[Boundedness of iterates]
		\label{prop:boundedness_iterates}
		Under \ref{assumption:A0_optimal_gap} except for the $g$-smoothness of $f$, 
  the set $\{\param_n : n\geq1\}$ is bounded.
	\end{prop}
	
	\begin{proof}
		Let $m\sum_{n=1}^{\infty}\Delta_n =T<\infty$, by Prop. \ref{prop:forward_monotonicity_proximal} we immediately have $f(\param_n)\leq f(\param_0)+m\sum_{k=1}^{n}\Delta_k< f(\param_0)+T$. Consider $K=\{\param \in \Param: f(\param)\leq f(\param_0)+T\}$, by \ref{assumption:A0_optimal_gap}\textbf{(i)}, $K$ is compact. Hence the set $\{\param_n : n\geq1\}$ is bounded.
	\end{proof}

	The following two propositions to be introduced will only be used in deriving the rate of convergence results, and they assume the geodesic convexity of the constraint sets. 
 
    If $\M$ is a Riemannian manifold and if $x,y\in \M$ such that $d(x,y)< \rinj(x)$, then the tangent vector $\eta_{x}(y):=\Exp_{x}(y)\in T_{x}\M$ is uniquely defined. Now we introduce the following notations: 
    \begin{align}\label{eq:def_eta_n_n}
    \eta_{n+1}^{(i)}(n):=\eta_{\theta_{n+1}^{(i)}}(\theta_{n}^{(i)}),\quad   \eta_{n}^{(i)}(n+1):=\eta_{\theta_{n}^{(i)}}(\theta_{n+1}^{(i)}). 
    \end{align}
    Note that the above tangent vectors are well-defined for all $n$ sufficiently large under \ref{assumption:A0_optimal_gap} and \ref{assumption:A1_smoothness_surrogates}\textbf{(i)-(iii)}. Indeed, by Proposition \ref{prop:forward_monotonicity_proximal} we have $d(\param_{n},\param_{n+1})=o(1)$. Also, the injectivity radius is uniformly lower bounded by some constant $r_{0}>0$ (see \ref{assumption:A1_smoothness_surrogates}\textbf{(ii)}). Hence 
     \begin{align}\label{eq:iterate_stability}
            d(\param_{n},\param_{n+1}) < r_{0} \quad \textup{for all $n\ge N_{0}$ for some $N_{0}\in \N$},
     \end{align}
    so the tangent vectors in \eqref{eq:def_eta_n_n}  are well-defined for $n\ge N_{0}$.

    First, we show that the first-order variation of the objective function along the trajectory of iterates $\param_{n}$ is summable. 
	
	\begin{prop}[Finite first-order variation]\label{prop:finite_first-order_variation}
		Suppose \ref{assumption:A0_optimal_gap}, \ref{assumption:A1_smoothness_surrogates}\textbf{(i)-(ii)}, and \ref{assumption:A1_1}\textbf{(iii)} with  $\phi(x)=x^2$ hold. Further assume that the constraint set $\Theta^{(i)}$ is strongly convex in $\M^{(i)}$ for $i=1,\dots,m$. Then
		\begin{equation}
			\sum_{n=N_{0}}^{\infty} \sum_{i=1}^{m}  \left\lvert \left\langle\grad  f_{n}^{(i)}(\theta_{n+1}^{(i)}), \eta_{n+1}^{(i)}(n) \right\rangle \right\rvert < \frac{L_f+2}{2}\left(f(\param_{0})-f^{*}\right)+3m\sum_{n=0}^{\infty}\Delta_{n+1}<\infty.
		\end{equation}
		
	\end{prop}
	\begin{proof}
  By Lemma \ref{lem:g_smooth_linear_approx},
		\begin{equation}
			\left\lvert f_{n}^{(i)}(\theta_{n}^{(i)})-f_{n}^{(i)}(\theta_{n+1}^{(i)})-\left\langle\grad  f_{n}^{(i)}(\theta_{n+1}^{(i)}), \eta^{(i)}_{n+1}(n)  \right\rangle\right\rvert \le \frac{L_f}{2} d^{2}(\theta_{n}^{(i)}, \theta_{n+1}^{(i)}). 
		\end{equation}
		Then by Proposition  \ref{prop:forward_monotonicity_proximal},
		\begin{align}
			\left\lvert \left\langle\grad  f_{n}^{(i)}(\theta_{n+1}^{(i)}),  \eta^{(i)}_{n+1}(n) \right\rangle \right\rvert&\le \frac{L_f}{2} d^{2}(\theta_{n}^{(i)}, \theta_{n+1}^{(i)})+\left\lvert f_{n}^{(i)}(\theta_{n}^{(i)})-f_{n}^{(i)}(\theta_{n+1}^{(i)}) \right\rvert \\
			&\le \frac{L_f}{2} d^{2}(\theta_{n}^{(i)}, \theta_{n+1}^{(i)})+ f_{n}^{(i)}(\theta_{n}^{(i)})-f_{n}^{(i)}(\theta_{n+1}^{(i)}) + 2\Delta_{n+1}.
		\end{align}
		Summing up for all $i=1,\dots,m$ and then for $n\ge 1$, we get
		\begin{align}
			&\sum_{n=N_0}^{\infty} \sum_{i=1}^{m}  \left\lvert \left\langle\grad  f_{n}^{(i)}(\theta_{n+1}^{(i)}), \eta^{(i)}_{n+1}(n)  \right\rangle \right\rvert \\
			&\qquad \le \frac{L_f}{2}\sum_{n=N_0}^{\infty}d^2(\param_{n},\param_{n+1})+ \sum_{n=N_0}^{\infty}\left(f(\param_{n})-f(\param_{n+1})\right)+2m\sum_{n=N_0}^{\infty}\Delta_{n+1} \\
			&\qquad \le \frac{L_f}{2}\sum_{n=0}^{\infty}\phi(d(\param_{n}, \param_{n+1}))+f(\param_{0})-f^{*}+2m\sum_{n=0}^{\infty}\Delta_{n+1}\\
			&\qquad \le \frac{L_f+2}{2}\left(f(\param_{0})-f^{*}\right)+3m\sum_{n=0}^{\infty}\Delta_{n+1}
		\end{align}
		where the last inequality comes from Proposition \ref{prop:forward_monotonicity_proximal}.
		
	\end{proof}
	
	Next, we relate the first-order approximation of the objective difference $f(\param_{n+1})-f(\param_{n})$ with that of the worst-case difference $f(\param)-f(\param_{n})$. This will be used crucially in the proof of Theorem \ref{thm:BMM_prox} and \ref{thm:BMM_rate}. 
	
	\begin{prop}[Asymptotic first-order optimality]\label{prop:asymptotic_first-order_optimality}
		Suppose \ref{assumption:A0_optimal_gap} and \ref{assumption:A1_smoothness_surrogates}\textbf{(i)-(iii)} hold. Further assume that the constraint set $\Theta^{(i)}$ is strongly convex in $\M^{(i)}$ for $i=1,\dots,m$. Fix a sequence $(b_{n})_{n\ge 1}$ of positive reals with $b_n \le \hat{r}:=\min\{r_0,1\}$ for all $n$.  Then the following hold for all $n\ge N_{0}$:
		\begin{description}
			\item[(i)] ($g$-smooth surrogates) Suppose \ref{assumption:A1_smoothness_surrogates}\textbf{\textup{(i-gs)}} holds.  Then there exists constant $c_{1}, c_{2}>0$ independent of $\param_{0}$ such that
			\begin{align}
				& \sum_{i=1}^{m} \left\langle  \grad f_{n+1}^{(i)}(\theta_{n}^{(i)}), \eta^{(i)}_n (n+1) \right\rangle \\
				&\qquad \leq  \hat{r}^{-1}b_{n+1} \sum_{i=1}^{m}\inf _{ \eta^{(i)}\in  T^{*}_{\theta_{n}^{(i)}} ,\|\eta^{(i)}\|\le 1}\left\langle \grad g_{n+1}^{(i)}(\theta_{n}^{(i)}) + \grad_i f(\param_n)- \grad f_{n+1}^{(i)}(\theta^{(i)}_{n}),\eta^{(i)}\right\rangle  \label{eq:first_order_optimality_bd0} \\
				&\qquad \qquad + c_1 b_{n+1}^{2} + c_2 d^{2}(\param_{n}, \param_{n+1})+m\Delta_n. 
			\end{align}
		
		\item[(ii)] (Riemannian proximal surrogates) Suppose \ref{assumption:A1_smoothness_surrogates}\textbf{\textup{(i-rp)}} holds. Then there exists constant $c_{1}, c_{2}>0$ independent of $\param_{0}$ such that	
		\begin{align}\label{eq:first_order_optimality_bd1}
		\sum_{i=1}^{m} \left\langle  \grad f_{n+1}^{(i)}(\theta_{n}^{(i)}), \eta^{(i)}_n (n+1) \right\rangle &\le \hat{r}^{-1}b_{n+1} \sum_{i=1}^{m}\inf _{\eta^{(i)}\in  T^{*}_{\theta_{n}^{(i)}} ,\|\eta^{(i)}\|\le 1}\left\langle  \grad f_{n+1}^{(i)}(\theta_{n}^{(i)}) ,\eta^{(i)}\right\rangle  \\
			&\qquad  +c_1 b_{n+1}^{2} + c_2 d^{2}(\param_{n}, \param_{n+1})+m\Delta_n.
		\end{align}
		\end{description}
		\end{prop}
		
\begin{proof}
			We first show \textbf{(ii)}. Let $\hat{r}=\min\{r_0,1\}$, where $r_0$ is the lower bound of injectivity radius (see \ref{assumption:A1_smoothness_surrogates} \textbf{(ii)}). Fix arbitrary $\param=[\theta^{(1)},\cdots,\theta^{(m)}]\in \Param$ such that \begin{align}
				d(\theta^{(i)},\theta_{n}^{(i)})\le  b_{n+1} \le \hat{r}\le \rinj(\theta^{(i)})
			\end{align}
			 for all $i$. For each $i$, let $\gamma^{(i)}:[0,1]\rightarrow \M^{(i)}$ be the unique distance-minimizing geodesic from $\theta_{n}^{(i)}$ to $\theta^{(i)}$. 
			 Denote $\eta^{(i)}:=(\gamma^{(i)})'(0)$. Then $\eta^{(i)}\in T^{*}_{\theta_{n}^{(i)}}$.  
    Denote $h_{n}^{(i)}:=g_{n}^{(i)}-f_{n}^{(i)}$. 
			
			Fix $i=1,\dots,m$. We will drop the superscripts $(i)$ indicating block $i$ in the proof until we mention otherwise. 
			As in \ref{assumption:A0_optimal_gap}\textbf{(ii)}, let  $\theta_{n+1}^{\star}$ be an exact minimizer of $g_{n+1}$ over $\Theta$. Then
			\begin{align}
				& \langle \grad f_{n+1}(\theta_{n}),\eta_{n}(n+1)  \rangle-\frac{L_f}{2}d^{2}(\theta_{n},\theta_{n+1}) \\
				&\qquad  \overset{(a)}{\leq} f_{n+1}(\theta_{n+1})-f_{n+1}(\theta_{n}) \\
				&\qquad  \overset{(b)}{\leq} g_{n+1}(\theta_{n+1})-g_{n+1}(\theta_{n}) \\
				&\qquad  \overset{(c)}{\leq} g_{n+1}(\theta^{\star}_{n+1})+\Delta_n -g_{n+1}(\theta_{n}) \label{eq:g_n_difference_pf} \\
				&\qquad  \overset{(d)}{\leq} f_{n+1}(\gamma(1))-f_{n+1}(\theta_{n})+\Delta_n + h_{n+1}(\gamma(1))   \\
				&\qquad  \overset{(e)}{\leq}
				\langle \grad f_{n+1}(\theta_{n}), \eta \rangle+\frac{L_f}{2}d^{2}(\theta_{n},\theta)+\Delta_n + h_{n+1}(\gamma(1))
				. 
			\end{align}
			Here, $(a)$ follows from geodesic $L_{f}$-smoothness of $f$ (see \ref{assumption:A0_optimal_gap}\textbf{(i)}), $(b)$ follows from definition of majorizing surrogates (Def. \ref{def:majorizing_surrogates}), $(c)$ follows from the definition of $\theta_{n}^{\star}$  and the optimality gap $\Delta_{n}$ (see \ref{assumption:A0_optimal_gap}\textbf{(ii)}), $(d)$ follows from the fact that $g_{n+1}(\theta_{n})=f_{n+1}(\theta_{n})$ and  $g_{n+1}(\theta_{n+1}^{\star}) \le g_{n+1}(\gamma(t))$ for all $t\in [0,1]$, $(e)$ uses geodesic $L_{f}$-smoothness of $f$. Also,  
			\begin{align}\label{eq:pf_majorization_gap}
				h_{n+1}(\gamma(1)) = \frac{\lambda_{n}}{2} d^{2}(\theta, \theta_{n}) \le \frac{\lambda_{n}}{2} b_{n+1}^{2}. 
			\end{align}
			This gives 
			\begin{align}\label{eq:lemma_rate_optimality_pf1}
			 \left\langle \grad f_{n+1}(\theta_{n}), \, \eta_{n}(n+1) \right\rangle 
				&  \le  \left\langle \grad f_{n+1}(\theta_{n}), \, \eta \right\rangle    
				+c_{1} b_{n+1}^2+ c_{2} d^{2}(\theta_{n}, \theta_{n+1}) + \Delta_n ,
			\end{align}
			where $c_{1}:=\frac{L_f+ \lambda_n}{2}$ and $c_{2}=\frac{L_f}{2}$.
   
The above holds for all $\theta\in \Theta$ with $d(\theta,\theta_n)\le b_{n+1}$. Hence 
\begin{align}\label{eq:lemma_rate_optimality_pf2}
			 \left\langle \grad f_{n+1}(\theta_{n}), \, \eta_{n}(n+1) \right\rangle 
				&  \le  \inf_{\|\eta\|\le b_{n+1},\eta\in T^{*}_{\theta_n}}\left\langle \grad f_{n+1}(\theta_{n}), \, \eta \right\rangle    
				+c_{1} b_{n+1}^2+ c_{2} d^{2}(\theta_{n}, \theta_{n+1}) + \Delta_n .
			\end{align}
Now observe by g-convexity of $\Theta$, if $\alpha\eta \in T^*_{\theta_n}$ for $\alpha\ge 1$, then $\eta\in  T^*_{\theta_n}$. Therefore we get
\begin{align}
    \sup_{\|\eta\|\le \hat{r},\eta\in T^{*}_{\theta_n}}\left\langle -\grad f_{n+1}(\theta_{n}), \, \eta \right\rangle &= \sup_{\|u\|\le b_{n+1},\hat{r} b^{-1}_{n+1}u\in T^{*}_{\theta_n}}\left\langle -\grad f_{n+1}(\theta_{n}), \, \hat{r} b^{-1}_{n+1}u \right\rangle \\
    & \le \hat{r} b^{-1}_{n+1}\sup_{\|u\|\le b_{n+1},u\in T^{*}_{\theta_n}}\left\langle -\grad f_{n+1}(\theta_{n}), \, u \right\rangle
\end{align}
   The above together with \eqref{eq:lemma_rate_optimality_pf2} show \textbf{(ii)}.

   Lastly, we note that \textbf{(i)}  be shown by a similar argument. The main difference is that we do not have an explicit expression for the majorization gap $h_{n+1}$ as in \eqref{eq:pf_majorization_gap} but instead, we have $g$-smoothness of the surrogates. Hence we can proceed by using the $g$-smoothness of $g_{n+1}^{(i)}$ in line \eqref{eq:g_n_difference_pf} instead of decomposing $g_{n+1}^{(i)}=f_{n+1}^{(i)} + h_{n+1}^{(i)}$ and using the $g$-smoothness of $f_{n+1}^{(i)}$. The rest of the analysis is identical.
		\end{proof}

		\subsection{RBMM with Riemannian proximal surrogates }\label{sec:had}

		In this section, We analyze our RBMM algorithm with Riemannian proximal surrogates. Under \ref{assumption:A1_smoothness_surrogates}\textbf{(i-rp)}, 
		Algorithm \ref{algorithm:BMM} is the Riemannian analog of the Euclidean 
 block majorization-minimization with Euclidean proximal regularizer. 
  
        Before we present the proof of our general results in Theorem \ref{thm:RBMM_2} and \ref{thm:BMM_prox}, we will first delve into the special case of block Riemannian proximal methods on Hadamard manifolds, as briefly mentioned in Section \ref{sec:eg_had}. Examining this special case will be helpful for understanding RBMM with Riemannian proximal surrogates, particularly in two key aspects. First, the Riemannian proximal surrogate $g_{n}^{(i)}$ is geodesically strongly convex on Hadamard manifolds, which makes the subproblem for the minimization step in \eqref{eq:RBMM_MmMm} easy to solve. Second, $d^{2}(\theta,\theta^{(i)}_{n-1})$ is not $g$-smooth even on Hadamard manifolds, highlighting the necessity of analyzing the case of $g$-smooth surrogates and Riemannian/Euclidean proximal surrogates separately.

        We first recall some useful facts on Hadamard manifolds. First, recall that on Hadamard manifolds, the exponential map is a global diffeomorphism, and therefore its injectivity radius is $+\infty$ \cite[Theorem 4.1, p.221]{sakai1996riemannian}. Second, $d^2 (\cdot,p):\mathcal{M}\to \R$ is geodesically $2$-strongly convex for fixed $p\in \mathcal{M}$ where $\mathcal{M}$ is a Hadamard manifold (\cite{bento2015proximal}). We give the formal definition of geodesically strongly convex below.

		\begin{definition}[Geodesic strong convexity]
			\label{def:g_strongly_convex}
			Suppose $\M$ is a Hadamard manifold. A function $f: \mathcal{M}\to \R$ is geodesically $\alpha$-strongly convex if it satisfies
			\begin{equation}\label{eqn:def_strongly_convex}
				f(y) - f(x)-\left\langle \grad f(x), \operatorname{Exp}_x^{-1}(y)\right\rangle_x \geq \frac{\alpha}{2} d^2(x, y)
			\end{equation}
			for all $x$, $y\in \mathcal{M}$, where $d(x,y)$ is the geodesic distance between $x$ and $y$ on $\M$.
		\end{definition}
		
		We can also show by definition, that the Riemannian proximal surrogate $g_{n}^{(i)}(\theta)$ is geodesically strongly convex with appropriate Riemannian proximal regularization parameter. This is formally stated in the following proposition.
		
		\begin{prop}[Geodesical strong convexity of surrogates]\label{prop:strong_convexity_surrogate}
			Let $\mathcal{M}$ be a Hadamard manifold and $f: \mathcal{M}\to \R$ is geodesically $L_f$-smooth (see Definition \ref{def: G-L-smooth}). Fix $p\in \mathcal{M}$, let $g: \mathcal{M}\to \R$
			\begin{equation}
				g(x)=f(x)+ \frac{c}{2}d^{2}(x,p)
			\end{equation}
			where $c>L_f$.
			Then $g$ is geodesically strongly convex with parameter $c-L_f$.
		\end{prop}
		\begin{proof}
			First by $g$-smoothness of $f$ using Lemma \ref{lem:g_smooth_linear_approx} and the geodesical 2-strong convexity of $d^{2}(\cdot, p)$, 
			\begin{align}
				f(y) - f(x)-\left\langle \grad f(x), \operatorname{Exp}_x^{-1}(y)\right\rangle_x &\geq -\frac{L_f}{2} d^2 (x, y) \\
				d^2(y,p) - d^2(x,p)-\left\langle \grad d^2(x,p), \operatorname{Exp}_x^{-1}(y)\right\rangle_x &\geq  d^2 (x, y)
			\end{align}
			for any $x$, $y\in \mathcal{M}.$ Finally multiplying the second inequality by $\frac{c}{2}$ and adding up with the first inequality finish the proof.
		\end{proof}
		Hence, a major benefit of using Riemannian proximal surrogates on Hadamard manifolds is that the sub-problems in RBMM are geodesically strongly convex so can be efficiently solved by using standard convex Riemannian optimization methods (see, e.g., \cite{udriste1994convex, zhang2016first, liu2017accelerated}). 

  Another benefit from the geodesic strong convexity of surrogates is related to the inexact computation, which is stated in the following proposition. Namely, the square of the geodesic distance between the inexact and the exact minimizer of the proximal surrogate \eqref{eq:def_prox_surrogate} is upper bounded by the optimality gap \eqref{eq:def_optimality_gap}, when the surrogate is geodesically strongly convex. Therefore, \eqref{eq:assumption_inexact_sol_convervence} in \ref{assumption:A0_optimal_gap}\textbf{(ii)} is automatically satisfied on Hadamard manifolds.
		
		\begin{prop}[Optimality gap for iterates]
			\label{prop:optimality_gap_ite}
			For each $n \geq 1$ and $i \in\{1, \ldots, m\}$, let $\theta_n^{(i \star)}$ be the exact minimizer of the geodesically $\left(\lambda_n-L_f\right)$-strongly convex function $\theta \mapsto g_n^{(i)}(\theta)$ in \eqref{eq:def_prox_surrogate} over the strongly convex set $\Theta^{(i)}$. Then
			$$
			\frac{\lambda_n-L_f}{2}d^{2}\left(\theta_n^{(i \star)},\theta_n^{(i)}\right) \leq \Delta_n .
			$$
		\end{prop}
		\begin{proof}
			Note the optimality condition of $\theta_n^{(i\star)}$ over $\Theta^{(i)}$ is 
			\begin{equation}
				\langle g_n^{(i)}(\theta_n^{(i\star)}),\Exp^{-1}_{\theta_n^{(i\star)}}(\theta)\rangle\geq 0, \quad \forall \theta \in \Theta^{(i)}
			\end{equation}
			then we have
			\begin{align}
				\frac{\lambda_n-L_f}{2}d^{2}\left(\theta_n^{(i \star)},\theta_n^{(i)}\right)&\leq g_n^{(i)}(\theta_n^{(i)})-g_n^{(i)}(\theta_n^{(i\star)})-\langle g_n^{(i)}(\theta_n^{(i\star)}),\Exp^{-1}_{\theta_n^{(i\star)}}(\theta_n^{(i)})\rangle\\
				&\leq g_n^{(i)}(\theta_n^{(i)})-g_n^{(i)}(\theta_n^{(i\star)}) \\
				&\leq \Delta_n .
			\end{align}
		\end{proof}
		
		Next, we give an illustration of $d^2(\cdot,p)$ is not $g$-smooth in general. The following proposition gives the expression of the Riemannian gradient of geodesic distance. 
		\begin{prop}[Riemannian gradient of geodesic distance]
			\label{prop:R_gradient_geodesic_distance}
			Let $\mathcal{M}$ be a complete Riemannian manifold, $p\in \mathcal{M}$ with $r_{\operatorname{inj}}(p)= r$. Consider the function $h: \mathcal{M}\rightarrow \mathbb{R}$ where $h(x)=d_{\mathcal{M}}^2(x, p)$, where $p \in \mathcal{M}$ is fixed. If $d(x,p)<r$, then $\grad(h)=-2 \Exp _x^{-1}(p)$ as a vector in $T_x \mathcal{M}$. 
		\end{prop}
		\begin{proof}
			See e.g. \cite[Proposition 4.8, p. 108]{sakai1996riemannian}.
		\end{proof}
		
		With Prop. \ref{prop:R_gradient_geodesic_distance}, we can now explain why the surrogate given in \ref{assumption:A1_smoothness_surrogates}\textbf{(i-rp)} is not $g$-smooth in general. Recall Def. \ref{def: G-L-smooth}, in order $\frac{1}{2}d^2(x,p)$ in Prop. \ref{prop:R_gradient_geodesic_distance} to be $g$-smooth, we need the following to hold,
		\begin{equation}
			\label{eqn:g_smooth_d2}
			\frac{1}{2}\left\lVert \grad d^2(y,p) - \Gamma_{x\rightarrow y}(\grad d^2(x,p)) \right\rVert = \left\lVert - \Exp _y^{-1}(p) + \Gamma_{x\rightarrow y}( \Exp _x^{-1}(p)) \right\rVert \le L d(x, y)
		\end{equation}
		In Euclidean space, parallel transport $\Gamma_{x\rightarrow y}$ is identical, and \eqref{eqn:g_smooth_d2} would degenerate to equality with $L=1$. Hence $\frac{1}{2}d^2(x,p)$ is $g$-smooth in Euclidean space, see Fig. \ref{fig:tri_hyp}(a) for an illustrative example. However, on a general Riemannian manifold, even on the Hadamard manifold, this is no longer true in general, see Fig. \ref{fig:tri_hyp}(b)
		for a counterexample of hyperbolic space (see Example \ref{eg:hyperbolic_space}). Some more examples of $S^1$ are shown in Fig. \ref{fig:circle}.

		\begin{figure*}[h]
			\centering
			\includegraphics[width=0.7 \linewidth]{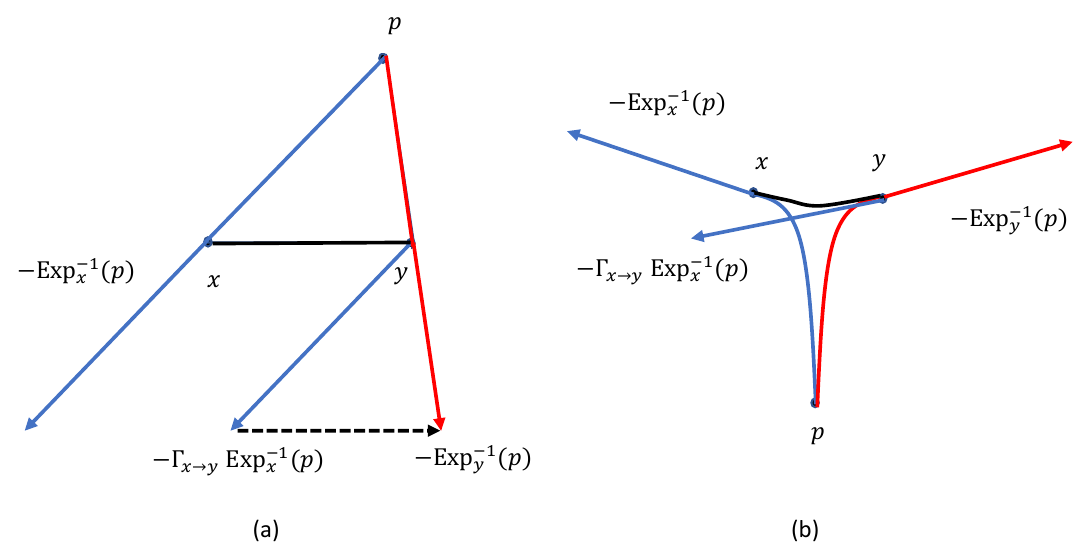}
			\vspace{-0.3cm}
			\caption{ Examples on $g$-smoothness of $d^2(x,p)$. Panel (a) is an example in Euclidean space, the length of the dashed black arrow represents the LHS of \eqref{eqn:g_smooth_d2} and the length of the black segment is $d(x,y)$; Panel (b) is a counterexample in hyperbolic space.
			}
			\label{fig:tri_hyp}
		\end{figure*}  
		
		\begin{figure*}[h]
			\centering
			\includegraphics[width=1 \linewidth]{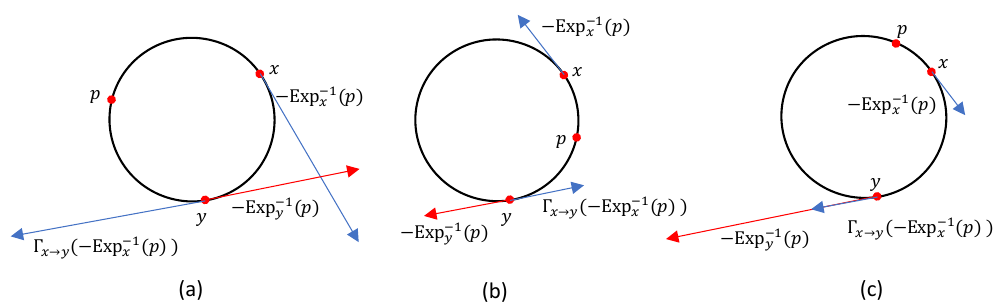}
			\vspace{-0.3cm}
			\caption{ Examples on $g$-smoothness of $d^2(x,p)$ on $S^1$. Panel (a) is the case when \eqref{eqn:g_smooth_d2} does not hold; Panel (b) (c) is the case when \eqref{eqn:g_smooth_d2} becomes an equality with $L=2$.
			}
			\label{fig:circle}
		\end{figure*}

		Now we come back to establish the proof of Theorems \ref{thm:RBMM_2} and \ref{thm:BMM_prox}. Let $\param_{\infty}=[\theta^{(1)}_{\infty},\cdots,\theta^{(m)}_{\infty}]$ be a limit point of $(\param_n)_{n}$, for simplicity of notations, replace $(\param_n)_{n}$ by the convergent subsequence. We introduce the following notations of tangent vectors and sets, which will be used throughout the remaining proofs. For $n\ge1$, let 
  \begin{align}
  \label{def:Theta_n}
      \Theta^{(i)}_n = \{\theta\in \Theta^{(i)} :d(\theta,\theta^{(i)}_n)\le r_0 /2 \}
  \end{align}
  where $r_0$ is the uniform lower bound of injectivity radius (see \ref{assumption:A1_smoothness_surrogates}\textbf{(ii)}). Similarly, denote 
  \begin{equation}\label{def:Theta_infty}
        \Theta^{(i)}_\infty =\{\theta\in \Theta^{(i)} :d(\theta,\theta^{(i)}_\infty)\le r_0 /2 \}.
  \end{equation}
  
  For any $\theta^{(i)}\in \Theta^{(i)}_\infty$, let  
    \begin{align}\label{eq:eta_n_nstar}
    \eta_{n}^{(i)}:=\eta_{\theta_{n}^{(i)}}(\theta^{(i)}),\quad   \eta_{n}^{(i\star)}:=\eta_{\theta_{n}^{(i\star)}}(\theta^{(i)}), \;\textup{and}\;\; \eta_{\infty}^{(i)}:=\eta_{\theta_{\infty}^{(i)}}(\theta^{(i)}).
    \end{align}
    Note that the above tangent vectors are well-defined for all $n$ sufficiently large under \ref{assumption:A0_optimal_gap} and \ref{assumption:A1_smoothness_surrogates}\textbf{(i)-(ii)}. In fact, since $\param_n$ converges to $\param_\infty$,
    \begin{align}\label{eq:iterate_conv}
        d(\theta^{(i)}_{n},\theta^{(i)}_{\infty}) < r_{0}/4 \quad \textup{for all $n\ge N_{1}$ for some $N_{1}\in \N$}.
     \end{align}
    Also, by \ref{assumption:A0_optimal_gap},
  \begin{align}\label{eq:iterate_conv}
        d(\theta^{(i)}_{n},\theta^{(i\star)}_{n}) < r_{0}/4 \quad \textup{for all $n\ge N_{2}$ for some $N_{2}\in \N$}.
     \end{align}
     Therefore by triangle inequality, the tangent vectors in \eqref{eq:eta_n_nstar} are well defined for $n\ge \max\{N_1,N_2\}$. The following proposition shows the asymptotic first-order optimality of inexact solutions, which is the key proposition to prove Theorem \ref{thm:RBMM_2}.

		\begin{prop}[Asymptotic first order optimality of inexact solutions] 
			\label{prop:asym_first_order_inexact}
			Assume  \ref{assumption:A1_smoothness_surrogates}\textbf{(i-rp), (ii)}, \ref{assumption:A0_optimal_gap}. Let $\param_{\infty}=[\theta^{(1)}_{\infty},\cdots,\theta^{(m)}_{\infty}]$ be a limit point of $(\param_n)_{n}$, for simplicity of notations, replace $(\param_n)_{n}$ by the convergent subsequence. Consider any $\theta^{(i)}\in \Param^{(i)}_\infty$ (see \eqref{def:Theta_infty}) such that $\eta_{\infty}^{(i)}$ defined in \eqref{eq:eta_n_nstar} satisfies $\|\eta_{\infty}^{(i)}\|\le 1$. Let $\eta_{n}^{(i\star)}$ and $\eta_{n}^{(i)}$ be tangent vectors defined in \eqref{eq:eta_n_nstar}.
			For each $i=1,\cdots,m$, the following holds, 
			\begin{equation}
				\left|\left\langle\grad g^{(i)}_{n}(\theta^{(i\star)}_{n}),\eta_{n}^{(i\star)} \right\rangle-\left\langle\grad g^{(i)}_{n}(\theta^{(i)}_{n}),\eta_{n}^{(i)} \right\rangle\right|=o(1).
			\end{equation}
		\end{prop}
		\begin{proof}
			First by triangle inequality and \ref{assumption:A0_optimal_gap},
			\begin{equation}
				\label{eqn:Hadamard_conv_theta_star}  d(\theta^{(i\star)}_{n},\theta^{(i)}_{\infty})\leq d(\theta^{(i)}_{n},\theta^{(i)}_{\infty})+ d(\theta^{(i\star)}_{n},\theta^{(i)}_{n})=o(1)
			\end{equation}
			so $\theta^{(i\star)}_{n}\to \theta^{(i)}_{\infty}$ as $n\to \infty$. 
			
			Next, note that Prop. \ref{prop:forward_monotonicity_proximal}, 
			\begin{equation}
				\label{eqn:distance_consecutive1}
				d(\theta^{(i)}_{n},\theta^{(i)}_{n-1})=o(1)
			\end{equation}
			this means when $n>>1$, $d(\theta^{(i)}_{n},\theta^{(i)}_{n-1})\leq r^{(i)}\leq r_{\operatorname{inj}}(\theta^{(i)}_n )$.
			
			Therefore, by \eqref{eqn:distance_consecutive1} and \ref{assumption:A0_optimal_gap},
			\begin{equation}
				\|\Exp^{-1}_{\theta^{(i\star)}_{n}}(\theta^{(i)}_{n-1})\|=d(\theta^{(i\star)}_{n},\theta^{(i)}_{n-1})\leq d(\theta^{(i\star)}_{n},\theta^{(i)}_{n})+ d(\theta^{(i)}_{n},\theta^{(i)}_{n-1})=o(1) . 
			\end{equation}
			Hence 
			\begin{equation}
				\label{eqn:Exp_inv_consecutive}
				\Exp^{-1}_{\theta^{(i\star)}_{n}}(\theta^{(i)}_{n-1})\to \mathbf{0}  .
			\end{equation} 
			Similarly, we have $\Exp^{-1}_{\theta^{(i)}_{n}}(\theta^{(i)}_{n-1})\to \mathbf{0}$. 
			
			Let $\Gamma_{\theta^{(i\star)}_{n}\rightarrow \theta^{(i)}_{\infty}}$ be the parallel transport along the minimal geodesic joining $\theta^{(i\star)}_{n}$ and $\theta^{(i)}_{\infty}$. 
   Since $\theta^{(i\star)}_{n}\to \theta^{(i)}_{\infty}$, by smoothness of the exponential map, we have $\Gamma_{\theta^{(i\star)}_{n}\rightarrow \theta^{(i)}_{\infty}}\eta_{n}^{(i\star)} \to \eta_{\infty}^{(i)}$ where $\Exp_{\theta_{\infty}^{(i)}}(\eta_{\infty}^{(i)})=\theta^{(i)}$ (see \cite[Remark 6.11]{azagra2005nonsmooth} for a concise conclusion). Also note $\|\eta_{\infty}^{(i)}\|\le 1$. Therefore by continuity of the Riemann metric together with \eqref{eqn:distance_consecutive1}, \eqref{eqn:Exp_inv_consecutive},  and the geodesical smoothness of $f$ as well as $\lambda_n=O(1)$, we have
			\begin{align}
				\left\langle\grad g^{(i)}_{n}(\theta^{(i\star)}_{n}),\eta_{n}^{(i\star)}\right\rangle =& \left\langle\grad f^{(i)}_{n}(\theta^{(i\star)}_{n}),\eta_{n}^{(i\star)}\right\rangle- \left\langle \lambda_{n}\Exp^{-1}_{\theta^{(i\star)}_{n}}(\theta^{(i)}_{n-1}),\eta_{n}^{(i\star)}\right\rangle \\
				\to& \left\langle \grad_{i}f(\param_{\infty}),  \eta_{\infty}^{(i)}\right\rangle.
			\end{align}
			Similarly, we have $\left\langle\grad g^{(i)}_{n}(\theta^{(i)}_{n}),\eta_{n}^{(i)}\right\rangle\to \left\langle \grad_{i}f(\param_{\infty}),  \eta_{\infty}^{(i)}\right\rangle$. The assertion follows.
		\end{proof}
		We are now ready to prove Theorem \ref{thm:RBMM_2}.
		\begin{proof}[\textbf{Proof of Theorem \ref{thm:RBMM_2} under \ref{assumption:A1_smoothness_surrogates}(i-rp)}. ]
			Assume \ref{assumption:A0_optimal_gap}, and \ref{assumption:A1_smoothness_surrogates}\textbf{ (i-rp), (ii)}. Fix a convergent subsequence $\left(\boldsymbol{\theta}_{n_k}\right)_{k \geq 1}$ of $\left(\boldsymbol{\theta}_n\right)_{n \geq 1}$. We wish to show that $\param_{\infty}=\lim _{k \rightarrow \infty} \param_{n_k}$ is a stationary point of $f$ over $\Param$. Fix any $\theta^{(i)}\in \Theta^{(i)}_\infty$ (see \eqref{def:Theta_infty})  with $d(\theta^{(i)},\theta^{(i)}_\infty)\le 1$, by first-order optimality of $\theta^{(i\star)}_{n_k}$,
			\begin{equation}
				\left\langle\grad g^{(i)}_{n_k}(\theta^{(i\star)}_{n_k}),\eta_{n_k}^{(i\star)}\right\rangle\geq 0 .
			\end{equation}
			Note that by Prop. \ref{prop:forward_monotonicity_proximal}, 
			\begin{equation}
				\label{eqn:distance_consecutive}
				d(\theta^{(i)}_{n},\theta^{(i)}_{n-1})=o(1),
			\end{equation}
			which means when $n\gg 1$, $d(\theta^{(i)}_{n},\theta^{(i)}_{n-1})\leq r^{(i)}_{\operatorname{inj}}\leq r_{\operatorname{inj}}(\theta^{(i)}_n )$, so the inverse exponential map is well defined.
			
			Then by Prop. \ref{prop:R_gradient_geodesic_distance} and Prop. \ref{prop:asym_first_order_inexact},
			\begin{equation}
				\lim_{k\to\infty}\left\langle \grad g^{(i)}_{n_k}(\theta^{(i)}_{n_k}),\eta_{n_k}^{(i)}\right\rangle = 
				\lim_{k\to\infty}\left\langle\grad f^{(i)}_{n_k}(\theta^{(i)}_{n_k})- \lambda_{n_k}\Exp^{-1}_{\theta^{(i)}_{n_k}}(\theta^{(i)}_{n_k-1}),\eta_{n_k}^{(i)}\right\rangle \geq 0.
			\end{equation}

			Note by Prop. \ref{prop:forward_monotonicity_proximal} and $\lambda_{n}=O(1)$, we get $\lambda_{n_k}d(\param_{n_k},\param_{n_k-1})=o(1)$ and therefore $\|\lambda_{n_k}\Exp^{-1}_{\theta^{(i)}_{n_k}}(\theta^{(i)}_{n_k-1})\|=\lambda_{n_k}d(\theta^{(i)}_{n_k},\theta^{(i)}_{n_k-1})=o(1)$. Also recall $\|\eta_{\infty}^{(i)}\|\le 1$, so $\|\eta_{n_k}^{(i)}\|$ is uniformly bounded. Thus
			\begin{equation}
				\lim_{k\to \infty}\left\langle \grad f^{(i)}_{n_k}(\theta^{(i)}_{n_k}),\eta_{n_k}^{(i)}\right\rangle \geq 0 . 
			\end{equation}
			Denote $\boldsymbol{\theta}_{\infty}=\left[\theta_{\infty}^{(1)}, \ldots, \theta_{\infty}^{(m)}\right]$. Let $\Gamma_{\theta^{(i)}_{n_k}\rightarrow \theta^{(i)}_{\infty}}$ be the parallel transport along the minimal geodesic joining $\theta^{(i)}_{n_k}$ and $\theta^{(i)}_{\infty}$. By smoothness of exponential map, we have $\Gamma_{\theta^{(i)}_{n_k}\rightarrow \theta^{(i)}_{\infty}}\eta^{(i)}_{n_k} \to \eta_{\infty}^{(i)}\in T_{\theta^{(i)}_{\infty}}$ satisfying $\Exp_{\theta_{\infty}^{(i)}}(\eta_{\infty}^{(i)})=\theta^{(i)}$. Then the above gives
			\begin{align}
				&\left\langle\grad_i f\left(\theta_{\infty}^{(1)}, \ldots, \theta_{\infty}^{(i-1)}, \theta_{\infty}^{(i)}, \theta_{\infty}^{(i+1)}, \ldots, \theta_{\infty}^{(m)}\right),\eta_{\infty}^{(i)}\right\rangle
				= \lim_{k\to \infty}\left\langle\grad f^{(i)}_{n_k}(\theta^{(i)}_{n_k}),\eta_{n_k}^{(i)}\right\rangle
				\geq 0.
			\end{align}
			This holds for all $i=1, \ldots, m$. Therefore we verified $\langle\grad f(\param_{\infty}),\eta\rangle\geq 0$ for all $\eta \in T^{*}_{\param_{\infty}}$ with $\|\eta\|\le 1$, which means that $\boldsymbol{\theta}_{\infty}$ is a stationary point of $f$ over $\boldsymbol{\Theta}$, as desired.
		\end{proof}

		\begin{proof}[\textbf{Proof of Theorem \ref{thm:BMM_prox} under \ref{assumption:A1_smoothness_surrogates}(i-rp)} ]
			Let $b_n$ be any square-summable positive sequence. 
  
			Then by Prop. \ref{prop:asymptotic_first-order_optimality} with $h_{n}^{(i)}(\theta)=\frac{\lambda_n}{2}d^2(\theta,\theta_{n-1}^{(i)})$, using Prop. \ref{prop:forward_monotonicity_proximal} and Prop. \ref{prop:finite_first-order_variation},
			\begin{align}
				&\sum_{n=1}^{\infty} b_{n+1}\left[\sum_{i=1}^{m}-\inf _{\eta\in T^{*}_{\param_n}, \|\eta\|\le 1}\left\langle \grad_i f(\param_{n}),\frac{\eta^{(i)}}{\min\{r_0,1\}}\right\rangle\right] \\
				&\qquad \leq C\left(f(\param_0)-\inf _{\boldsymbol{\theta} \in \boldsymbol{\Theta}} f(\boldsymbol{\theta})+\sum_{n=1}^{\infty} b_n^2 + \sum_{n=1}^{\infty} d^2(\param_n,\param_{n+1})+\sum_{n=1}^{\infty}\Delta_n(\param_0)\right)
			\end{align}
			for some constant $C>0$ independent of $\boldsymbol{\theta}_0$ and the right hand side is finite. Thus by taking $b_n=1 /(\sqrt{n} \log n)$, using Lemma \ref{lem:sum_sequence}, we deduce
			\begin{equation}
				\min _{1 \leq k \leq n}\left[\sum_{i=1}^{m}-\inf _{\eta\in T^{*}_{\param_k},\|\eta\|\le 1}\left\langle\grad_i f(\param_{k}),\frac{\eta^{(i)}}{\min\{r_0,1\}}\right\rangle\right] \leq \frac{M+c\sum_{n=1}^{\infty}\Delta_n(\param_0)}{\sqrt{n} / \log n}
			\end{equation}  
			for some constants $M, c>0$ independent of $\param_0$. This shows \textbf{(i)}.
			
			Lastly, we show \textbf{(ii)}. Assume $\sup _{\boldsymbol{\theta}_0 \in \boldsymbol{\Theta}}\sum_{n=1}^{\infty}\Delta_n(\param_0)<\infty$, then the above inequality implies for some constant $M^{\prime}>0$ independent of $\boldsymbol{\theta}_0$,
			\begin{equation}
				\min _{1 \leq k \leq n} \sup _{\boldsymbol{\theta}_0 \in \boldsymbol{\Theta}}\left[\sum_{i=1}^{m}-\inf _{\eta\in T^{*}_{\param_k},\|\eta\|\le 1}\left\langle\grad_i f(\param_{k}),\frac{\eta^{(i)}}{\min\{r_0,1\}}\right\rangle\right] \leq \frac{M^{\prime}\log n}{\sqrt{n}} .
			\end{equation}
			Then we could conclude \textbf{(ii)} by using the fact that $n \geq 2 \varepsilon^{-2}\left(\log \varepsilon^{-2}\right)^2$ implies $(\log n)^2 / n \leq \varepsilon^2$ for all sufficiently small $\varepsilon>0$. This completes the proof.

		\end{proof}
\subsection{RBMM with Euclidean proximal surrogates}

In this section, we prove Theorems \ref{thm:RBMM_2} and \ref{thm:BMM_prox} under \ref{assumption:A1_smoothness_surrogates}\textbf{(i-ep)}. The proof is very similar to that of Theorems \ref{thm:RBMM_2} and \ref{thm:BMM_prox} under \ref{assumption:A1_smoothness_surrogates}\textbf{(i-rp)}, as shown above. However, it is not a direct corollary from the aforementioned theorems, due to the difference of assumptions. In the following proof, we omit the repeated details for conciseness and only show the key parts.

\begin{proof}[\textbf{Proof of Theorems \ref{thm:RBMM_2} and \ref{thm:BMM_prox} under \ref{assumption:A1_smoothness_surrogates}\textbf{(i-ep)}}.]
We first show the parallel results from Prop. \ref{prop:forward_monotonicity_proximal} and \ref{prop:asymptotic_first-order_optimality}. Using same analysis as in \ref{prop:forward_monotonicity_proximal}, with Euclidean proximal surrogate as in \ref{assumption:A1_smoothness_surrogates}\textbf{(i-ep)}, we have
\begin{align}
				\sum_{n=1}^{\infty}\sum_{i=1}^{m} \|\theta_{n-1}^{(i)}- \theta_{n}^{(i)}\|^2_F &\le  \sum_{n=1}^{\infty} \sum_{i=1}^{m}\left( g^{(i)}_{n}(\theta^{(i)}_{n}) -  f^{(i)}_{n}(\theta^{(i)}_{n}) \right) <  f(\param_{0}) - f^{*} +m\sum_{n=1}^{\infty}\Delta_n<\infty.
			\end{align}
In particular, $\|\theta_{n-1}^{(i)}- \theta_{n}^{(i)}\|=o(1)$ for all $i=1,\dots,m$. To get the parallel results from Prop. \ref{prop:asymptotic_first-order_optimality}, we further use Lemma \ref{lem:equiv_dist} to bound the geodesic distance by Euclidean distance,
\begin{align}\label{eq:Eu_first_order_optimality_bd1}
		\sum_{i=1}^{m} \left\langle  \grad f_{n+1}^{(i)}(\theta_{n}^{(i)}), \eta^{(i)}_n (n+1) \right\rangle &\le \hat{r}^{-1}b_{n+1} \sum_{i=1}^{m}\inf _{\eta^{(i)}\in  T^{*}_{\theta_{n}^{(i)}} ,\|\eta^{(i)}\|\le 1}\left\langle  \grad f_{n+1}^{(i)}(\theta_{n}^{(i)}) ,\eta^{(i)}\right\rangle  \\
			&\qquad  +c_1 b_{n+1}^{2} + c_2 \sum_{i=1}^{m} \|\theta_{n-1}^{(i)}- \theta_{n}^{(i)}\|^2_F+m\Delta_n.
		\end{align}
To show the complexity results, the rest of the proof is identical to that of Theorem \ref{thm:BMM_prox} under \ref{assumption:A1_smoothness_surrogates}\textbf{(i-rp)} and we omit it here.

To show the asymptotic convergence, first note that 
\begin{equation}
    \grad g^{(i)}_{n}(\theta^{(i)}_{n}) = 
				\grad f^{(i)}_{n}(\theta^{(i)}_{n})- \lambda_{n}\operatorname{Proj}_{T_{\theta^{(i)}_{n-1}}}(\theta^{(i)}_{n}-\theta^{(i)}_{n-1}).
\end{equation}
Also note $\|\operatorname{Proj}_{T_{\theta^{(i)}_{n-1}}}(\theta^{(i)}_{n}-\theta^{(i)}_{n-1})\| \le \|\theta^{(i)}_{n}-\theta^{(i)}_{n-1}\|=o(1)$. Hence, by an identical argument as in the proof of Theorem \ref{thm:RBMM_2} under \ref{assumption:A1_smoothness_surrogates}\textbf{(i-rp)}, we conclude $(\param_n)_n$ asymptotically converges to the set of stationary points.
    
\end{proof}

		\subsection{RBMM with $g$-smooth surrogates}
		
		In this section, we prove Theorems \ref{thm:RBMM_1}, \ref{thm:RBMM_2}, and \ref{thm:BMM_rate} for smooth surrogates. Throughout this section, we assume \ref{assumption:A1_smoothness_surrogates}\textbf{(i-gs)}, that is, the surrogates are $g$-smooth.
		
		One of the most important tools in our analysis for the general smooth surrogate case is stated in Proposition \ref{prop:surrogate_gap_function_gap}. Recall that in Proposition \ref{prop:forward_monotonicity_proximal}, we have shown that the surrogate gaps $h_{n}^{(i)}(\theta_{n}^{(i)})$ are summable. In the next proposition, we show that this is enough to deduce that the norm of the Riemannian gradient of the surrogate gaps is also summable. This will be used later to deduce asymptotic optimality with respect to the objective function from that with respect to the surrogates. 
		
		\begin{prop}[Bound on surrogate optimality gap]\label{prop:surrogate_gap_function_gap}
			Assume \ref{assumption:A0_optimal_gap},  \ref{assumption:A1_smoothness_surrogates}\textbf{(i-gs), (ii)}. Then for any sequence of nonnegative real numbers $(w_{n})_{n\ge 1}$,
			\begin{align}
				\sum_{n=1}^{\infty} \sum_{i=1}^{m} w_{n} \lVert  \grad g_{n}^{(i)}(\theta_{n}^{(i)})-\grad f_{n}^{(i)}(\theta_{n}^{(i)})  \rVert^{2} \le 2(L_{f}+L_{g}) \sum_{n=1}^{\infty} \sum_{i=1}^{m} w_{n} (g_{n}^{(i)}(\theta_{n}^{(i)})-f_{n}^{(i)}(\theta_{n}^{(i)})).
			\end{align}
		\end{prop}
		
		\begin{proof}
			Denote $\alpha_{n}^{(i)} = \grad g_{n}^{(i)}(\theta_{n}^{(i)})-\grad f_{n}^{(i)}(\theta_{n}^{(i)})$. Fix a tangent vector $\eta \in T_{\theta_{n}^{(i)}} \mathcal{M}^{(i)}$ and $\eps>0$.  By Lemma \ref{lem:g_smooth_linear_approx}, we can write
			\begin{align}
				\left| g_{n}^{(i)}(\Exp_{\theta_{n}^{(i)}}(\epsilon \eta)) - g_{n}^{(i)}(\theta_{n}^{(i)}) - \langle  \nabla g_{n}^{(i)}(\theta_{n}^{(i)}), \eps\eta \rangle \right| &\le \frac{L_{g}\eps^{2}}{2} \lVert \eta \rVert^{2} ,\\
				\left|  f_{n}^{(i)}(\Exp_{\theta_{n}^{(i)}}(\epsilon \eta)) -  f_{n}^{(i)}(\theta_{n}^{(i)}) - \langle \nabla  f_{n}^{(i)}(\theta_{n}^{(i)}), \eps \eta \rangle \right| &\le  \frac{L_{f}\eps^{2}}{2} \lVert \eta \rVert^{2},
			\end{align}
			for constants $L_{f},L_{g}\ge 0$ in 
 \ref{assumption:A1_smoothness_surrogates}.  
			Hence  
			\begin{align}
				-	\frac{L_{f}\eps^{2}}{2} \lVert \eta \rVert^{2} +   f_{n}^{(i)}(\theta_{n}^{(i)}) + \langle \nabla  f_{n}^{(i)}(\theta_{n}^{(i)}),  \eps \eta \rangle &\le    f_{n}^{(i)}(\theta_{n}^{(i)}+\eps \eta)  \\
				&\le   g_{n}^{(i)}(\theta_{n}^{(i)}+\eps \eta)   \\
				&\le g_{n}^{(i)}(\theta_{n}^{(i)}) + \langle \nabla g_{n}^{(i)}(\theta_{n}^{(i)}, \eps\eta \rangle+ \frac{L_{g}\eps^{2}}{2} \lVert \eta \rVert^{2}.
			\end{align}
			Thus, denoting $c=(L_{f}+L_{g})/2$, we obtain the following inequality
			\begin{align}\label{eq:gradient_key_inequality}
				\langle  \alpha_{n}, \eps\eta \rangle  \ge  f_{n}^{(i)}(\theta_{n}^{(i)}) - g_{n}^{(i)}(\theta_{n}^{(i)}) - c \eps^{2} \lVert \eta\rVert^{2}.
			\end{align}
			Choosing $\eta=- \alpha_{n}$, we get 
			\begin{align}
				-\eps \lVert  \alpha_{n} \rVert^{2} \ge   f_{n}^{(i)}(\theta_{n}^{(i)}) - g_{n}^{(i)}(\theta_{n}^{(i)}) - c\eps^{2} \lVert \alpha_{n} \rVert^{2}.
			\end{align}
			Rearranging, we get 
			\begin{align}
				(\eps-c\eps^{2}) \lVert \alpha_{n} \rVert^{2} \le  g_{n}^{(i)}(\theta_{n}^{(i)})  -  f_{n}^{(i)}(\theta_{n}^{(i)}). 
			\end{align}
			Then the assertion follows by multiplying both sides of the above inequality by $w_{n}$ and summing up in $i=1,\dots,m$ and then for $n\ge 1$.
		\end{proof}
		
		\subsection{Asymptotic stationarity}
		In this section, we prove Theorems \ref{thm:RBMM_1} and \ref{thm:RBMM_2}. We start with Theorem \ref{thm:RBMM_1} where the number of blocks is $m\le 2$. Below we focus on the proof when $m=2$. In fact, the proof of the single block case ($m=1$) follows from a similar analysis and is much simpler.
		
  For conciseness, denote $w(k, i)=( \theta_{k}^{(1)},\dots,\theta_{k}^{(i-1)},\theta_{k}^{(i)},\theta_{k-1}^{(i+1)},\dots,\theta_{k-1}^{(m)} )$, which is the variable value in the $k$-th cycle after updating $i$-th block. Denote $w(k,i)^{(j)}$ as the j-th component of $w(k,i)$. The following observations would be helpful: $w(k,i)^{(i+1)}=\theta_{k-1}^{(i+1)}$, $w(k,i+1)^{(i+1)}=\theta_{k}^{(i+1)}$. The following proposition gives an observation related to the limit behavior of surrogate values.
		
		\begin{prop}\label{prop:BMM_limit_same}
			Assume \ref{assumption:A1_smoothness_surrogates}\textbf{(i-gs), (ii)} and \ref{assumption:A0_optimal_gap}. 
			Suppose that for some $i \in\{1, \ldots, m\}$ the sequence $\{w(k, i)\}$ admits a limit point $\bar{w}$and the subsequence converging to $\bar{w}$ is denoted by $w(k(n),i)$, $n\geq 1$. Then we have
			\begin{equation}
				\lim_{n\to \infty}g^{(i^+)}_{k(n)^+}\left(w(k(n),i)^{(i^+)}\right)=\lim_{n\to \infty}g^{(i^+)}_{k(n)^+}\left(w(k(n),i^+ )^{(i^+)}\right)=f(\bar{w})
			\end{equation}
			where $i^+ = i(\operatorname{mod} m) +1$, $k(n)^+ = k(n)$ or $k(n)+1$ for $i\in \{1,\cdots,m-1\}$ or $i=m$ respectively 
		\end{prop}
		\begin{proof}
			For conciseness, let $i\in \{1,\cdots,m-1\}$ so that $i^+ =i+1$. Note that the proof for $i=m$ would be exactly the same after modifying the notations. 
			
			By Prop. \ref{prop:forward_monotonicity_proximal} along with the tightness and upper-boundedness of surrogate functions, we have the following inequalities,
			\begin{align}
				f\left(w(k(n),i)\right) &= f^{(i+1)}_{k(n)}\left(\theta^{(i+1)}_{k(n)-1}\right) \\
				& =g^{(i+1)}_{k(n)}\left(\theta^{(i+1)}_{k(n)-1}\right) \\
				&=g^{(i+1)}_{k(n)}\left(w(k(n),i)^{(i+1)}\right) \\
				&\geq g^{(i+1)}_{k(n)}\left(\theta_{k(n)}^{(i\star)}\right) \\
				&\geq g^{(i+1)}_{k(n)}\left(w(k(n),i+1)^{(i+1)}\right) -\Delta_{k(n)}\\
				&\geq f^{(i+1)}_{k(n)}\left(w(k(n),i+1)^{(i+1)}\right) -\Delta_{k(n)}\\
				&=f\left(w(k(n),i+1)\right)  -\Delta_{k(n)}\\
				&\geq f\left(w(k(n+1),i)\right) -m\sum_{i=k(n)}^{k(n+1)}\Delta_{i}.      
			\end{align}
			Recall 
			\begin{equation}
				\lim_{n\to \infty}f\left(w(k(n),i)\right)=\lim_{n\to \infty}f\left(w(k(n+1),i)\right)=f(\bar{w}) \quad\text{and} \quad\lim_{n\to\infty}\sum_{i=n}^{\infty}\Delta_i =0. 
			\end{equation}
			This completes the proof.
		\end{proof}

		\begin{prop}[Asymptotic first order optimality of inexact solutions] 
			\label{prop:general_asym_first_order_inexact}
			Assume \ref{assumption:A1_smoothness_surrogates}\textbf{(i-gs), (ii)}, and  \ref{assumption:A0_optimal_gap}. Consider any $\theta^{(i)}\in \Param^{(i)}_\infty$ (see \eqref{def:Theta_infty}) such that $\eta_{\infty}^{(i)}$ defined in \eqref{eq:eta_n_nstar} satisfies $\|\eta_{\infty}^{(i)}\|\le 1$. Let $\eta_{n}^{(i\star)}\in T_{\theta^{(i\star)}_{n}}$ and $\eta_{n}^{(i)}\in T_{\theta^{(i)}_{n}}$ be the tangent vectors defined in \eqref{eq:eta_n_nstar}).
   For each $i=1,\cdots,m$, the following holds: 
			\begin{equation}
				\left|\left\langle\grad g^{(i)}_{n}(\theta^{(i\star)}_{n}),\eta_{n}^{(i\star)}\right\rangle-\left\langle\grad g^{(i)}_{n}(\theta^{(i)}_{n}),\eta_{n}^{(i)} \right\rangle\right|=o(1).
			\end{equation}
		\end{prop}
		\begin{proof}
			Let $\param_{\infty}=[\theta^{(1)}_{\infty},\cdots,\theta^{(m)}_{\infty}]$ be a limit point of $(\param_n)_{n}$. For simplicity of notations, replace $(\param_n)_{n}$ by the convergent subsequence.
			
			First by triangle inequality and \ref{assumption:A0_optimal_gap}, 
			\begin{equation}
				\label{eqn:conv_theta_star}  d(\theta^{(i\star)}_{n},\theta^{(i)}_{\infty})\leq d(\theta^{(i)}_{n},\theta^{(i)}_{\infty})+ d(\theta^{(i\star)}_{n},\theta^{(i)}_{n})=o(1),
			\end{equation}
			so $\theta^{(i\star)}_{n}\to \theta^{(i)}_{\infty}$ as $n\to \infty$. 
			
			Let $T=m\sum_{k=1}^{\infty}\Delta_k<\infty$, by Prop. \ref{prop:forward_monotonicity_proximal} we have 
			\begin{equation}
				\sup_{\param\in \Param}f(\param)\leq f(\param_{0})+T.
			\end{equation}
			Now let $K=\{\param\in \Param : f(\param)\leq f(\param_{0})+T\}$. By \ref{assumption:A0_optimal_gap}, $K$ is compact. Denote $K=K^{(1)}\times \cdots \times K^{(m)}$ where $K^{(i)}\subseteq \mathcal{M}^{(i)}$, then $K^{(i)}$ is compact for each $i=1,\cdots,m$. Since $g^{(i)}_n$ is geodesically smooth, $\|\grad g^{(i)}_{n}(\theta^{(i)})\|$ is uniformly bounded for $\theta^{(i)}\in K^{(i)}$ by some constant, say $L_k >0$.
			
			Let $\Gamma_{\theta^{(i\star)}_{n}\rightarrow \theta^{(i)}_{\infty}}$ be the parallel transport along a minimal geodesic joining $\theta^{(i\star)}_{n}$ and $\theta^{(i)}_{\infty}$. 
   By smoothness of the exponential map, we have we have $\Gamma_{\theta^{(i\star)}_{n}\rightarrow \theta^{(i)}_{\infty}}\eta^{(i\star)}_{n} \to \eta_{\infty}^{(i)}\in T_{\theta^{(i)}_{\infty}}$ where $\Exp_{\theta_{\infty}^{(i)}}(\eta_{\infty}^{(i)})=\theta^{(i)}$. 
			
			Then 
			\begin{align}
				& \left|\left\langle\grad g^{(i)}_{n}(\theta^{(i\star)}_{n}), \eta_{n}^{(i\star)} \right\rangle-\left\langle\grad g^{(i)}_{n}(\theta^{(i)}_{\infty}),\eta_{\infty}^{(i)} \right\rangle\right| \\
				&\quad =\left|\left\langle\Gamma_{\theta^{(i\star)}_{n}\rightarrow \theta^{(i)}_{\infty}}\grad g^{(i)}_{n}(\theta^{(i\star)}_{n}),\Gamma_{\theta^{(i\star)}_{n}\rightarrow \theta^{(i)}_{\infty}}\eta_{n}^{(i\star)} \right\rangle-\left\langle\grad g^{(i)}_{n}(\theta^{(i)}_{\infty}),\eta_{\infty}^{(i)} \right\rangle\right| \\
				&\quad \leq  \left|\left\langle\Gamma_{\theta^{(i\star)}_{n}\rightarrow \theta^{(i)}_{\infty}}\grad g^{(i)}_{n}(\theta^{(i\star)}_{n}),\Gamma_{\theta^{(i\star)}_{n}\rightarrow \theta^{(i)}_{\infty}}\eta_{n}^{(i\star)} \right\rangle-\left\langle\grad g^{(i)}_{n}(\theta^{(i)}_{\infty}),\Gamma_{\theta^{(i\star)}_{n}\rightarrow \theta^{(i)}_{\infty}}\eta_{n}^{(i\star)} \right\rangle\right| \\
				&\quad + \left|\left\langle\grad g^{(i)}_{n}(\theta^{(i)}_{\infty}),\Gamma_{\theta^{(i\star)}_{n}\rightarrow \theta^{(i)}_{\infty}}\eta_{n}^{(i\star)} \right\rangle-\left\langle\grad g^{(i)}_{n}(\theta^{(i)}_{\infty}),\eta_{\infty}^{(i)} \right\rangle\right| \\
				&\quad \leq  d(\theta^{(i\star)}_{n},\theta^{(i)}_{\infty}) + \|\grad g^{(i)}_{n}(\theta^{(i)}_{\infty})\|\cdot \left\|\Gamma_{\theta^{(i\star)}_{n}\rightarrow \theta^{(i)}_{\infty}}\eta_{n}^{(i\star)} - \eta_{\infty}^{(i)}  \right\| \\
				&\quad =o(1),
			\end{align}
			where the last inequality is by $g$-smoothness of $g^{(i)}_{n}$, the last equality is by \eqref{eqn:conv_theta_star} and the boundedness of $\|\grad g^{(i)}_{n}(\theta^{(i)}_{\infty})\|$.
			
			Similarly, we also have $\left|\left\langle\grad g^{(i)}_{n}(\theta^{(i)}_{n}),\eta_{n}^{(i)} \right\rangle-\left\langle\grad g^{(i)}_{n}(\theta^{(i)}_{\infty}),\eta_{\infty}^{(i)} \right\rangle\right|=o(1)$.
			Therefore we have
			\begin{align}
				&\left|\left\langle\grad g^{(i)}_{n}(\theta^{(i\star)}_{n}),\eta_{n}^{(i\star)} \right\rangle-\left\langle\grad g^{(i)}_{n}(\theta^{(i)}_{n}),\eta_{n}^{(i)} \right\rangle\right| \\
				\leq & \left|\left\langle\grad g^{(i)}_{n}(\theta^{(i\star)}_{n}),\eta_{n}^{(i\star)} \right\rangle-\left\langle\grad g^{(i)}_{n}(\theta^{(i)}_{\infty}),\eta_{\infty}^{(i)} \right\rangle\right| \\
				&+ \left|\left\langle\grad g^{(i)}_{n}(\theta^{(i)}_{n}),\eta_{n}^{(i)} \right\rangle-\left\langle\grad g^{(i)}_{n}(\theta^{(i)}_{\infty}),\eta_{\infty}^{(i)} \right\rangle\right| \\
				=&o(1).
			\end{align}

		\end{proof}

		Now we show the following proposition, which is the Riemannian version of \cite[Prop. 3]{grippo2000convergence}.
		\begin{prop}\label{prop:Grippo}
			Assume \ref{assumption:A1_smoothness_surrogates}\textbf{(i-gs), (ii)}, and  \ref{assumption:A0_optimal_gap}. Further assume that the constraint set $\Theta^{(i)}$ is strongly convex in $\M^{(i)}$ for $i=1,\dots,m$.
			Suppose that for some fixed $i \in\{1, \ldots, m\}$ the sequence $(w(k, i))_{k\ge 1}$ admits a limit point $\bar{w}$ and the subsequence converging to $\bar{w}$ is denoted by $w(k(n),i)$, $n\geq 1$. Denote $i^+ = i(\operatorname{mod} m) +1$, $k(n)^+ = k(n)$ or $k(n)+1$ for $i\in \{1,\cdots,m-1\}$ or $i=m$ respectively and $w(k(n),i)^{(i)}$ as the i-th coordinate of $w(k(n),i)$. For $n\gg 1$, consider any $\theta^{(i)} \in \Theta^{(i)}$ such that $d(\theta^{(i)},\bar{w}^{(i)})\le r_0/2$ and $\theta^{(i^+)} \in \Theta^{(i^+)}$ such that $d(\theta^{(i^+)},\bar{w}^{(i^+)})\le r_0/2$, denote $\eta^{(i)}_n \in T_{\theta^{(i)}_n}$ such that $\Exp_{\theta^{(i)}_n}(\eta^{(i)}_n)=\theta^{(i)}$ and $\eta^{(i^+)}_n \in T_{w(n,i)^{(i^+)}}$ such that $\Exp_{w(n,i)^{(i^+)}}(\eta^{(i^+)}_n)=\theta^{(i^+)}$. Then we have
			\begin{equation}\label{eqn:prop_fosc_i}
				\lim_{n\to\infty}\left\langle \grad g^{(i)}_{k(n)}\left(w(k(n),i)^{(i)}\right), \eta^{(i)}_{k(n)}\right\rangle\geq 0 \quad \text{and} \quad \liminf_{n\to \infty}\left\langle\grad g^{(i^+)}_{k(n)^+}\left(w(k(n),i)^{(i^+)}\right),\eta^{(i^+)}_{k(n)}\right\rangle\geq 0
			\end{equation}
		\end{prop}
		
		\begin{proof}
			Firstly, by the first-order optimality condition of $\theta^{(i\star)}_{k(n)}$ with respect to $g^{(i)}_{k(n)}$, for any $\theta^{(i)}\in \Theta^{(i)}$, we have 
			\begin{equation}
				\left\langle \grad g^{(i)}_{k(n)}\left(\theta^{(i\star)}_{k(n)}\right), \eta^{(i\star)}_{k(n)}\right\rangle\geq 0, \quad \text{for all} \;n
			\end{equation}
			where $\eta^{(i\star)}_n \in T_{\theta^{(i\star)}_n}$ such that $\Exp_{\theta^{(i\star)}_n}(\eta^{(i\star)}_n)=\theta^{(i)}$. Then by Prop. \ref{prop:general_asym_first_order_inexact}, we get the first identity of \eqref{eqn:prop_fosc_i}.
			
			In order to prove the second identity in  \eqref{eqn:prop_fosc_i}, suppose towards a contradiction that there exists $\theta^{i^+} \in \Theta^{(i^+)}$ such that
			\begin{equation}
				\liminf_{n\to \infty}\left\langle\grad g^{(i^+)}_{k(n)^+}\left(w(k(n),i)^{(i^+)}\right),\eta^{(i^+)}_{k(n)}\right\rangle< 0.
			\end{equation} 
			Then there exists an infinite index set  $K_{1}\subset K=\{k(n):n\ge 1\}$ such that for any $k\in K_1$,
			\begin{equation}
				\left\langle \grad g^{(i^+)}_{k(n)^+}\left(w(k(n),i)^{(i^+)}\right), \eta^{(i^+)}_{k(n)}\right\rangle < 0.
			\end{equation}
			For conciseness, from now on we only consider the case $i\in \{1,\cdots,m-1\}$, then $i^+ =i+1$. Note that when $i=m$, the following proof would be exactly the same after some modification of notation. 
			
			Let the search direction 
			\[
			d_{k}^{(i+1)} =
			\begin{cases}
				\eta^{(i+1)}_{k} & \text{if $\|\eta^{(i+1)}_{k}\|\leq 1$} \\
				\eta^{(i+1)}_{k}/\|\eta^{(i+1)}_{k}\| & \text{if $\|\eta^{(i+1)}_{k}\|>1$} .
			\end{cases}
			\]
			Now for all $k\in K_{1}$ suppose we compute the step size $\alpha^{(i+1)}_{k}$ by the line search algorithm \ref{algorithm:line_search}, we have 
			\begin{equation}
				g^{(i+1)}_{k}\left(\Exp_{w(k,i)^{(i+1)}}(\alpha^{(i+1)}_{k} d_{k}^{(i+1)})\right) = g^{(i+1)}_{k}\left(\Exp_{\theta^{(i+1)}_{k-1}}(\alpha^{(i+1)}_{k} d_{k}^{(i+1)})\right)\leq g^{(i+1)}_{k}\left(\theta^{(i+1)}_{k-1}\right)=g^{(i+1)}_{k}\left(w(k,i)^{(i+1)}\right).
			\end{equation}
			Note here we have $\Exp_{\theta^{(i+1)}_{k-1}}(\alpha^{(i+1)}_{k} d_{k}^{(i+1)})\in \Param^{(i+1)}$ by geodesic convexity of $\Param^{(i+1)}$.
			
			Recall by definition, the optimality of $\theta^{(i+1)}_{k}$ gives
			\begin{equation}
				g^{(i+1)}_{k}\left(w(k,i+1)^{(i+1)}\right)-\Delta_n\leq g^{(i+1)}_{k}\left(\theta^{(i+1\star)}_{k}\right)=\min_{\theta\in\Param^{(i+1)}}g^{(i+1)}_{k}(\theta).
			\end{equation}
			We can then write
			\begin{equation}
				g^{(i+1)}_{k}\left(w(k,i+1)^{(i+1)}\right)-\Delta_n\leq g^{(i+1)}_{k}\left(\Exp_{w(k,i)^{(i+1)}}(\alpha^{(i+1)}_{k} d_{k}^{(i+1)})\right) \leq g^{(i+1)}_{k}\left(w(k,i)^{(i+1)}\right).
			\end{equation}
			Then by noting that $k\in K_1\subset K=\{k(n):n\ge 1\}$ and use Proposition \ref{prop:BMM_limit_same} we get
			\begin{equation}
				\lim_{k\to \infty} g^{(i+1)}_{k}\left(w(k,i)^{(i+1)}\right)-g^{(i+1)}_{k}\left(\Exp_{w(k,i)^{(i+1)}}(\alpha^{(i+1)}_{k} d_{k}^{(i+1)})\right)=0.
			\end{equation}
			Finally by Proposition \ref{prop:line_search_limit} we have
			\begin{equation}
				\lim_{k\to\infty}\grad g^{(i+1)}_{k}\left(w(k,i)^{(i+1)}\right)=\mathbf{0},
			\end{equation}
			which gives a contradiction.
		\end{proof}

		The preceding proposition indicates that every limit point of the sequence generated by Algorithm \ref{algorithm:BMM} is a critical point with respect to the first and last component, i.e. $\theta^{(1)}$ and $\theta^{(m)}$, which is formally stated below.
		\begin{corollary}\label{corollary:BMM_conv_first_last}
			Under the same assumptions as in Prop. \ref{prop:Grippo}. Let $\param_{n}$ be a sequence generated by Algorithm \ref{algorithm:BMM} which admits a limit point $\param_{\infty}$, denote the subsequence converging to $\param_{\infty}$ by $\param_{n_k}$, then for any $\theta^{(m)}\in \Theta^{(m)}_\infty$ and $\theta^{(1)}\in \Theta^{(1)}_\infty$ 
			\begin{align}
				\lim_{k\to \infty}\langle\grad g^{(m)}_{n_k}\left(\theta_{n_k}^{(m)}\right),\eta^{(m)}_{n_k}\rangle\ge 0 \quad \text{and} \quad \liminf_{k\to \infty}\langle\grad g^{(1)}_{n_k+1}\left(\theta_{n_k}^{(1)}\right),\eta^{(1)}_{n_k}\rangle\ge 0.
			\end{align}
		\end{corollary}

		We are now ready to give a proof of Theorem \ref{thm:RBMM_1} for the case \ref{assumption:A1_smoothness_surrogates}\textbf{(i-gs)} of $g$-smooth surrogates on general manifolds. 
		\begin{proof}[\textbf{Proof of Theorem \ref{thm:RBMM_1} for $g$-smooth surrogates}]
			Assume  \ref{assumption:A1_smoothness_surrogates}\textbf{(i-gs), (ii)}, and  \ref{assumption:A0_optimal_gap}. Further assume that the constraint set $\Theta^{(i)}$ is strongly convex in $\M^{(i)}$ for $i=1,\dots,m$. Fix a convergent subsequence $(\param_{k_{n}})_{k\ge 1}$ of $(\param_{n})_{n\ge 1}$. We wish to show that  $\param_{\infty}=\lim_{k\rightarrow \infty} \param_{k_{n}}$ is a stationary point of $f$ over $\Param$ when $m=2$. The proof of single block case ($m=1$) follows from the same analysis and is much simpler, which can be done without Corollary \ref{corollary:BMM_conv_first_last}. Below we omit the $m=1$ case and focus on proving it for $m=2$. 
			
			First, we apply Proposition \ref{prop:surrogate_gap_function_gap} with $w_{n}\equiv 1$ and Proposition \ref{prop:forward_monotonicity_proximal} \textbf{(ii)} to deduce that 
			\begin{align}
				\sum_{n=1}^{\infty} \sum_{i=1}^{m} \lVert  \grad g_{n}^{(i)}(\theta_{n}^{(i)})-\grad f_{n}^{(i)}(\theta_{n}^{(i)})  \rVert^{2} <\infty.
			\end{align}
			
			\begin{align}\label{eq:o1_changes}
				\sum_{i=1}^{m} \lVert  \grad g_{n}^{(i)}(\theta_{n}^{(i)})-\grad f_{n}^{(i)}(\theta_{n}^{(i)})  \rVert^{2} = o(1).
			\end{align}
			Hence, by Corollary \ref{corollary:BMM_conv_first_last}, for any $\theta^{(1)}\in \Param^{(1)}_\infty$ and $\theta^{(m)}\in \Param^{(m)}_\infty$ we get 
			\begin{align}
				\liminf_{k\rightarrow \infty} \, \langle\grad f^{(1)}_{n_k+1}(\theta^{(1)}_{n_k}),\eta^{(1)}_{n_k}\rangle  \ge0 \qquad and \qquad \liminf_{k\rightarrow \infty} \, \langle\grad f^{(m)}_{n_k}(\theta^{(m)}_{n_k})  ,\eta^{(m)}_{n_k}\rangle  \ge0
			\end{align}
			where $\Exp_{\theta_{n_k}^{(m)}}(\eta^{(m)}_{n_k})=\theta^{(m)}$ and $\Exp_{\theta_{n_k}^{(1)}}(\eta^{(1)}_{n_k})=\theta^{(1)}$.
			
			Note that $\grad f^{(1)}_{n_k+1}(\param^{(1)}_{n_k})=\grad_{1}f(\param_{n_k})$ \text{and} $\grad f^{(m)}_{n_k}(\param^{(m)}_{n_k})=\grad_{m}f(\param_{n_k})$.   Thus for the case when $m=2$, by continuity of Riemannian metric and the continuity of $\grad f$, we verified for each $i=1,2$ and any $\theta\in \Theta^{(i)}_\infty$ we have $\langle\grad_i f (\param_{\infty} ),\eta\rangle \ge 0 $, where $\Exp_{\theta^{(i)}_{\infty}}(\eta)=\theta$. This shows $\param_{\infty}$ is a stationary point of $f$ over $\Param$, as desired.
		\end{proof}

		Next, we prove Theorem \ref{thm:RBMM_2} for smooth surrogates. 
		
		\vspace{0.1cm}
		\begin{proof}[\textbf{Proof of Theorem \ref{thm:RBMM_2} for $g$-smooth surrogates}]
			Suppose \ref{assumption:A1_smoothness_surrogates}\textbf{(i-gs), (ii), (iii)} and \ref{assumption:A0_optimal_gap} hold. Fix a convergent subsequence $(\param_{n_{k}})_{k\ge 1}$ of $(\param_{n})_{n\ge 1}$. We wish to show that  $\param_{\infty}=\lim_{k\rightarrow \infty} \param_{n_{k}}$ is a stationary point of $f$ over $\Param$. 
			First, we apply Proposition \ref{prop:surrogate_gap_function_gap} with $w_{n}\equiv 1$ and Proposition \ref{prop:forward_monotonicity_proximal} \textbf{(ii)} to deduce that 
			\begin{align}
				\sum_{n=1}^{\infty} \sum_{i=1}^{m}   \phi\left( d\left(\theta_{n-1}^{(i)},\theta_{n}^{(i)}\right) \right)<\infty ,\qquad    \sum_{n=1}^{\infty} \sum_{i=1}^{m} \lVert  \grad g_{n}^{(i)}(\theta_{n}^{(i)})-\grad f_{n}^{(i)}(\theta_{n}^{(i)})  \rVert^{2} <\infty.
			\end{align}
			In particular, this yields 
			\begin{align}\label{eq:o1_changes}
				\sum_{i=1}^{m}   \phi\left( d\left(\theta_{n-1}^{(i)},\theta_{n}^{(i)}\right) \right)=o(1),\qquad  \sum_{i=1}^{m} \lVert  \grad g_{n}^{(i)}(\theta_{n}^{(i)})-\grad f_{n}^{(i)}(\theta_{n}^{(i)})  \rVert^{2} = o(1).
			\end{align}
            Fix $\theta^{(i)}\in \Theta^{(i)}_{\infty}$ (see \eqref{def:Theta_infty}) such that $\eta_{\infty}^{(i)}$ defined in \eqref{eq:eta_n_nstar} satisfies $\|\eta_{\infty}^{(i)}\|\le 1$. Let $\eta_{n}^{(i\star)}\in T_{\theta^{(i\star)}_{n}}$ and $\eta_{n}^{(i)}\in T_{\theta^{(i)}_{n}}$ be the tangent vectors defined in \eqref{eq:eta_n_nstar}). Since $\theta_{n}^{(i\star)}$ minimizes $g_{n}^{(i)}$ over $\Theta^{(i)}$, we have 
			\begin{equation}
				\langle \grad g_{n}^{(i)} (\theta^{(i\star)}_{n}),\eta_{n}^{(i\star)}\rangle \geq 0.
			\end{equation}
			Note by Prop. \ref{prop:general_asym_first_order_inexact}, we have 
			\begin{equation}
				\lim_{k\to\infty}\langle \grad g_{n_k}^{(i)} (\theta^{(i)}_{n_k}),\eta^{(i)}_{n_k}\rangle \geq 0.
			\end{equation}

			Then recall the second part of \eqref{eq:o1_changes}, $\lVert  \grad g_{n}^{(i)} (\theta^{(i)}_{n})  -  \grad f_{n}^{(i)} (\theta^{(i)}_{n})   \rVert = o(1).$
			Note since $\theta^{(i)}_{n_k}\to \theta^{(i)}_{\infty}$ and $\|\eta^{(i)}_{\infty}\|\le 1$, we have $\|\eta^{(i)}_{n_k}\|$ is uniformly bounded by some constant $C_0>1$. Hence, for each $i=1,\dots,m$, we get
			\begin{equation}
				\left\|\left\langle \grad g_{n}^{(i)} (\theta^{(i)}_{n}),\eta^{(i)}_{n}\right\rangle - \left\langle \grad f_{n}^{(i)} (\theta^{(i)}_{n}),\eta^{(i)}_{n}\right\rangle\right\| \leq C_0 \|\grad g_{n}^{(i)} (\theta^{(i)}_{n})-\grad f_{n}^{(i)} (\theta^{(i)}_{n}) \|=o(1),
			\end{equation}
			so $\liminf_{k\to \infty}\left\langle \grad f_{n_k}^{(i)} (\theta^{(i)}_{n_k}),\eta^{(i)}_{n_k}\right\rangle \geq 0$.
			
			Let $\Gamma_{\theta^{(i)}_{n_k}\rightarrow \theta^{(i)}_{\infty}}$ be the parallel transport along a minimal geodesic joining $\theta^{(i)}_{n_k}$ and $\theta^{(i)}_{\infty}$. By smoothness of exponential map, we have $\Gamma_{\theta^{(i)}_{n_k}\rightarrow \theta^{(i)}_{\infty}}\eta^{(i)}_{n_k}\to \eta_\infty^{(i)}\in T_{\theta^{(i)}_{\infty}}$. By the first part of \eqref{eq:o1_changes} and the fact $\phi$ is strictly increasing, we have $d\left(\theta_{n-1}^{(i)},\theta_{n}^{(i)} \right)=o(1)$ and hence $\theta_{n_{k}-1}^{(j)}\rightarrow \theta_{\infty}^{(j)}$ as $k\rightarrow \infty$ for all $j=1,\dots,m$. Hence,
			\begin{align}
				0\leq \liminf_{k\to \infty}\left\langle \grad f_{n_k}^{(i)} (\theta^{(i)}_{n_k}),\eta^{(i)}_{n_k}\right\rangle & = \liminf_{k\to \infty}\left\langle \Gamma_{\theta^{(i)}_{n_k}\rightarrow \theta^{(i)}_{\infty}} \grad f_{n_k}^{(i)} (\theta^{(i)}_{n_k}), \Gamma_{\theta^{(i)}_{n_k}\rightarrow \theta^{(i)}_{\infty}}\eta^{(i)}_{n_k}\right\rangle \\
				&= \left\langle \grad_{i} f( \theta^{(1)}_{\infty},\dots,\theta^{(i-1)}_{\infty}, \theta^{(i)}_{\infty}, \theta^{(i+1)}_{\infty},\dots, \theta^{(m)}_{\infty} ),\eta_{\infty}^{(i)}\right\rangle. 
			\end{align}
			Since this holds for any $\theta^{(i)}\in \Theta^{(i)}_\infty$ with $d(\theta^{(i)},\theta_{\infty}^{(i)})\le 1$ and also holds for any $i=1,\cdots,m$, we conclude 
			\begin{equation}
				\left\langle \grad f (\param_{\infty} ), \mathbf{\eta} \right\rangle\geq 0, \quad \forall \eta\in T^{*}_{\param_{\infty}} \;\;\textup{with} \;\; \|\eta\|\le 1,
			\end{equation}
			which means $\param_{\infty}$ is a stationary point of $f$ in $\Param$, as desired.
		\end{proof}

		\subsection{Rate of convergence}
		
		In this section, we prove Theorem \ref{thm:BMM_rate}.  
		
		\begin{proof}[\textbf{Proof of Theorem \ref{thm:BMM_rate}}]
			We first show \textbf{(i)}. Let $b_n$ be any square-summable positive sequence, i.e. $\sum_{n} b_{n}^{2}<\infty$. 
			By Prop. \ref{prop:asymptotic_first-order_optimality} and Prop. \ref{prop:forward_monotonicity_proximal},
			\begin{align}
				&\sum_{n=0}^{\infty}b_{n+1}\left[\sum_{i=1}^{m}-\inf _{\eta\in T^{*}_{\param_n}}\left\langle \grad g_{n+1}^{(i)}(\theta^{(i)}_{n})+\grad_i f(\param_n)-\grad f_{n+1}^{(i)}(\theta^{(i)}_{n}),\frac{\eta^{(i)}}{\min\{r_0,1\}}\right\rangle\right] \\
				&\qquad \le C\left( \sum_{n=0}^{\infty} b_{n+1}^{2}+ \left(f(\param_{0})-f^{*}\right) + \sum_{k=1}^{\infty}d^{2}(\param_k,\param_{k+1} )+ \sum_{n=1}^{\infty}\Delta_n(\param_0)\right) <\infty.
			\end{align}
			for some constant $C>0$ independent of $\param_{0}$ and the right hand side is finite.

			Using Prop. \ref{prop:surrogate_gap_function_gap}, we get
			\begin{align}
				\sum_{n=0}^{\infty}b_{n+1}&\left[\left| -\sum_{i=1}^{m}\inf _{\eta\in T^{*}_{\param_n}}\left\langle \grad g_{n+1}^{(i)}(\theta^{(i)}_{n})+\grad_i f(\param_n)-\grad f_{n+1}^{(i)}(\theta^{(i)}_{n}),\frac{\eta^{(i)}}{\min\{r_0,1\}}\right\rangle \right| \right. \\
				&+\left. \sum_{i=1}^{m}\|\grad g_{n+1}^{(i)}(\theta^{(i)}_{n})- \grad f_{n+1}^{(i)}(\theta^{(i)}_{n})\|^{2} \right] < \infty.
			\end{align}
			Then by Lemma \ref{lem:sum_sequence} we have 
			\begin{align}
				\min_{1\leq k\leq n} &\left[ \left| -\sum_{i=1}^{m}\inf _{\eta\in T^{*}_{\param_k}}\left\langle \grad g_{k+1}^{(i)}(\theta^{(i)}_{k})+\grad_i f(\param_k)-\grad f_{k+1}^{(i)}(\theta^{(i)}_{k}),\frac{\eta^{(i)}}{\min\{r_0,1\}}\right\rangle \right|\right. \\ &\left.+\sum_{i=1}^{m}\|\grad g_{k+1}^{(i)}(\theta^{(i)}_{k})- \grad f_{k+1}^{(i)}(\theta^{(i)}_{k})\|^{2} \right]\leq 
				O\left(\left(\sum_{k=1}^n b_k\right)^{-1}\right).
			\end{align}
			Let $t_n \in\{1, \ldots, n\}$ for $n \geq 1$ be such that the minimum above is achieved. Namely, denoting the term in the minimum above by $A_{t_n}$, we have $A_{t_n}=O\left(\left(\sum_{k=1}^n b_k\right)^{-1}\right)$. Since all terms in $A_{t_n}$ are nonnegative, it follows that there exists a constant $c'_1, c'_2>0$ such that for all $n \geq 1$, almost surely,
			\begin{align}
				&\left| -\sum_{i=1}^{m}\inf _{\eta\in T^{*}_{\param_n}}\left\langle \grad g_{n+1}^{(i)}(\theta^{(i)}_{n})+\grad_i f(\param_n)-\grad f_{n+1}^{(i)}(\theta^{(i)}_{n}),\frac{\eta^{(i)}}{\min\{r_0,1\}}\right\rangle \right| \leq \frac{c'_1}{\sum_{k=1}^n b_k}\\
				&\text{and}\quad \|\grad g_{n+1}^{(i)}(\theta^{(i)}_{n})- \grad f_{n+1}^{(i)}(\theta^{(i)}_{n})\|\leq \frac{c'_{2}}{\left(\sum_{k=1}^n b_k\right)^{1/2}}\qquad \text{for any $i=1,\cdots,m$}.
			\end{align}

			On the other hand, by the Cauchy-Schwartz inequality, for all $\eta\in T^{*}_{\param_n}$,
			\begin{equation}
				\left| \sum_{i=1}^{m}\left\langle \grad g_{n+1}^{(i)}(\theta^{(i)}_{n})- \grad f_{n+1}^{(i)}(\theta^{(i)}_{n}),\frac{\eta^{(i)}}{\min\{r_0,1\}}\right\rangle \right| \leq \sum_{i=1}^{m}\|\grad g_{n+1}^{(i)}(\theta^{(i)}_{n})- \grad f_{n+1}^{(i)}(\theta^{(i)}_{n})\|.
			\end{equation}
			Hence by Lemma \ref{lem:sum_sequence}, it follows that there exists some constant $c'_3>0$ such that for all $n \geq 1$,
			\begin{equation}\label{eq:convergence_rate_pf0}
				\min_{1\leq k\leq n}\left[-\sum_{i=1}^{m}\inf_{\eta\in T^{*}_{\param_k},\|\eta\|\le 1}\left\langle \grad f(\param_{k}),\frac{\eta^{(i)}}{\min\{r_0,1\}}\right\rangle\right]\leq \frac{c'_3}{\left(\sum_{k=1}^n b_k\right)^{1/2}}.
			\end{equation}
			Thus by taking $b_{n}=n^{-1/2}(\log n)^{-1}<1$, we have $\sum_{n} b_{n}^{2}<\infty$, $\sum_{k=1}^{n} b_{k} = O(n^{1/2}/\log n)$. For some $M,c>0$ independent of $\param_{0}$ we get,
			\begin{equation}\label{eqn:proof_thm_bmm_ii}
				\min_{1\leq k\leq n}\left[-\sum_{i=1}^{m}\inf_{\eta\in T^{*}_{\param_k},\|\eta\|\le 1}\left\langle \grad f(\param_{k}),\frac{\eta^{(i)}}{\min\{r_0,1\}}\right\rangle\right]\leq \frac{M+c\sum_{n=1}^{\infty}\Delta_n(\param_0)}{n^{1/4}/(\log n)^{1/2}}.
			\end{equation}
			
			Next, it is easy to derive \textbf{(ii)} from \textbf{(i)}. According to \ref{assumption:A0_optimal_gap}, $\sup_{\param_{0} \in \Param}\sum_{n=1}^{\infty}\Delta_n(\param_0)<\infty$. Then the above implies for some $M^{'}$ independent of $\param_{0}$ we have
			\begin{equation}
				\min_{1\leq k\leq n}\sup_{\param_{0} \in \Param}\left[-\sum_{i=1}^{m}\inf_{\eta\in T^{*}_{\param_k},\|\eta\|\le 1}\left\langle \grad f(\param_{k}),\frac{\eta^{(i)}}{\min\{r_0,1\}}\right\rangle\right] \leq \frac{M^{'}(\log n)^{1/2}}{n^{1/4}}.
			\end{equation}
			Then we can conclude \textbf{(ii)} by using the fact that $n \geq 2 \varepsilon^{-4}\log \varepsilon^{-2}$ implies $(\log n)^{1/2} / n^{1/4} \leq \varepsilon$ for all sufficiently small $\epsilon>0$.
			
			Lastly, we show \textbf{(iii)}. For this, we assume \ref{assumption:A1_1}\textbf{(iii)} holds for $cx^{2} \le \phi(x) \le Cx^{2}$ for some constants $c,C>0$. Then by Prop. \ref{prop:asymptotic_first-order_optimality} and Prop. \ref{prop:forward_monotonicity_proximal},
			\begin{align}
				&\sum_{n=0}^{\infty}b_{n+1}\left[-\sum_{i=1}^{m}\inf _{\eta\in T^{*}_{\param_k},\|\eta\|\le 1}\left\langle \grad_i f(\param_k),\, \frac{\eta^{(i)}}{\min\{r_0,1\}}\right\rangle\right] \\
				&\qquad \le C'\left( \sum_{n=0}^{\infty} b_{n+1}^{2}+ \left(f(\param_{0})-f^{*}\right) + \sum_{k=1}^{\infty}d^{2}(\param_k,\param_{k+1} )+ \sum_{n=1}^{\infty}\Delta_n(\param_0)\right) <\infty
			\end{align}
			for some constant $C'>0$. Hence by Lemma \ref{lem:sum_sequence}, it follows that there exists some constant $c'>0$ such that for all $n \geq 1$, 
			\begin{equation}\label{eq:convergence_rate_pf1}
				\min_{1\leq k\leq n}\left[-\sum_{i=1}^{m}\inf_{\eta\in T^{*}_{\param_k},\|\eta\|\le 1}\left\langle \grad f(\param_{k}),\frac{\eta^{(i)}}{\min\{r_0,1\}}\right\rangle\right]\leq \frac{c'}{\sum_{k=1}^n b_k }.
			\end{equation}
			Notice that the bound in \eqref{eq:convergence_rate_pf1} has an improvement of the power of $2$ of the bound in \eqref{eq:convergence_rate_pf0}. Now the rest of the argument is identical to before. 
		\end{proof}

		\section*{Acknowledgements} 
		YL is partially supported by the Institute for Foundations of Data Science RA fund through NSF Award DMS-2023239 and by the National Science Foundation through grants DMS-2206296. HL is partially supported by the National Science Foundation through grants DMS-2206296 and DMS-2010035. DN is partially supported by NSF DMS-2011140 and NSF DMS-2108479. LB is partially supported by NSF CAREER award CCF-1845076 and ARO YIP award W911NF1910027.

		\vspace{0.3cm}
		\newcommand{\etalchar}[1]{$^{#1}$}
\providecommand{\bysame}{\leavevmode\hbox to3em{\hrulefill}\thinspace}
\providecommand{\MR}{\relax\ifhmode\unskip\space\fi MR }
\providecommand{\MRhref}[2]{%
  \href{http://www.ams.org/mathscinet-getitem?mr=#1}{#2}
}
\providecommand{\href}[2]{#2}

		\vspace{0.5cm}
		\addresseshere
		
		\newpage 
		
		\appendix

		\section{Background on Riemannian Optimization}\label{sec:notes}

		Let $\gamma:[a, b] \rightarrow \mathcal{M}$ be a piece-wise differentiable curve, then it assigns to each $t \in(a, b)$ a vector $\gamma^{\prime}(t)$ in the vector space $T_{\gamma(t)} \mathcal{M}$, the size of which can be measured by the norm $\|\cdot\|_{\gamma(t)}$. The length of $\gamma$ is given by $L(\gamma)=\int_a^b\left\|\gamma^{\prime}(t)\right\|_{\gamma(t)} dt$. The distance for any $x,y \in \mathcal{M}$ is then given by
		$$
		d_{\mathcal{M}}(x, y)=\inf \{L(\gamma): \gamma \text { a piecewise continuously differentiable curve from } x \text { to } y\}
		$$
		We drop the subscript $\mathcal{M}$ when it is clear from context. If $L(\gamma)=d_{\mathcal{M}}(x, y)$, then $\gamma$ is called a \textit{distance-minimizing geodesic} joining $x$ and $y$.

		For each $\theta,\theta'\in \mathcal{M}$, define $\eta=\eta_{\theta}(\theta')$ to be the 
		(Such tangent vector need not be unique unless $\theta'$ is within the injectivity radius at $\theta$)
		tangent vector in $T_{\theta}$ such that $\Exp_{\theta}(\eta)=\theta'$. 
  
  For a subset $\Param\subseteq \mathcal{M}$ and $x\in \Param$, define the \textit{tangent cone} $\mathcal{T}_{\Param}(x)$ and the \textit{normal cone} $\mathcal{N}_{\Param}(x)$ at $x$ as 
	\begin{align}\label{eq:def_tangent_normal_cone}
		&	\mathcal{T}_{\Param}\left(x \right):=\left\{u \in T_{x}\M \,\bigg|\,  \textup{$\Exp_{x}\left(t \frac{u}{\lVert u \rVert}\right)\in \Param$ for some $t\in (0, \rinj(x))$} \right\}\cup \{ \mathbf{0} \}, \\
		&	\mathcal{N}_{\Param}\left(x \right):=\left\{u \in T_{x}\M \,\bigg|\, \left\langle u, \eta\right\rangle \leq 0 \,\,  \textup{for all $\eta\in T_{x}\M$ s.t. $\Exp_{x}\left(t \frac{\eta}{\lVert \eta\rVert}\right)\in \Param$ for some $t\in (0, \rinj(x))$} \right\}.
	\end{align}
	Note that	$\mathcal{T}_{\Param}(x)=T_{x}\M$ and $\mathcal{N}_{\Param}(x)=\{ \mathbf{0} \}$ if $x$ is in the interior of $\Param$. When $\Param$ is strongly convex, then the tangent cone $\mathcal{T}_{\Param}(x)$ is a convex cone in the tangent space $T_{x}\M$ (see \cite[Prop.1.8]{cheeger1972structure} and \cite{afsari2011riemannian}.

		\section{Auxiliary lemmas}

		Recall that for each $\theta,\theta'\in \mathcal{M}$, we define $\eta_{\theta}(\theta')$ to be the set of all tangent vectors $\eta\in T_{\theta}\M $ such that $\Exp_{\theta}(\eta)=\theta'$. Define the (Riemannian) subdifferential of $f:\mathcal{M}\rightarrow \R$ by 
		\begin{equation}
			\partial f\left(\theta\right):=\left\{v \in T_{\theta}\M \mid f(\theta')-f\left(\theta\right) \geq\left\langle v,\eta_{\theta}(\theta')\right\rangle+o\left(d(\theta,\theta' )\right) \text { as } \theta'\rightarrow \theta\right\}.
		\end{equation}
		Recall that for a subset $\Param\subseteq \mathcal{M}$, its (Riemannian) \textit{normal cone} at $\theta\in \Theta$ is defined as (equivalent to \eqref{eq:def_tangent_normal_cone})
		\begin{align}
			\mathcal{N}_{\Param}\left(\theta \right)
			&= \left\{u \in T_{\theta} \mid\left\langle u, \eta \right\rangle \leq 0 \quad\forall \, \theta' \in \Theta\,\, \forall \eta\in \eta_{\theta}(\theta')\right\}.
		\end{align}

		\begin{lemma}[Bound on the linear approximation for $g$-smooth functions]
			\label{lem:g_smooth_linear_approx}
			Suppose the function $f: \mathcal{M}\to \R$ is geodesically $L$-smooth (see Definition \ref{def: G-L-smooth}) and $\mathcal{M}$ is a 
			Riemannian manifold. Suppose $x,y\in \M$  and there exists a distance-minimizing geodesic $\gamma:[0,1]\rightarrow \M$ from $x$ to $y$. Then \begin{equation}\label{eq:linear_approx_quad_bd}
				\left|f(y) - f(x)- \left\langle \grad f(x), \gamma'(0) \right\rangle_x  \right|\le \frac{L}{2} d^{2}(x, y),
			\end{equation}
			where $d(x,y)$ is the Riemannian distance between $x$ and $y$. Moreover, if $\Exp^{-1}_x (y)$ is well defined, the above can be rewritten as 
			\begin{equation}\label{eq:linear_approx_inv_quad_bd}
				\left|f(y) - f(x)- \left\langle \grad f(x), \Exp^{-1}_x (y)\right\rangle_x   \right|\le \frac{L}{2} d^2(x, y).
			\end{equation}
		\end{lemma}

		\begin{proof}

   Denote the minimal geodesic from $x$ to $y$ as $\gamma:[0,1]\rightarrow \mathcal{M}$. That is, $\gamma(0)=x$, $\gamma(1)=y$, and $\int_{0}^{1} \lVert \gamma'(s) \rVert \,ds=d(x,y)$. Since the geodesic has a constant speed, we have $\lVert \gamma'(s) \rVert \equiv d(x,y)$. Then by the fundamental theorem of calculus, 
			\begin{align}
				f(y) - f(x) = f(\gamma(1)) - f(\gamma(0)) = \int_{0}^{1} (f\circ \gamma)'(s) \,ds =  \int_{0}^{1} \left\langle \grad f \left( \gamma(s) \right),\, \gamma'(s) \right\rangle_{\gamma(s)} \,ds.
			\end{align}
			By Cauchy-Schwarz inequality and geodesic $L$-smoothness of $f$, 
			\begin{align}
				&\left| \int_{0}^{1} \left\langle \grad f \left( \gamma(s) \right),\, \gamma'(s) \right\rangle_{\gamma(s)} - \int_{0}^{1} \left\langle \grad f \left( \gamma(0) \right),\, \gamma'(0) \right\rangle_{\gamma(0)} \,ds  \right| \\
				&\qquad = \left| \int_{0}^{1} \left\langle \grad f \left( \gamma(s) \right),\, \gamma'(s) \right\rangle_{\gamma(s)} - \int_{0}^{1} \left\langle \Gamma_{\gamma(0)\rightarrow \gamma(s)}\grad f \left( \gamma(0) \right),\,\gamma'(s) \right\rangle_{\gamma(s)} \,ds  \right| \\
				&\qquad = \left| \int_{0}^{1} \left\langle \grad f \left( \gamma(s) \right) - \Gamma_{\gamma(0)\rightarrow \gamma(s)} \grad f \left( \gamma(0) \right),\, \gamma'(s)  \right\rangle_{\gamma(s)} \,ds  \right| \\
				&\qquad \le \int_{0}^{1} \left\lVert \grad f \left( \gamma(s) \right) - \Gamma_{\gamma(0)\rightarrow \gamma(s)} \grad f \left( \gamma(0) \right)\right\rVert\, \lVert \gamma'(s) \rVert \,ds \\
				&\qquad \le \int_{0}^{1} L d(\gamma(s), \gamma(0))\lVert \gamma'(s) \rVert  \,ds \\
					&\qquad \overset{(*)}{=} L d^{2}(\gamma(1), \gamma(0))\int_{0}^{1}  s  \,ds \\
					&\qquad = \frac{L}{2} d^{2}(y, x). 
			\end{align}
 
			Now the assertion follows by noting that 
			\begin{align}
				\int_{0}^{1} \left\langle \grad f \left( \gamma(0) \right),\, \gamma'(0) \right\rangle_{\gamma(0)} \,ds  = \left\langle \grad f \left( x \right),\, \gamma'(0) \right\rangle_{x} = \left\langle \grad f \left( x \right),\, \frac{\gamma'(0)}{\lVert \gamma'(0) \rVert}\right\rangle_{x} d(x,y). 
			\end{align}
		\end{proof}

		\begin{lemma}[Additivity of $g$-smooth functions]
			\label{lem:additivity_g_smooth}
			Suppose the function $f,g: \mathcal{M}\to \R$ are geodesically smooth function with positive constants $L_{f}$ and $L_{g}$, respectively (see definition \ref{def: G-L-smooth}). Then $f+g$ is geodesically smooth with constant $L_{f}+L_{g}$. 
		\end{lemma}
		
		\begin{proof}
			First note that by definition of $\grad f$ linearity of $D(\cdot)$, the operator of directional derivative, for any $\eta\in T_{x}\mathcal{M}$,
			\begin{align}
				\langle \grad f(x)+\grad g(x),\eta\rangle_{x} &= \langle \grad f(x),\eta\rangle_{x}+ \langle \grad g(x),\eta\rangle_{x} \\
				&=D(f)(x)[\eta]+D(g)(x)[\eta] \\
				&=D(f+g)(x)[\eta]
			\end{align}
			thus by definition, $\grad f(x)+\grad g(x)=\grad (f+g)(x)$.
			
			Also note that the parallel transport $\Gamma_{x\rightarrow y}:T_{x}\mathcal{M}\to T_{y}\mathcal{M}$ is a linear isomorphism, therefore we have
			\begin{align}
				&\| \grad(f+g)(x)-\Gamma_{x\rightarrow y}\grad(f+g)(y)\| \\
				\leq & \| \grad(f)(x)-\Gamma_{x\rightarrow y}\grad(f)(y)\| + \| \grad(g)(x)-\Gamma_{x\rightarrow y}\grad(g)(y)\| \\
				\leq & (L_g+L_f)d(x,y)
			\end{align}
		\end{proof}

		\begin{lemma}\label{lem:sum_sequence}
			Let $\left(a_n\right)_{n \geq 0}$ and $\left(b_n\right)_{n \geq 0}$ be sequences of nonnegative real numbers such that $\sum_{n=0}^{\infty} a_n b_n<\infty$. Then
			\begin{equation}
				\min _{1 \leq k \leq n} b_k \leq \frac{\sum_{k=0}^{\infty} a_k b_k}{\sum_{k=1}^n a_k}=O\left(\left(\sum_{k=1}^n a_k\right)^{-1}\right)
			\end{equation}
		\end{lemma}
		\begin{proof}
			The assertion follows by noting that
			\begin{equation}
				\left(\sum_{k=1}^n a_k\right) \min _{1 \leq k \leq n} b_k \leq \sum_{k=1}^n a_k b_k \leq \sum_{k=1}^{\infty} a_k b_k<\infty
			\end{equation}
		\end{proof}

		The following propositions are about the line search method on Riemannian manifolds and are parallel to the Euclidean versions in \cite{grippo2000convergence}. 
		
		Consider a sequence $\{x_{k}\}\in \Param$ with partition $x_{k}=(x_{k}^{(1)},\cdots, x_{k}^{(m)})$ and the searching directions $d_{k}^{(i)}\in T_{\Theta^{(i)}}$ satisfying the following assumptions:
		\begin{definition}[Gradient related searching directions]\label{assumption:gradient_related}
			Let $\{d_{k}^{(i)}\}\in T_{\Theta^{(i)}}^{*}$ be the sequence of searching directions such that they are gradient-related, i.e. 
			\begin{enumerate}
				\item there exists a number $M>0$ such that $\left\|d_k^{(i)}\right\| \leqslant M$ for all $k$;
				\item $\liminf_{k\rightarrow\infty}\, \langle \grad_{i}g_k^{(i)}(x_k), d_k^{(i)}\rangle<0$.
			\end{enumerate}
			then we call $\{d_{k}^{(i)}\}$ \textit{gradient related}.
		\end{definition}
		An Armijo-type line search method can be described as follows.
		\begin{algorithm}[H]
			\small
			\caption{Armijo-type line search algorithm for surrogates } 
			\label{algorithm:line_search}
			\begin{algorithmic}[1]
				\State \textbf{Input:} $\sigma \in(0,1), \beta\in(0,1)$; $d_{k}^{(i)}$ (search direction)
				\State Compute
				\quad \begin{equation}\label{eqn:line_search_stop_condition}
					\alpha_k =\min _{j\ge 0}\left\{\beta^j: g_k^{(i)}\left(\Exp_{x^{(i)}_k}(\beta^j d^{(i)}_k)\right)\leqslant g_k^{(i)}\left(x_k\right)+\sigma\beta^j\ \langle \grad g_k^{(i)}\left(x_k\right), d^{(i)}_k\rangle\right\}
				\end{equation}
				\State \textbf{output:}  $\alpha_k$ 
			\end{algorithmic}
		\end{algorithm}

		Next, we show some well-known results on the Riemannian line search algorithm. It is worthwhile to point out that, in what follows, the sequence $\{x_{k}\}$ is a given sequence that may not depend on the line search algorithm, in the sense that $x_{k+1}$ may not be generated by line search along $d_{k}$. Nevertheless, this has no substantial effect on the convergence proof, which can be deduced easily from known results (see e.g. \cite{absil2008optimization}).
		\begin{prop}\label{prop:line_search_limit}
			Let $\{x_{k}\}$ be a sequence in $\Theta^{(i)}$ and let $\{d_{k}\}\in T_{\Theta^{(i)}}^{*}$ be the sequence of searching directions satisfies Definition \ref{assumption:gradient_related}. Let $\alpha_k$ be computed by Algorithm \ref{algorithm:line_search}, then 
			\begin{enumerate}
				\item[(i)] There exists a finite integer $j\ge 0$ such that $\alpha_k=\left(\beta_i\right)^j$ satisfies the acceptability condition \eqref{eqn:line_search_stop_condition};
				\item[(ii)] If  $\{x_{k}\}$ converges to $\bar{x}$ and 
				\begin{equation}
					\lim_{k\to \infty}g_{k+1}^{(i)}\left(\Exp_{x_k}(\alpha_k d_k)\right) -g_{k+1}^{(i)}\left(x_k \right)=0
				\end{equation}
				then
				\begin{equation}
					\lim_{k\to \infty}\grad g_{k+1}^{(i)}(x_{k} )=\mathbf{0}.
				\end{equation}
			\end{enumerate}
		\end{prop}
		
		\begin{proof}
			(i) is obvious by $\langle \grad_i g_k^{(i)}\left(x_k\right), d^{(i)}_k\rangle<0 $ and smoothness of $g_k^{(i)}$. To prove (ii), suppose for a contradiction that $\lim_{k\to \infty}\grad g_{k+1}^{(i)}(x_{k} )\neq \mathbf{0}$.
			By the choice of search directions $d_{k}$, there exists $\delta>0$ such that 
			\begin{align}
				\left\langle\operatorname{grad} g_{k+1}^{(i)}\left(x_k\right), d_k\right\rangle_{x_k} < -\delta<0 \quad \textup{for all sufficiently large $k$}.
			\end{align}
			By the choice of $\alpha_{k}$, we have
			\begin{equation}
				g_{k+1}^{(i)}\left(x_k\right)-g_{k+1}^{(i)}\left(\Exp_{x_k}(\alpha_k d_k)\right) \geq- \sigma \alpha_k\left\langle\operatorname{grad} g_{k+1}^{(i)}\left(x_k\right), d_k\right\rangle_{x_k} > \sigma \alpha_{k} \delta> 0
			\end{equation}
			for all sufficiently large $k$. Since $g_{k+1}^{(i)}\left(x_k\right)-g_{k+1}^{(i)}\left(\Exp_{x_k}(\alpha_k d_k)\right)$ goes to zero,  we must have $\alpha_k \rightarrow 0$. Recall that $\alpha_k$ 's are determined from the Armijo rule, so $\alpha_k=\beta^{m_k}$ for some integer $\m_{k}\ge 0$ all $k\ge 1$. Since $\alpha_{k}=o(1)$, $m_{k}$ must diverge, so $m_{k}\ge 1$ for all $k\ge \bar{k}$ for some integer $\bar{k}\ge 1$. Then $\alpha_{k}/\beta = \beta^{m_{k}-1}\le 1$, and the step-size  $\frac{\alpha_k}{\beta}$ did not satisfy the Armijo condition. Hence 
			\begin{equation}
				g_{k+1}^{(i)}\left(x_k\right)-g_{k+1}^{(i)}\left(\Exp_{x_k}(\frac{\alpha_k}{\beta} d_k)\right) < - \sigma \frac{\alpha_k}{\beta}\left\langle\operatorname{grad} g_{k+1}^{(i)}\left(x_k\right), d_k\right\rangle_{x_k}, \quad k\geq \bar{k}
			\end{equation}
			Denoting 
			\begin{equation}
				\hat{g}_{x}=g_{k+1}^{(i)}\circ \Exp_{x} \quad \text{and} \quad \Tilde{\alpha}_k=\frac{\alpha_k}{\beta}
			\end{equation}
			the inequality above reads
			\begin{equation}
				\frac{\hat{g}_{x_k}(\mathbf{0})-\hat{g}_{x_k}(\Tilde{\alpha}_k d_k)}{\Tilde{\alpha}_k} < - \sigma \left\langle\operatorname{grad} g_{k+1}^{(i)}\left(x_k\right), d_k\right\rangle_{x_k} < \sigma\delta \quad \forall\,\, k\geq \bar{k}
			\end{equation}
			The mean value theorem ensures that there exists $t\in[0,\Tilde{\alpha}_k]$ such that
			\begin{equation}
				-\operatorname{D}\hat{g}_{x_k}(t d_k)[d_k]< - \sigma \left\langle\operatorname{grad} g_{k+1}^{(i)}\left(x_k\right), d_k\right\rangle_{x_k}, \quad k\geq \bar{k}
			\end{equation}
			Now since $\Tilde{\alpha}_k\to 0$ and recall that $\operatorname{D}\hat{g}_{x_k}(0)[d_k]=\left\langle\operatorname{grad} g_{k+1}^{(i)}\left(x_k\right), d_k\right\rangle_{x_k}$, we obtain
			\begin{equation}
				-\liminf_{k\to \infty}\left\langle\operatorname{grad} g_{k+1}^{(i)}\left(x_k\right), d_k\right\rangle_{x_k} \leq -\sigma\liminf_{k\to \infty} \left\langle\operatorname{grad} g_{k+1}^{(i)}\left(x_k\right), d_k\right\rangle_{x_k}
			\end{equation}
			Since $\sigma<1$, it follows that $\liminf_{k\to \infty}\left\langle\operatorname{grad} g_{k+1}^{(i)}\left(x_k\right), d_k\right\rangle_{x_k}\geq 0$, which is a contradiction. 
		\end{proof}
		
		\begin{prop}[Properties of inverse exponential map on Hadamard manifold.]
			\label{prop:prop_inverse_exp_Hadamard}
			Let $\mathcal{M}$ be a Hadamard manifold , $(x_n)_{n\geq 1}\subset \mathcal{M}$ and $x_0 \in \mathcal{M}$,
			\begin{itemize}
				\item[(i)] For any $y \in \mathcal{M}$, we have
				$$
				\Exp _{x_n}^{-1} (y) \longrightarrow \Exp _{x_0}^{-1} (y) \quad \text { and } \quad \Exp _y^{-1} (x_n) \longrightarrow \Exp _y^{-1} (x_0)
				$$
				\item[(ii)] If $v_n \in T_{x_n} \mathcal{M}$ and $v_n \rightarrow v_0$, then $v_0 \in T_{x_0}\mathcal{M}$
				\item[(iii)] Given $u_n, v_n \in T_{x_n} \mathcal{M}$ and $u_0, v_0 \in T_{x_0} \mathcal{M}$, if $u_n \rightarrow u_0$ and $v_n \rightarrow v_0$, then
				$$
				\left\langle u_n, v_n\right\rangle \longrightarrow\left\langle u_0, v_0\right\rangle .
				$$
			\end{itemize}
		\end{prop}
		\begin{proof}
			See e.g. \cite[Lemma 2.4]{li2009monotone}.
		\end{proof}

\section{Details of Section \ref{sec:apps}}
\subsection{Riemannian Hessian on fixed-rank manifold}\label{sec:Hess_fixed_rank}
Here we provide details about why the Euclidean distance function is not g -smooth over low-rank
manifolds, as discussed in Section \ref{sec:apps}.  
Let $\mathcal{R}_{r}\subseteq \R^{m\times n}$ be the manifold of rank-$r$ matrices as in \eqref{eq:def_fixed_rank_manifold}. Let $X\in \mathcal{R}_r$, and without loss of generality, let $X=U\Sigma V^T$ where $U\in \mathcal{V}^{m\times r}$,  $V\in \mathcal{V}^{n\times r}$ and $\Sigma = \diag(\sigma_1,\cdots,\sigma_r)$ with $\sigma_1\ge \cdots \ge \sigma_r>0$. Following from \cite[Section 4.3]{absil2013extrinsic}, the Riemannian Hessian of $f$ at $X$ for $Z\in T_X \mathcal{R}_r$ is given by
\begin{align}\label{eq:Hessian_fixed_rank}
    \operatorname{Hess}f(X)[Z] &= \mathcal{P}_X \nabla^2 f(X) Z+\mathcal{P}_X \mathrm{D}_Z \mathcal{P}_X \nabla f(X) \\
    &=(\nabla^2 f(X) \mathrm{P}_V+\mathrm{P}_U \nabla^2 f(X)-\mathrm{P}_U \nabla^2 f(X) \mathrm{P}_V) Z + \nabla f(X) Z^T (X^+)^T + (X^+)^T Z^T \nabla f(X),
\end{align}
where $\nabla f$ and $\nabla^2 f$ are the Euclidean gradient and Euclidean Hessian of $f$, $\mathrm{P}_U = UU^T$, $\mathrm{P}_V = VV^T$, $X^+ = V\Sigma^{-1}U^T$. $\mathcal{P}_X$ is the projection operator onto the tangent space at $X$. $D_Z$ is the directional derivative following the tangent vector $Z$. A detailed discussion of these operators can be found in \cite{absil2013extrinsic}. 

Note for a simple Euclidean distance squared function $f(X)=\|X-X_0\|_F^2$, we have $\nabla f = 2(X-X_0)$ and $\nabla^2 f = 2 \mathbf{I}$. In order to show that $f$ is $g$-smooth on the fixed-rank manifold, one needs to verify 
\begin{align}
\sup_{X\in \mathcal{R}_{r}}\,  \max_{Z\in T_X}\, (\langle\operatorname{Hess}f(X)[Z],Z\rangle) \le C
\end{align}
using \eqref{eq:Hessian_fixed_rank}, where $C\in (0,\infty)$ is a constant (see \cite[Lemma C.6]{nguyen2019calculating}). Below we give a counterexample to show such a constant $C$ does not exist.

Take $Z=UV^T$ and $X_0 = -UV^T$. Note the inner product of the first term in \eqref{eq:Hessian_fixed_rank} with $Z$ is bounded. Hence we only need to show the inner product of the last two terms in \eqref{eq:Hessian_fixed_rank} with $Z$ can be unbounded. In fact, the last two terms are the same by the cyclic property of trace, 
\begin{align}
    \langle (X^+)^T Z^T \nabla f(X), Z\rangle =\operatorname{Tr}((X^+)^T Z^T \nabla f(X)Z^T) = \operatorname{Tr}(\nabla f(X)Z^T(X^+)^T Z^T ) = \langle \nabla f(X) Z^T (X^+)^T, Z\rangle.
\end{align}

Therefore we only compute the second term,
\begin{align}\label{eq:gradient_Hessian}
    \langle \nabla f(X) Z^T (X^+)^T, Z\rangle &= 2\langle U\Sigma V^T V U^T U\Sigma^{-1}V^T + UV^TVU^T U\Sigma^{-1}V^T, Z\rangle \\
    &=2\operatorname{Tr}((UV^T)^T UV^T) + 2\operatorname{Tr}((U\Sigma^{-1}V^T)^T UV^T) \\
    &= 2\|UV^T\|^2 + 2\operatorname{Tr}(V\Sigma^{-1}V^T)\\
    &= 2\|UV^T\|^2 + 2\operatorname{Tr}(\Sigma^{-1}). 
\end{align}
Now we take a sequence of $X^{(k)}\in \mathcal{R}_r$ such that the smallest singular value of $X^{(k)}$ goes to zero, i.e. $\sigma_r^{(k)} \to 0$ as $k\to \infty$. Therefore we have $(\sigma_r^{(k)})^{-1} \to \infty$ as $k\to \infty$. Hence \eqref{eq:gradient_Hessian} is unbounded. We conclude that $f(X)=\|X-X_0\|_F^2$ is not $g$-smooth on the fixed-rank manifold.

\subsection{Details of Section \ref{app:GST}}
 \label{appendix:GST}
		We give the details of the MM update for $\Theta$ \cite{blocker2023dynamic}. For fixed $H$ and $Y$, we first simplify \eqref{eqn:obj_subspace_QTheta},
		\begin{align}
			\label{eqn:subspace_separable}
			f(Q,\Theta)&=-\sum_{i=1}^T\left\|X_i^{T}\left(H \cos \left(\Theta t_i\right)+Y \sin \left(\Theta t_i\right)\right)\right\|_{\mathrm{F}}^2 \\
			&= -\sum_{i=1}^T \sum_{j=1}^k r_{i, j} \cos \left(2 \theta_j t_i-\phi_{i, j}\right)+b_{i, j}
		\end{align}
		where $\theta_j$ is the $j$-th diagonal element of $\Theta$. Defining $\arctan 2(y, x)$ as the angle of the point $(x, y)$ in the $2 \mathrm{D}$ plane counter-clockwise from the positive $x$-axis, the associated constants $r_{i, j}, \phi_{i, j}, b_{i, j}$ in the above equation are defined as
		\begin{align}
			\phi_{i, j} & =\arctan 2\left(\beta_{i, j}, \frac{\alpha_{i, j}-\gamma_{i, j}}{2}\right) & \alpha_{i, j} & =\left[H^{T} X_i X_i^{T} H\right]_{j, j} \\
			r_{i, j} & =\sqrt{\left(\frac{\alpha_{i, j}-\gamma_{i, j}}{2}\right)^2+\beta_{i, j}^2} & \beta_{i, j} & =\operatorname{real}\left\{\left[Y^{T} X_i X_i^{T} H\right]_{j, j}\right\} . \\
			b_{i, j} & =\frac{\alpha_{i, j}+\gamma_{i, j}}{2} & \gamma_{i, j} & =\left[Y^{T} X_i X_i^{T} Y\right]_{j, j}
		\end{align}
		Note that \eqref{eqn:subspace_separable} is separable for each diagonal element $\theta_j$ of $\Theta$, so we could find the minimizer separately by a univariate minimization. Let 
		\begin{equation}
			f^{(2)}_{n+1, j}\left(\theta_j\right) :=-\sum_{i=1}^{T}(r_{i, j})_{n+1} \cos \left(2 \theta_j t_i-(\phi_{i, j})_{n+1}\right)+(b_{i, j})_{n+1}
		\end{equation}
		which is the marginal objective function at iteration $n+1$ for the j-th diagonal component of $\Theta$. The subscript $n+1$ outside the parenthesis denotes values of the parameters at iteration $n+1$. The gradient, which actually becomes scalar, is given by 
		\begin{equation}
			\nabla f^{(2)}_{n+1, j}\left(\theta_j\right)= \dot{f}^{(2)}_{n+1, j}\left(\theta_j\right)=2 \sum_{i=1}^{T}(r_{i, j})_{n+1} t_i \sin \left(2 \theta_j t_i-(\phi_{i, j})_{n+1}\right),
		\end{equation}
		which is Lipschitz continuous with parameter $L_{n+1,j}= 4 \sum_{i=1}^{T}(r_{i, j})_{n+1} t_i^{2}$. We consider the following prox-linear majorizer of $f_{n+1,j}^{(2)}$: For $\lambda > L_{n+1,j}$, 
		\begin{equation}
			g_{n+1,j}^{(2)}(\theta_j)= f^{(2)}_{n+1, j}\left((\theta_j)_n\right) + \nabla f^{(2)}_{n+1, j}\left((\theta_j)_n\right) \left(\theta_j - (\theta_j)_n\right) + \frac{\lambda}{2}\left(\theta_j - (\theta_j)_n\right)^2
		\end{equation}
		where $(\theta_j)_n$ is the value of $\theta_j$ at iteration $n$. 
		
		Then by using \eqref{eq:block_prox_linear},
		\begin{align}\label{eq:subspace_g22}
			(\theta_j)_{n+1} &= \argmin g_{n+1,j}^{(2)}(\theta_j) \\
			&=\textup{Proj}_{\R}\left((\theta_j)_{n}-\frac{1}{\lambda}\nabla f^{(2)}_{n+1, j}\left((\theta_j)_n\right)\right) \\
			&=(\theta_j)_{n}-\frac{1}{\lambda}\nabla f^{(2)}_{n+1, j}\left((\theta_j)_n\right).
		\end{align}
		
\subsection{Details of Section \ref{sec:app_fisher_rao}}
\label{appendix:fisher_rao}
The following propositions give the closed form of Fisher-Rao (FR) distance under certain circumstances,
		\begin{prop}[FR distance for Gaussian distributions with identical mean\cite{atkinson1981rao}]
			If $\mathcal{N}\left(\hat{\mu}, \Sigma_0\right)$ and $\mathcal{N}\left(\hat{\mu}, \Sigma_1\right)$ are Gaussian distributions with identical mean $\hat{\mu} \in \mathbb{R}^n$ and covariance matrices $\Sigma_0, \Sigma_1 \in \mathbb{S}_{++}^n$, the set of $n\times n$ positive definite matrices, we have
			\begin{equation}\label{eqn:distance_identical_mu}
				d\left(\Sigma_0, \Sigma_1\right)=\frac{1}{\sqrt{2}}\left\|\log \left(\Sigma_1^{-\frac{1}{2}} \Sigma_0 \Sigma_1^{-\frac{1}{2}}\right)\right\|_F  
			\end{equation}
			where $\log (\cdot)$ represents the matrix logarithm, and $\|\cdot\|_F$ stands for the Frobenius norm.
		\end{prop}
		
		\begin{remark}
			It is worthwhile to point out the FR metric on the tangent space $T_{\Sigma} \mathbb{S}_{++}^n$ at $\Sigma \in \mathbb{S}_{++}^n$ can be re-expressed as (see \cite[p. 382]{atkinson1981rao})
			\begin{equation}\label{eqn:FR_metric_re_expressed}
				\left\langle\Omega_1, \Omega_2\right\rangle_{\Sigma} \triangleq \frac{1}{2} \operatorname{Tr}\left(\Omega_1 \Sigma^{-1} \Omega_2 \Sigma^{-1}\right) \quad \forall \Omega_1, \Omega_2 \in T_{\Sigma} \mathbb{S}_{++}^n.
			\end{equation}
		\end{remark}
		\begin{prop}[FR distance for Gaussian distributions with identical covariance\cite{atkinson1981rao}]
			If $\mathcal{N}\left(\mu_0, \hat{\Sigma}\right)$ and $\mathcal{N}\left(\mu_1, \hat{\Sigma}\right)$ are Gaussian distributions with identical covariance matrix $\hat{\Sigma} \in \mathbb{S}_{++}^n$ and mean vectors $\mu_0, \mu_1 \in \mathbb{R}^n$, we have
			$$
			\bar{d}\left(\mu_0, \mu_1\right)=\sqrt{\left(\mu_0-\mu_1\right)^{T} \hat{\Sigma}^{-1}\left(\mu_0-\mu_1\right)}
			$$
		\end{prop}

  Next, we show the optimization problem in \eqref{eqn:loss_likelihood} satisfies our assumptions for RBMM. We first cite the following proposition from \cite{nguyen2019calculating}, which states the geodesical convexity of the constraint sets.
		\begin{lemma}[Convexity of constraint sets]
			$\Theta^{(1)}$ and $\Theta^{(2)}$ are strongly convex.
		\end{lemma}
The following lemma shows the geodesical smoothness of $f_{n}^{(1)}$ and $f_{n}^{(2)}$, aiming to show that $f(\mu,\Sigma)$ satisfies \ref{assumption:A0_optimal_gap}\textbf{(i)}.
		\begin{lemma}[Geodesic smoothness of marginal objective function]\label{lem:FR_marginal_loss_smooth}
			For problem \eqref{eqn:loss_likelihood}, we have
			\begin{enumerate}
				\item $f^{(1)}_n$ is (geodesically) $L^{(1)}_n$-smooth with $L^{(1)}_n=1/\lambda_{\min}(\Sigma_{n})$.
				\item $f^{(2)}_n$ is geodesically $L^{(2)}_n$-smooth with $L^{(2)}_n=2 \lambda_{\max }(S_n)\lambda_{\min }(\widehat{\Sigma})^{-1} \exp (\sqrt{2} \rho_2)$.
			\end{enumerate}
		\end{lemma}
		We also need a uniform upper bound for $L^{(1)}_{n}$ and $L^{(2)}_{n}$. The following proposition gives lower and upper bounds for the eigenvalues of $\Sigma\in \Theta^{(2)}$.
		\begin{prop}[Property of Fisher-Rao ball \cite{nguyen2019calculating}]\label{prop:bound_eigenvalue_Sigma}
			The FR ball $\Theta^{(2)}$ has the following property. For any $\Sigma \in \Theta^{(2)}$, we have $\lambda_{\min }(\hat{\Sigma}) e^{-\sqrt{2} \rho_{2}} \cdot I_n \preceq \Sigma \preceq \lambda_{\max }(\hat{\Sigma}) e^{\sqrt{2} \rho_{2}} \cdot I_n$.
		\end{prop}
		We further set a fixed $\rho_{1}$ given in the following proposition in order to upper bound $\lambda_{\max}(S_n)$.
		\begin{prop}\label{prop:bound_eigenvalue_Sn}
			Let the radius $\rho_{1}$ in $\Theta^{(1)}$ to be
			$$\rho_{1}=\max_{m=1,\cdots,M}\|x_m - \hat{\mu}\|_{2}$$
			then $\lambda_{\max}(S_{n})\leq 4\rho_{1}^{2}$.
		\end{prop}
		\begin{proof}
			Recall that $\Theta^{(1)}=\left\{\mu \in \mathbb{R}^{n}:(\mu-\hat{\mu})^{T} (\mu-\hat{\mu}) \leq \rho_{1}^2\right\}$ and $\hat{\mu}=\frac{1}{M}\sum_{m=1}^{M}x_m$. 
			
			Now with $\rho_{1}=\max_{m}\|x_m - \hat{\mu}\|_{2}$ we have 
			\begin{align}
				\operatorname{tr}\left(\left(x_m-\mu_n \right)\left(x_m-\mu_n\right)^{T}\right)=\operatorname{tr}\left(\left(x_m-\mu_n \right)^{T}\left(x_m-\mu_n\right)\right)\leq 4\rho_{1}^{2}, \qquad \text{for} \quad 1\leq m\leq M
			\end{align}
			therefore $\lambda_{\max}(S_{n})\leq \operatorname{tr}\left(\max_{m}\left(x_m-\mu_n \right)^{T}\left(x_m-\mu_n\right)\right)\leq 4\rho_{1}^2$.
		\end{proof}
		Lemma \ref{lem:FR_marginal_loss_smooth} together with Prop. \ref{prop:bound_eigenvalue_Sigma} and Prop. \ref{prop:bound_eigenvalue_Sn} gives 
		\begin{equation}
			L_{1}^{(n)}\leq 1/\lambda_{\min }(\hat{\Sigma}) e^{-\sqrt{2} \rho_{2}} \quad \text{and}\quad  L_{2}^{(n)}\leq 8\rho_{1}^{2}\lambda_{\min }(\widehat{\Sigma})^{-1} \exp (\sqrt{2} \rho_2) \quad \text{for all $n$.}
		\end{equation}
		which shows $f(\mu,\Sigma)$ satisfies \ref{assumption:A0_optimal_gap}\textbf{(i)}.

\subsection{Details of Section \ref{sec:RPCA}}
\label{appendix:RPCA}  
The closed form solution of \eqref{eq:BMM_RPCA} is as follows,
\begin{align}
    L_{k+1}&=\frac{1}{\mu \lambda_{k+1}}U\Sigma_r V,\quad  \textup{where}\;\; U\Sigma V \;\;\textup{is the SVD of }\;\; M-S_k+\lambda_{k+1}L_k, \\
    (S_{k+1})_{i j}&=B_{i j}-\tilde{\mu} \min \left\{\rho, \max \left\{-\lambda, \frac{B_{i j}}{\sigma+\tilde{\mu}}\right\}\right\} \text { for } i=1, \ldots, m \text { and } j=1, \ldots, n	,
\end{align}
where $\tilde{\mu}=\frac{\mu}{\mu\lambda_{k+1}+1}$ and $B=\frac{1}{\mu\lambda_{k+1}+1}M-L_{k+1} + \frac{\mu\lambda_{k+1}}{\mu\lambda_{k+1}+1}S_k$.
	\end{document}